\documentclass[12pt,oneside]{amsart}
\usepackage{latexsym,amssymb,amsmath,amsfonts,amsthm,tikz, float,enumitem, theoremref, dsfont, mathrsfs,pstricks,stmaryrd,mathtools,pgfplots}
\usepackage{comment}
\usepackage{xcolor}
\usepackage[linktocpage]{hyperref}
\usepackage{bbm}
\usepackage[margin=3cm]{geometry}
\mathtoolsset{showonlyrefs}
\pgfplotsset{compat=1.17}
\usetikzlibrary{decorations.pathreplacing,calligraphy}
\DeclareFontFamily{OML}{rsfs}{\skewchar\font'177}
\DeclareFontShape{OML}{rsfs}{m}{n}{ <5> <6> rsfs5 <7> <8> <9>
	rsfs7 <10> <10.95> <12> <14.4> <17.28> <20.74> <24.88> rsfs10 }{}
\DeclareMathAlphabet{\mathfs}{OML}{rsfs}{m}{n}

\definecolor{greytext}{gray}{0.5}
\definecolor{darkAlice}{RGB}{16,120,149}
\definecolor{darkIndigo}{RGB}{8,48,69}
\definecolor{myIndigo}{RGB}{11,55,77}
\definecolor{darkKelly}{RGB}{130,146,84}
\definecolor{darkRuby}{RGB}{154,37,21}
\definecolor{darkCoral}{RGB}{194,86,26}
\definecolor{midnightblue}{rgb}{0.1, 0.1, 0.44}
\definecolor{lapislazuli}{rgb}{0.15, 0.38, 0.61}
\definecolor{vividauburn}{rgb}{0.58, 0.15, 0.14}
\definecolor{olive}{rgb}{0.5, 0.5, 0.0}
\definecolor{skobeloff}{rgb}{0.0, 0.48, 0.45}
\definecolor{warmblack}{rgb}{0.0, 0.26, 0.26}
\definecolor{upforestgreen}{rgb}{0.0, 0.27, 0.13}
\definecolor{linkcolour}{RGB}{0,51,153}

\definecolor{burgundy}{RGB}{128, 0, 33}
\definecolor{complementaryblue}{RGB}{0,66,128}
\definecolor{complementarygreen}{RGB}{32,128,0}

\definecolor{lightgreen}{RGB}{143,252,135}
\definecolor{lightred}{RGB}{252,135,143}
\definecolor{lightmagenta}{RGB}{244,135,252}
\definecolor{lightblue}{RGB}{77,169,255}

\definecolor{linkblue}{RGB}{0,20,128}
\definecolor{linkred}{RGB}{166,20,0}

\hypersetup{
  linkcolor  = burgundy,
  citecolor  = burgundy,
  urlcolor   = linkblue,
  colorlinks = true,
}

\newtheorem{theorem}{Theorem} 

\newtheorem{thm}{Theorem}
\numberwithin{thm}{section}

\newtheorem{prop}[thm]{Proposition}

\newtheorem{lemma}[thm]{Lemma}
\newtheorem{cor}[thm]{Corollary}

\theoremstyle{definition}
\newtheorem{definition}[thm]{Definition}
\newtheorem{algorithm}{Algorithm}

\theoremstyle{remark}
\newtheorem{remark}[thm]{Remark}

\newcounter{expcnt}
\newcommand{\newexpo}{%
\refstepcounter{expcnt}%
\ensuremath{\xi_{\theexpcnt}}}
\newcommand{\expo}[1]{\ensuremath{\xi_{\ref*{#1}}}}

\newcounter{betcnt}
\newcommand{\newbet}{%
\refstepcounter{betcnt}%
\ensuremath{\beta_{\thebetcnt}}}
\newcommand{\bet}[1]{\ensuremath{\beta_{\ref*{#1}}}}

\newcounter{etacnt}
\newcommand{\newet}{%
\refstepcounter{etacnt}%
\ensuremath{\eta_{\theetacnt}}}
\newcommand{\et}[1]{\ensuremath{\eta_{\ref*{#1}}}}

\newcounter{cnt}
\newcommand{\newcnt}{%
\refstepcounter{cnt}%
\ensuremath{c_{\thecnt}}}
\newcommand{\cnt}[1]{\ensuremath{c_{\ref*{#1}}}}

\newcounter{Cnt}
\newcommand{\newCnt}{%
\refstepcounter{Cnt}%
\ensuremath{C_{\theCnt}}}
\newcommand{\Cnt}[1]{\ensuremath{C_{\ref*{#1}}}}

\newcounter{uBound}
\newcommand{\newuBound}{%
\refstepcounter{uBound}%
\ensuremath{u_{\theuBound}}}
\newcommand{\uBound}[1]{\ensuremath{u_{\ref*{#1}}}}

\DeclareMathOperator{\Geo}{Geom}
\DeclareMathOperator{\Poi}{Poi}
\DeclareMathOperator{\len}{len}
\DeclareMathOperator{\capa}{cap}
\DeclareMathOperator{\range}{\mathcal{R}}
\DeclareMathOperator{\trace}{edges}
\DeclareMathOperator{\supp}{supp}

\DeclareMathOperator{\ST}{ST}

\DeclareMathOperator{\sPath}{start}
\DeclareMathOperator{\ePath}{end}

\newcommand{\ind}[1]{\mathds{1}_{ #1 }}
\newcommand{\st}{ \: : \: }

\newcommand{\const}[2]{\hyperref[#1]{\ensuremath{#2}}}

\DeclareMathOperator{\cI}{\mathcal{I}}

\DeclareMathOperator{\RI}{RI}

\newcommand{\prob}{\mathbb{P}}
\newcommand{\prw}{\mathbf{P}}

\title[]{The chemical distance in random interlacements in the low-intensity regime} 

\author[Hern\'{a}ndez-Torres]{Sara\'{i} Hern\'{a}ndez-Torres}
\author[Procaccia]{Eviatar B. Procaccia}
\author[Rosenthal]{Ron Rosenthal}

\address{Instituto de Matem\'{a}ticas, Universidad Nacional Aut\'{o}noma de M\'{e}xico}
\email{saraiht@im.unam.mx}
\address{Technion -- Israel Institute of Technology}
\email{eviatarp@technion.ac.il }
\address{Technion -- Israel Institute of Technology}
\email{ron.ro@technion.ac.il }

\thanks{
SHT is supported by ISF grant 1692/17. EBP is supported by  BSF grant 2018330 and NSF grant DMS-1812009. RR is supported by BSF grant 2018330 and ISF grant 771/17. }

\begin{document}

\begin{abstract} 
	In $\mathbb{Z}^d$ with $d\ge 5$,
	we consider the time constant $\rho_u$ associated to the chemical distance in random interlacements at low intensity $u \ll 1$. We prove an upper bound of order $u^{-1/2}$ and a lower bound of order $u^{-1/2+\varepsilon}$. The upper bound agrees with the conjectured scale in which $u^{1/2}\rho_u$ converges to a constant multiple of the Euclidean norm, as $u\to 0$. Along the proof, we obtain a local lower bound on the chemical distance between the boundaries of two concentric boxes, which might be of independent interest. For both upper and lower bounds, the paper employs probabilistic bounds holding as $u\to 0$; these bounds can be relevant in future studies of the low-intensity geometry.
\end{abstract}
\maketitle
\tableofcontents

\section{Introduction}

Random interlacements on $\mathbb{Z}^d$, with $d \geq 3$, is a Poissonian soup of doubly infinite random walk paths. It is often used in the study of the local structure of a simple random walk covering a fixed proportion of the volume of a torus $(\mathbb{Z}/N\mathbb{Z})^d$, see e.g. \cite{procaccia2014range,Sznitman,BenjaminiSznitman2008}.
The model is defined as a Poisson point process of labelled doubly infinite trajectories, where each of the trajectories has the law of a doubly infinite simple random walk. 
A parameter $u > 0$ is the intensity of the point process, and as the intensity increases, more trajectories appear in the Poissonian soup, which we denote by $\RI^u$ (see Section \ref{sec:Background} for a precise definition). The range and edges in $\RI^u$ define the random interlacements graph $\cI^u$. The graph $\cI^u$ corresponds to a bond percolation of $\mathbb{Z}^d$ with long-range correlations. An edge $e = (x,y)$  of $\mathbb{Z}^d$ is open in $\cI^u$ if any of the trajectories in $\RI^u$ traverses $e$. In this case, we write $x \sim_u y$ and say that $x$ and $y$ are nearest neighbours in $\cI^u$.

The focus of this work is on the geometry of the interlacements graph $\mathcal{I}^u$.  Previous works  \cite{vcerny2012internal, DrewitzRathSapozhnikov, ProcacciaRosenthalSapozhnikov, ProcacciaTykesson, RathSapozhnikov, RathSapozhnikovQuenched,sapozhnikov2017random}  have shown that the metric properties of $\mathcal{I}^u$ are comparable to the geometry of $\mathbb{Z}^d$. It is conjectured that as $u \to 0$, the geometry of $~\mathcal{I}^u$ converges to the Euclidean geometry. We are interested in understanding this convergence and thus investigate the graph distance as  $u \to 0$.

In~\cite{Sznitman}, Sznitman proved that, for any intensity value $u > 0$, the graph $\cI^u$ is an infinite connected subset of $\mathbb{Z}^d$ almost surely. 
Procaccia and Tykesson, and independently R{\'a}th and Sapozhnikov, proved connectivity along the traversed edges of the interlacements in~\cite{ProcacciaTykesson,RathSapozhnikov2012}, showing that any two $x , y \in \cI^u$ are almost surely connected by at most $ \lceil d/2 \rceil $ trajectories in $\RI^u$. These results imply that $\cI^u$ is an infinite connected random subgraph of $\mathbb{Z}^d$.

The chemical distance in $\cI^u$ is defined as the intrinsic graph distance. For $u > 0$, we define the \textbf{chemical distance} $d_{u}: \cI^{u} \times \cI^{u} \to \mathbb{R}_+ $ in $\cI^{u}$ by 
\[
	d_{u} (x,y) \coloneqq \inf \Big\{ n \geq 0 \: : \: 
						\begin{array}{c} \text{ There exists } x_0, \ldots, x_n \in \mathcal{I}^{u} \text{ such that }  \\
							x_0 = x, x_n = y, \text{ and }	  x_{k-1} \sim_u x_{k} , \text{ for } k = 1, \ldots , n \end{array} 
								 \Big\}. 
\]

For a vertex $x \in \mathbb{Z}^d$, we define $[x]_u$ as the $\ell^1$-closest vertex in $\cI^u$ to $x$. In the case of a tie, we choose $[x]_u$ according to some fixed but arbitrary order. The only requirement for such order is invariance under symmetries of $\mathbb{Z}^d$, e.g.~\cite[Section IX A]{DrewitzRathSapozhnikov}.   
This definition extends the chemical distance to a function $d_u : \mathbb{Z}^d  \times \mathbb{Z}^d \to \mathbb{R}_+$ by setting $ d_u (x,y) \coloneqq d_u \left( [x]_u, [y]_u \right) $.

For each $x \in \mathbb{Z}^d $, the \textbf{time constant} $\rho_{u}$ in the direction of $x$ is the limit
\begin{equation}
	\rho_{u} (x) \coloneqq \lim_{n \to \infty} \frac{d_{u} ([0]_u,  [n \cdot x]_u) }{n}, \qquad \text{ a.s. and } L^1. 
\end{equation}

The convergence to the time constant is a classical consequence of the subadditive ergodic theorem (see~\cite{DrewitzRathSapozhnikov} for a proof in a general framework of certain long-range percolation models). Furthermore, the time constant is almost surely finite as a consequence of \v{C}ern{\'y} and Popov's shape theorem for the chemical distance in $\cI^u$, in~\cite{vcerny2012internal}. Remark~\ref{rmk: other_defs} indicates the need to be careful in the use of previous results on the random interlacements graph. In this case, results in~\cite{RathSapozhnikov} and arguments in~\cite[Appendix A]{vcerny2012internal} prove the necessary large deviation bound \cite[Theorem 1.3]{vcerny2012internal} for the chemical distance in $\cI^u$. 
The function $\rho_u$ has a natural homogeneous extension to the domain $\mathbb{Q}^d$, and then a unique continuous extension to $\mathbb{R}^d$. 
This extension defines a deterministic norm  $\rho_{u}$ on  $\mathbb{R}^d$.

In dimensions $d \geq 5$, we study the time constant of $\cI^u$ for low-intensity parameters $ u  < 1 $.
Our main results are upper and lower bounds for the time constant with respect to the Euclidean norm $\vert \cdot \vert_2$ on $\mathbb{R}^d$.

\begin{theorem} \label{thm:main_upper}
	Let $d \geq 5$.
	There exists a constant $\newCnt \label{c:up} = \Cnt{c:up}(d) \in (0, \infty)  $ such that for any $u\in (0,1)$
	\begin{equation} \label{eq:main_upper}
		\rho_{u} (x) \leq  \Cnt{c:up}\frac{1}{\sqrt{u}} \vert x \vert_2, \qquad \text{for every } x \in \mathbb{R}^d.
	\end{equation}
\end{theorem}

\begin{theorem} \label{thm:main_lower}
	Let $d \geq 5$. There exists a constant  $  \newCnt = \Cnt{c:low} (d) \in(0, \infty) \label{c:low} $ such that the following holds.
	For every $\varepsilon > 0$, there exists $\newuBound  \label{u:Thm} = \uBound{u:Thm} (\varepsilon)\in (0,1)$  so that for any $  u\in (0, \uBound{u:Thm})$
	\begin{equation} \label{eq:main_lower}
		\Cnt{c:low} \frac{u^{\varepsilon}}{\sqrt{u}} \vert x \vert_2  \leq \rho_{u} (x)   , \qquad \text{for every } x \in \mathbb{R}^d.
	\end{equation}
\end{theorem}

\begin{remark}
An intuitive explanation for the $\frac{1}{\sqrt{u}}$ scaling for $d\ge 5$ is the following. The capacity along a random walk path grows linearly with the run time. At intensity $u$, the expected number of paths in the random interlacements intersecting a random walk running $t$ steps is of order $ut$. Thus one needs to take $t=\frac{1}{u}$ to get order unity. Next, note that the Euclidean distance a random walk traverses after $\frac{1}{u}$ steps is of order $\frac{1}{\sqrt{u}}$.

As for dimensions $d=3,4$, the arguments in this paper fail, but we believe that, for $d=4$, the results should hold with the same power $1/2$. For $d=3$, we expect a different power, smaller than $1/2$, due to multiple intersections of the paths on all scales. However, we have no concrete conjecture for the exact power. Note that even for the chemical distance along a three-dimensional simple random walk, there is no explicit value for the power \cite{shiraishi2018growth}. 
\end{remark}

\begin{remark} \label{rmk: other_defs}
	The time constant $\rho_{u}$ is the asymptotic norm of the chemical distance in $\cI^u$.
	In this work, we have introduced $\cI^u$ as the graph whose vertices and edges are the ones traversed by the random interlacements $\RI^{u}$ (see~\eqref{eq:RI-def} below for a formal definition).
	This graph has been studied in~\cite{RathSapozhnikov}.
	However, previous works on random interlacements have used a different definition for the associated graph. 
	In \cite{vcerny2012internal,DrewitzRathSapozhnikov,Sznitman}, the authors focus on the interlacements set, which is the set of vertices $\hat{\cI^{u}}$ in the range of the trajectories of the random interlacements. 
	Then they analyse the chemical distance $\hat{d}_{u}$ associated to the subgraph of $\mathbb{Z}^d$ induced by $\hat{\cI^{u}}$. That is, any pair of vertices $x,y\in\hat{\cI^{u}}$ such that $|x-y|_1=1$ are connected, regardless of whether or not the edge $(x,y)$ was crossed by a trajectory in $\RI^{u}$. 
	While the sets of vertices in  the induced subgraph $\hat{\cI^{u}}$ and the graph $\cI^u$ are equal, they  differ in their sets of edges. We note that $ \trace ( \cI^u )  \subseteq \trace ( \hat{\cI^u} )$ and from this relation we obtain $ \hat{d}_{u} (x,y) \leq {d}_{u} (x,y)$, for every $x , y \in \hat{\cI^{u}}$. 
	It follows that the upper bound in Theorem~\ref{eq:main_upper} holds for the asymptotic norm of $ \hat{d}_{u}$. 
	Our arguments should lead to the corresponding lower bound for $\hat{\cI^{u}}$ as in Theorem~\ref{thm:main_lower}. We do not include the proof of this case to avoid further technicalities and highlight the novel ideas of our arguments.  
\end{remark}


As discussed above, \v{C}ern{\'y} and Popov proved in~\cite{vcerny2012internal} that the chemical distance $d_u$ satisfies a shape theorem. The limit shape is the closed unit ball in the $\rho_u$ norm, namely
\[
		D^u \coloneqq \{ x \in \mathbb{R}^d \st  \rho_u (x) \leq 1 \}.
\]

The emergence of a limit shape is common to several models in probability. The first  shape theorem was proved for the Richardson's growth model~\cite{Richardson1973}. This prototypical theorem was generalized to first-passage percolation by Cox and Durrett~\cite{CoxDurrett} and  Kesten~\cite{Kesten1993}. Shape theorems have also been proved for the frog model~\cite{AlvesMachadoPopov}, the spread of an epidemic~\cite{KestenSidoravicius2008}, the orthant model~\cite{HolmesSalisbury}, and a family of percolation models with long-range correlations~\cite{andres2021first,boivin1990first, DrewitzRathSapozhnikov}, to cite some examples.
In every instance, understanding the geometry of the limit shape is a fundamental question but challenging for some models (see~\cite{AuffingerDamronHanson} for a survey on the progress for first-passage percolation). 
Some geometric understanding has been achieved via the ``Wulff construction program'' by presenting the limit shape as a ball in some deterministic norm, e.g. super-critical percolation \cite{alexander1990wulff,biskup2015isoperimetry,Gold2018}, Ising model \cite{Cerf2006,CerfPisztora,dobrushin1992wulff,ioffe1998dobrushin} and self-interacting random walks \cite{berestycki2019condensation, biskup2018eigenvalue}. 
In the case of random interlacements, there are some immediate observations on the limit shape $D^u$ and a standing conjecture.

When $u \to \infty$, the range of random interlacements monotonically converges to $\mathbb{Z}^d$. Therefore it is easy to verify that $D^u$ converges to the closed $\ell^1$-ball in the Hausdorff metric for compact subsets of $\mathbb{R}^d$. This result means that in the high-intensity regime, the chemical distance for random interlacements behaves as the graph distance of $\mathbb{Z}^d$, and hence in the scaling limit, as  $  \vert \cdot \vert_1 $ on $\mathbb{R}^d$.

In the low-intensity regime, a conjecture in~\cite{vcerny2012internal} states that $D^u$ (when scaled appropriately) converges to the Euclidean unit ball as $u \to 0$.  
Within this context, Theorems~\ref{thm:main_upper} and~\ref{thm:main_lower} indicate the correct scaling to obtain such convergence.  
It would thus be interesting to improve Theorem \ref{thm:main_lower} to match the upper bound in Theorems~\ref{thm:main_upper} and show that there is some constant $c>0$ such that
\[
		\lim_{u\to 0}c\sqrt{u}\rho_u(x)=|x|_2, \quad \text{ for every } x \in \mathbb{R}^d.
\]
Note that in the context of Bernoulli percolation, the scaling of the chemical distance in the limit $p\downarrow p_c(\mathbb{Z}^d)$ is unknown, though a similar conjecture of convergence to the Euclidean norm holds~\cite{Benjamini2013,biskup2015isoperimetry,duminil2013limit}. 

\subsection{Proof strategy}
\subsubsection{{Theorem~\ref{thm:main_upper} (upper bound)}}
We renormalize $\mathbb{Z}^d$ into boxes of size of order $u^{-1/2}$. We show uniformly, as $u\to0$, that there is a high probability for two adjacent boxes to include trajectories that intersect within $u^{-1}$ steps. We then construct a path connecting far away points via an exploration process that follows these local connections. Here we require $d\ge 5$ to ensure sufficient independence between different local connections. 

\subsubsection{{Theorem~\ref{thm:main_lower} (lower bound)}}
Providing lower bounds on the chemical distance  is more delicate since one needs to control all possible paths connecting far away points. Then the main component of the proof is a local lower bound on the chemical distance between the boundaries of two concentric boxes of radii proportional to $L_u:=u^{-1/2+4\varepsilon}$
(see Proposition \ref{prop:annulus}). 
A single random walk path in $d\ge 5$ would cross the annulus defined by the boxes in $L_u^2$ steps.  Since the capacity of a single path in the annulus is of order $u L_u^2=u^{8\varepsilon}$, most trajectories in the random interlacements will not have intersections with other trajectories while crossing the annulus. However, we require a uniform bound on the chemical distance, and the proof proceeds via stochastic domination by a sub-critical branching
process. This branching process dominates the number of trajectories in each cluster (the connected components in the annulus).
We then bound the chemical distance in each cluster. Subsection~\ref{subsec: chemDist ring} presents an extended sketch of the proof of the bound for each cluster.

Once we obtain local control on the chemical distance, the proof proceeds via a multi-scale renormalization argument.  
We call a box ``good'' if it satisfies the local lower bound on the chemical distance.
Using renormalization, we prove that any sufficiently long path must cross enough good boxes. Each of these good boxes contributes an order of $L_u^2$ to the chemical distance, thus leading to the stated lower bound in Theorem \ref{thm:main_lower}. Note that in most applications of a renormalization argument, there is some relationship (e.g. connectivity) between adjacent good boxes. Our use of the multi-scale renormalization is simpler since we do not have a connectivity assumption, but we only need the density of good boxes to be sufficiently high.


\subsection{Organization of the paper}

Section~\ref{sec:Background} summarizes background on random interlacements and essential properties of simple random walks in $d \geq 5$.
In Section~\ref{sec:capacity}, we provide concentration results on the number of cut-points along the range of a random walk in dimension $d \geq 5$, as well as capacity estimates.
The proof of Theorem~\ref{thm:main_upper} is in Section~\ref{sec:upper_bound}. The proof of Theorem~\ref{thm:main_lower} is the content of Sections~\ref{sec:local_lower_bound} and~\ref{sec: lower bound}. In Section~\ref{sec:local_lower_bound}, we provide a local lower bound. Section~\ref{sec: lower bound} completes the proof with a multi-scale renormalization argument.

\section{Background and notation} \label{sec:Background}

\subsection{Sets and norms}

We denote the $l^1$, $l^2$ and $l^{\infty}$ norms on $\mathbb{R}^d$ by $ \vert \cdot \vert_1 $, $\vert \cdot \vert_2$ and $\vert \cdot \vert_{\infty}$, respectively.

The graph $\mathbb{Z}^d$ is defined by the set of vertices $ \{ (x_1, \ldots, x_d) \st x_i \in \mathbb{Z} \text{ for } i = 1, \ldots d \} $ and the set of \emph{nearest-neighbour edges} $E(\mathbb{Z}^d) \coloneqq \{ (x,y) \in \mathbb{Z}^d  \times \mathbb{Z}^d \st  \vert x - y \vert_1 = 1 \} $. We often denote the set of vertices by $\mathbb{Z}^d$, and write $x \sim y$ when $(x , y ) \in E(\mathbb{Z}^d) $.
If $x = (x_1, \ldots , x_d) \in \mathbb{Z}^d$ and $r > 0$, let
\[
	r \cdot x \coloneqq \left( [ r\cdot x_1], \ldots,  [r \cdot x_d ] \right),
\]
where in each coordinate we take the integer part of the product.

For $z \in \mathbb{Z}^d$ and $r > 0$ we denote the half-closed discrete box by
\begin{equation} \label{eq: box}
		B_z (r) \coloneqq z + [- r , r)^d \cap \mathbb{Z}^d,
\end{equation}
and write $ B (r) \coloneqq B_0 (r) $.

Let $K \subset \mathbb{Z}^d$. The \emph{inner vertex boundary} of $K$ is the set of vertices
\[
	\partial_{i} K \coloneqq  \{  x \in K \st  x \sim y \text{ for some } y \in \mathbb{Z}^d \setminus K \},
\]
and the \emph{outer vertex boundary} of $K$ is 
\[
	\partial K \coloneqq \{ x \in \mathbb{Z}^d \setminus K \st  x \sim y \text{ for some } x \in K \}.
\]
The \emph{edge boundary} of $K$, denoted by $\partial_{e} K$, is the set of edges that connect $K$ with its complement $\mathbb{Z}^d \setminus K$.

\subsection{Spaces of trajectories}\label{subsec_trajec}

For $-\infty\leq n< m\leq \infty$ define $[n,m]_ \mathbb{Z} = [n,m]\cap\mathbb{Z}$. Whenever $n$ or $m$ are infinite, the interval does not include infinity.  A \textbf{path} in $\mathbb{Z}^d $ is a function 
 \[  w : [  n,  m ]_{\mathbb{Z}}   \to  \mathbb{Z}^d \]
such that $\vert w (k) - w(k+1) \vert_1 = 1$ for all $k\in [n,m-1]_\mathbb{Z}$ and some $ -\infty \leq n < m \leq \infty$. 
The path $w$ may be either finite ($-\infty< n , m <\infty$), infinite (e.g. $n >-\infty$ and $m = \infty$),  or doubly-infinite ($n=-\infty$ and $m = \infty$).

The length of a finite  path $w$ is $ \len (w) = m - n $, otherwise we set $\len (w) = \infty$.
For a path $w$, we write $\sPath (w) = n$ and $\ePath (w) = m$. When these are finite, they represent the first and last time-indices of the path.
A path $w$ is \emph{transient} if $w^{-1} (x)$ is a finite set for every $x \in \mathbb{Z}^d$, meaning that $w$ visits any vertex a finite number of times.
We often define a path $w : \{ 0, \ldots ,m \} \to \mathbb{Z}^d$  by specifying a sequence of nearest-neighbour vertices in $\mathbb{Z}^d$, e.g. $ w = [ w(0), \ldots, w(m) ] $.

For $ -\infty \le n <  m \le \infty$, let  $\mathcal{W} [n,m] $ be the collection of transient paths of the form $ w : [n,m]_\mathbb{Z} \to \mathbb{Z}^d $.
We define the \emph{spaces of doubly-infinite paths and the space of finite paths}, respectively, by 
\[  
	\mathcal{W} \coloneqq \mathcal{W}[-\infty,\infty] , \qquad 	\mathcal{W}_f \coloneqq \bigcup_{ -\infty < n \leq m < \infty } \mathcal{W}[n,m].
\]
Note that  $\mathcal{W} [0, \infty]$ is the \emph{space of (one-sided) infinite paths}.
We equip each of the spaces above with the topology generated by the coordinate maps and the associated Borel $\sigma$-algebra.

Given a path  $w \in \mathcal{W}[n,m]$ (finite or infinite) with $n>-\infty$ and $ n \leq  T < m $, we define $w\vert_{\ge T}\in \mathcal{W}[0,m-T]$, by $w\vert_{\ge T}(i)=w(i-T)$ for all $i\in [0,m-T]_\mathbb{Z}$.

For $k \in \mathbb{Z}$, let  $\theta_k :  \mathcal{W} \to  \mathcal{W}$ be the time-shift function defined by $ \theta_k (w) (i) := w(i+k)$ for all $i\in\mathbb{Z}$.

We define an equivalence relation on paths, by saying that
$w_1 \sim w_2$ if $w_1 = \theta_k (w_2)$ for some $k \in \mathbb{Z}$. 
Let $\mathcal{W}^*$ be the quotient space $\mathcal{W} / \sim$ with quotient map $\pi : \mathcal{W} \to \mathcal{W}^*$. 
We endow $\mathcal{W}^*$ with the $\sigma$-algebra generated by the quotient topology. 
Let $w^* \in \mathcal{W}^*$ be a trajectory and let $ w \in \mathcal{W}$ be a class representative in $\pi^{-1} (w^{*})$.
The \emph{range} of a trajectory $w^*$ (and each of its representatives) is defined as
\[
	\range (w^*)  \coloneqq \range (w) \coloneqq \bigcup_{  j \in [\sPath (w) ,  \ePath (w)]_{\mathbb{Z}}  } w(j), 
\]
and the \emph{trace} of $w^*$ (and each of its representatives) is the union of edges traversed by the walk
\[
	\trace (w^*)  \coloneqq \trace (w) \coloneqq \bigcup_{ j\in [\sPath (w) ,  \ePath (w)-1]_{\mathbb{Z}} } (w(j),w( j + 1)),
\]
and $\len (w^*) = \len (w)$.
Note that the range, the trace and the length of a path are properties of its equivalence class, so they are well-defined for a trajectory.

Let $K$ be a finite subset of $\mathbb{Z}^d$. For a path $w \in \mathcal{W} [n,m]$, we refer to the stopping time
\begin{equation} \label{eq: w-entrance}
	H_K (w) \coloneqq \inf \{j\in [n,m]_\mathbb{Z} \st  w(j) \in K \} 
\end{equation}
as the \emph{first entrance time} of $K$; and to the stopping time
\begin{equation} \label{eq: w-Positivehitting}
	\tilde{H}_K (w) \coloneqq \inf \{j\in [n+1,m]_\mathbb{Z} \st w(j) \in K \}
\end{equation}
as the \emph{first hitting time} of $K$, where in both cases we use the convention that the infimum of the empty set is infinity. Note that both stopping times are always well defined since all paths are transient. When it is clear from the context, we simply write $H_K$ and $\tilde{H}_K$ (without reference to $w$).

We denote the collection of doubly-infinite paths intersecting $K$ by 
\[ 
		\mathcal{W}_K  \coloneqq \{ w \in \mathcal{W} \st H_K (w) < \infty \},
\]
and $\mathcal{W}_K^* \coloneqq \pi (\mathcal{W}_K)$ is the corresponding collection of trajectories. 
We will consider paths in $\mathcal{W}_K$ starting from their first entrance to $K$. We thus define
\[
		\vec{\mathcal{W}}_K  \coloneqq  \{ w \in \mathcal{W} [0, \infty] \st w(0) \in K \}.
\]
For a path $w \in \mathcal{W}_K$, its sub-path starting at $K$ is $w^0_K \coloneqq w \vert_{\geq H_K}  \in \vec{\mathcal{W}}_K$. 
With this definition, we obtain a map $  $ 
\begin{equation} \label{eq: start function}
		\vec{s}_K : \mathcal{W}_K \rightarrow  \vec{\mathcal{W}}_K, \quad \vec{s}_K (w)= w^0_K .
\end{equation} 
Note that the corresponding map for the space of trajectories $\vec{s}^{\, *}_K \coloneqq \vec{s}_K \, \circ \, \pi^{-1} :  \mathcal{W}^*_K \rightarrow \vec{\mathcal{W}}_K $ is well-defined.

\subsection{Random walk, capacity and equilibrium measure}

A \emph{simple random walk} starting at $x$ is a stochastic process $X = (X(n))_{n \geq 0}$, taking values in the set of vertices $\mathbb{Z}^d$, with law
$\prw_x$ which satisfies $\prw_x (X (0) = x) = 1$ and
\[
	\prw_x \left( X  (n + 1) =  x \, \vert \, X (n)  = y \right) = \frac{1}{2d} \ind{x \sim y}.
\]

The process $X$ defines an infinite path $[X(0), X(1), \ldots] \in \mathcal{W} [0, \infty]$, which we also call a simple random walk. Thus $\prw_x$ can also be considered as a measure on $\mathcal{W}[0,\infty]$. For a simple random walk $X$, we often write $ X [0, n] = [X(0), \ldots , X(n) ]\in\mathcal{W}[0,n]  $ to indicate the path up to time $n$.

The \emph{equilibrium measure}  of a finite set $K \subset \mathbb{Z}^d$ is defined by
\begin{equation} \label{eq: equilibrium measure}
	e_K (x) \coloneqq \prw_x \left( \tilde{H}_K = \infty \right) \ind{x \in K}, \quad \text{ for } x \in \mathbb{Z}^d; 
\end{equation}
and its \emph{capacity} by \[ \capa (K) \coloneqq \sum_{x \in K} e_K (x) .\] 
The \emph{normalized equilibrium measure} $\tilde{e}_K$, defined by 
\begin{equation} \label{eq: normalized equilibrium measure}
	\tilde{e}_K (x) \coloneqq \frac{e_K (x)}{\capa (K)}, \qquad x \in K
\end{equation}
is then a probability measure on $K$.

When the set $K$ is a ball, its capacity is well-known. 

\begin{prop}[{\cite[Proposition 6.5.2]{LawlerLimic}}] \label{prop:capa_ball}
	There exist constants $ \newcnt \label{c:capL} =  \cnt{c:capL}(d)> 0  $ and  $ \newcnt \label{c:capU} =  \cnt{c:capU}(d)> 0  $ such that
	\[
     \cnt{c:capL} \cdot r^{d-2}  \leq  \capa (B_z (r)) \leq \cnt{c:capU} \cdot r^{d-2} .
	\] 
\end{prop}

For a finite set $K \subset \mathbb{Z}^d$, define the $\sigma$-finite measure $Q_K$ on $\mathcal{W}$ as the measure satisfying 
 \[ 
 Q_K \left( (w_{-n})_{n \geq 0} \in A, \, w_0 = x, \, (w_{n})_{n \geq 0} \in B \right)
 = 
 	\prw_x \left(  A \, \vert \, \tilde{H}_K = \infty \right) \cdot  
 	e_K (x) \cdot 
 	\prw_x \left(  B \right)
 \]
for every Borel sets $A , B \subset \mathcal{W}[0,\infty]$.
The existence of $Q_K$ was verified in~\cite{Sznitman}, together with the existence of a pull-back measure $\nu$ on the space of doubly infinite trajectories $\mathcal{W}^* $ satisfying for every finite $K\subset \mathbb{Z}^d$
\begin{equation} \label{eq: measure inter}
	\nu (A ) \coloneqq Q_K (\pi^{-1} (A^*))
\end{equation}
for every Borel set $A^*\subset\mathcal{W}^*_K$. The measure $\nu$ is used below to define random interlacements.

For a probability measure $\mu$ supported on the finite subset $K \subset \mathbb{Z}^d$, define
\[
	\prw_{\mu} \coloneqq \sum_{x \in K} \mu (x) \prw_x
\]
to be the law of a simple random walk with initial probability distribution $\mu$. 
For random interlacements, the measure $\prw_{\tilde{e}_K }$ has a role in its local construction (see~\eqref{eq: local-forward measure} below).

\subsection{Random interlacements}  \label{subsection:RI}

Consider the $\sigma$-finite measure $\nu \otimes \Lambda$ on the space of labelled doubly-infinite trajectories $\mathcal{W}^* \times \mathbb{R}_{+}$. Here $\Lambda$ is the Lebesgue measure on $\mathbb{R}_{+}$ and $\nu$ is the measure defined in~\eqref{eq: measure inter}. 
The \emph{random interlacements process}, denoted by $\RI $,  is the Poisson point process on $ \mathcal{W}^* \times \mathbb{R}_{+}$ with intensity measure $ \nu \otimes \Lambda$. For a realization of the random interlacements process
\begin{equation} \label{def: RI} 
	\RI = \sum_{n \geq 0} \delta_{(w_n^*, u_n)},
\end{equation} 
where $(w_n^*, u_n) \in \mathcal{W}^* \times \mathbb{R}_{+}$ for each $n \geq 0$, the
\emph{random interlacements at level $u$} is the random point measure on the space $\mathcal{W}^* \times\mathbb{R}_{+}$ defined by
\[
	\RI^u \coloneqq \sum_{u_n \leq u} \delta_{(w_n^*, u_n)}.
\]
The support $ \supp (\RI^{u}) $ is the collection of trajectories in $\RI$ at level $u$.

The  \emph{random interlacements graph} $\cI^{u}$ (at density level $u$) is the subgraph of $\mathbb{Z}^d$ with 
\begin{equation} \label{eq:RI-def}
	\bigcup_{ u_n \leq u } \range (w_n^*)\qquad, \qquad  \bigcup_{ u_n \leq u } \trace (w_n^*)
\end{equation}
as its set of vertices and edges, respectively. 
With this definition, the edge set corresponds to those edges  traversed by trajectories in the support of $\RI^u$. 
We write $\prob^{u}$ for the law of $\cI^{u}$ on the probability space $\{ 0,1 \}^{E} $, equipped with the $\sigma$-algebra $\mathcal{F}$ generated by the canonical coordinate maps
\begin{equation} \label{eq:coordmaps}
	\Psi_{e} : \{ 0, 1 \}^E \to \{ 0, 1 \}, \qquad e \in E
.
\end{equation}
where $ \Psi_e ( \zeta  ) = \zeta_e $ for $ \zeta \in \{ 0,1 \}^{E} $.

In the multi-scale renormalization argument, we will use two well-known properties of random interlacements: stochastic monotonicity in $u$ and decorrelation for monotone events, that we now recall.

An event $G \in \mathcal{F}$ is called \emph{increasing} (respectively, \emph{decreasing}) if for all $\zeta \in G$ and $ \hat{\zeta}  \in  \{ 0, 1 \}^E  $ satisfying 
$ \zeta_e \leq \hat{\zeta}_e $ (respectively $ \zeta_e \geq \hat{\zeta}_e $) for all $e \in E$ it holds that $ \hat{\zeta} \in G $. The events $(G_i)_{i \geq 1}$ are \emph{monotone} if either all $G_i$ are increasing, or all $G_i$ are decreasing events.

The random interlacements measure preserves the monotonicity of events, and this property is known as stochastic monotonicity. It is an immediate consequence of the definition of the random interlacements graph in~\eqref{eq:RI-def}.

\begin{prop} \label{prop:RI_monotone}
	Let $ \hat{u} > u > 0$, then  $\prob^u (G) \leq \prob^{\hat{u}} (G)$  for any increasing event $G \in \mathcal{F}$. 
\end{prop}

The next theorem follows from  \cite[{Theorem 1.1}]{PopovTeixeira}. The proof of this theorem first goes through $\mathcal{I}^u$, and then proceeds to the interlacements set. We state the inequality for the interlacements graph, boxes of $\mathbb{Z}^d$ and increasing events (the original paper states the result for more general subsets of $\mathbb{Z}^d$ and monotone events).

\begin{thm}[{\cite[{Theorem 1.1}]{PopovTeixeira}}] \label{thm:RI_decorrelation}
	There exist constants $ \newcnt \label{c:RIdecoupling} = \cnt{c:RIdecoupling} (d) > 0$ and $ \newexpo \label{e:RIdecoupling}  =  \expo{e:RIdecoupling} (d)  > 0$ such that the following holds.
	Let $L \geq 1$ be an integer and let $x_1$ and $x_2$ be vertices of $\mathbb{Z}^d$ with $\vert x_1 - x_2 \vert > 2L$.
	Let $R > 0$ be the distance between $B(x_1,L)$ and $B(x_2,L)$. 
	For $i = 1, 2$, let
	$A_i \in \sigma  ( \Psi_{(w,z)} \st w,z \in B(x_i, L)  ) $ 
	be increasing events.
	For any $u > 0$ and $\varepsilon \in (0,1)$ we have
	\[
	\prob^u \left( A_1 \cap A_2 \right) \leq 
	\prob^{ (1 + \varepsilon) u } \left( A_1 \right) \cdot \prob^{ (1 + \varepsilon) u } \left( A_2 \right) + \cnt{c:RIdecoupling} ( R  + 3L)^d  e^{ - \expo{e:RIdecoupling}\varepsilon^2 u R^{d-2}} .
	\]
\end{thm}

\subsection{Intersections of random interlacements and a set}

For a finite subset $K \subset \mathbb{Z}^d$, recall that $\mathcal{W}_K^*$ is the set of trajectories intersecting $K$,  $\vec{\mathcal{W}}_K$ is the set of infinite paths starting at $K$ and $\vec{s}^{\, *}_{K} $ is a map between these spaces (defined after~\eqref{eq: start function}).
Let $u > 0$. For $ \RI = \sum_{n \geq 0} \delta_{(w_n^*, u_n)}$, define  
\begin{equation} \label{eq: measure paths at K}
		 \vec{\RI}^{u}_{K} \coloneqq \sum_{u_n\leq u} \delta_{ \vec{s}^{\, *}_K ( w_n^*)} \ind{ w_n^* \in \mathcal{W}^{*}_K }.
\end{equation}
Note that  $\vec{s}^{\, *}_K ( w_n^*)$ is the path representing the range of the trajectory $w_n^*$ in $\RI^{u}$ after its first entrance to the set $K$, so $\vec{\RI}^{\, u}_{K}$ is the point measure of such paths up to level $u$.

Next, we define a random point measure with the same distribution as $ \vec{\RI}^u_{K} $, providing a local construction of random interlacements intersecting a finite set $K$.
Let $N_K$ be a Poisson random variable with parameter $ u \cdot \capa (K) $, and let   $ (X^j)_{j \geq 1} \subset \vec{\mathcal{W}}_K$ be i.i.d. random walks with distribution $P_{\tilde{e}_K}$, which are independent of $N_K$. Then the point measure on $\vec{\mathcal{W}}_K$ is
\begin{equation} \label{eq: local-forward measure}
		\vec{\mu}^{\, u}_{K} \coloneqq  \sum_{ j = 1}^{N_K} \delta_{X^j}.
\end{equation}
In the rest of this paper,  $N_K$ indicates the number of trajectories in $\RI^u$ 
intersecting the set $K$.

The construction of Poisson point processes in \cite[Equation (4.2.1)]{DRSbook} implies the equivalence of the two measures above.

\begin{prop} \label{prop: K-intersection}
	The point measures  $\vec{\mu}^{\,u}_{K}$ and $\vec{\RI}^{u}_K $ are equal in distribution.
\end{prop}

\subsection{Stochastic domination}

For random variables $X$ and $Y$ (or more generally distributions of random variables), we denote by $X \geq_{\ST} Y$ the fact that $X$ stochastically dominates $Y$.

\begin{prop} \label{prop: Poisson ST}
	Let $r > s > 0$ and consider random variables $X \sim \Poi (r)$ and $Y \sim \Poi (s) $. Then $ X \geq_{\ST} Y $. 
\end{prop}

\subsection{Asymptotic notation and constants}

Throughout the rest of the paper we write $C$, $c$ and $c'$ for constants (depending only on the dimension) whose value might change from one appearance to the next. Constants whose values are important for the proofs, and that might depend on additional parameters,  are numbered, denoted by Greek letters (e.g. $\xi,\eta,\beta$) and fixed in their first appearance.

We write $f(x) \asymp g(x)$ if there exist positive constants $c, c'$ (depending only on $d$) such that $c g(x) \leq f(x) \leq c'g(x) $ for all $x > 0$. 
Similarly,  $ f(x) \preceq g(x) $ indicates that there exists a positive constant $c$, depending  only on the dimension, such that $f(x) \leq c g(x)$ for all $x > 0$.

Finally, we occasionally use real numbers when integers are required. In such occurrences, one should always take either the ceiling or the floor, depending on the instance.

\section{Estimates on the capacity and  the intersections of random walks}
\label{sec:capacity}

In this section, we recall some estimates regarding the capacity of random walks, the number of intersections between different random walks, and the chemical distance in the range of a random walk in high dimensions, which are needed for the proof of Theorem~\ref{thm:main_upper} and Theorem~\ref{thm:main_lower}.

\subsection{Capacity of the range of a random walk}

Let $X = (X (n))_{n \geq 0}$ be a simple random walk on $\mathbb{Z}^d$. 
Recall that we write $X [0,n] = [X(0), \ldots , X(n)] $, and denote $\range_n:=\range (X [0,n])$.

The behaviour of the capacity of the range of a simple random walk in high dimension has been  thoroughly studied in~\cite{asselah2020deviations,AsselahSchapiraSousi,JainOrey,Schapira2020}. 
For the purposes of this work, we only need a law of large numbers and a large deviations result.

\begin{thm}[{\cite[Theorem 2]{JainOrey}}] \label{thm: capacity_lln}
 For every $d \geq 5$, there exists a constant $\gamma_d > 0$  (depending only on  the dimension), such that the following holds.
	For a simple random walk $X = (X (n))_{n \geq 0}$ on $\mathbb{Z}^d$
 \begin{equation} \label{eq: SLLN capa}
 	 \lim_{n \to \infty} \frac{\capa (\range_n)}{n} = \gamma_d \qquad \prw-\text{a.s.}
 \end{equation}
 In particular,  for any $\gamma < \gamma_d$ 
	\[
	 \lim_{n \to \infty} \prw \big(  \capa (\range_n ) > \gamma  n   \big ) = 1. 
	\]
\end{thm}

\begin{thm}[{\cite[Theorem 1.7]{asselah2020deviations}}]\label{thm:cap_ldp}
	For every $d\ge 5$, there exist $\gamma^{+}_d >\gamma_d$ (with $\gamma_d$ as in Theorem \ref{thm: capacity_lln}), $c > 0$ and $\newexpo  \label{e:cap_ldp}  \in (0, \infty)$ such that for all $n \geq 1$
	\[ 
		\prw\left(\capa(\range_n)> \gamma_d^+  n\right)\le c e^{- \expo{e:cap_ldp} n}.
	\]
\end{thm}

\subsection{Escape probabilities and intersections of random walks}

We first recall a pair of known results on the first entrance time to a subset of vertices.

\begin{prop}[{\cite[Proposition 2.4.5]{LawlerLimic}}] \label{prop: RW escape}
	Let $d \geq 5$ and let $X$ be a simple random walk on $\mathbb{Z}^d$ starting at $0$. 
	There exist $ \newexpo \label{e:RWescape} = \expo{e:RWescape} (d)\in (0,\infty)$ and $c\in (0,\infty)$ such that for all $n \geq 0$ and all $\lambda > 1$,
	\begin{equation} \label{eq: RW escape}
		\prw \left(\lambda^{-1} n^2 \leq H_{\partial B (n)}  
		\leq \lambda n^2  \right) \geq 1- c e^{- \lambda \expo{e:RWescape} }. 
	\end{equation}
\end{prop}

\begin{prop}[{\cite[Lemma 3.1]{CaiHanYeZhang}}] \label{prop: RW and set}
	Let $X$ be a simple random walk starting at $x\in\mathbb{Z}^d$. 
	For every $M > 0 $, there exists $\newet\label{et:rws}=\et{et:rws}(d,M)>0$ such that, for any $R\geq 1$and any $A \subset B_x ( M R)$ satisfying $d (x, A) \leq  R $, it holds that 
	\[
		\prw_x \left( H_A < R^2 \right) \geq \frac{\et{et:rws} \cdot \capa (A) }{R^{d-2}}.
	\]
\end{prop}

Next we review a large deviations result regarding intersections of random walks, that we use in Section~\ref{sec:local_lower_bound} in the proof of  Theorem~\ref{thm:main_lower}.

Given $k \geq 1$, consider a set of $k$ independent simple random walks on $\mathbb{Z}^d$, starting at arbitrary points.
Denote by $\mathcal{N}_k$ the number of points in $\mathbb{Z}^d$ which are visited by at least two of the random walks. We first cite a known result for $ k = 2$.

\begin{thm}[{\cite[Theorem 1.1]{asselah2020large}}] \label{thm:intersections2rw}
	Let $ d \geq 5$. There exists constants $\newexpo > 0 \label{e:intersections2rw}$ and $C>0$ such that for every $m \geq 1$ 
	\[
		\prw\left(\mathcal{N}_2>m\right) < Ce^{-\expo{e:intersections2rw} m^{1-\frac{2}{d}} }.
	\]
\end{thm}

From the last theorem we obtain a corollary for $k$ independent random walks. The proof follows from a union bound over the $\binom{k}{2}$ choices of pairs of paths intersecting and an application of Theorem~\ref{thm:intersections2rw} for each pair.

\begin{cor}\label{cor:numofintersections}
	Let $ d \geq 5$ and let $ \expo{e:intersections2rw} > 0$  be the exponent defined in Theorem~\ref{thm:intersections2rw}. Then for every $k > 1$ and $m \geq 1$,
	\[ 
		\prw\left(\mathcal{N}_k>k^2m\right) < k^2 \cdot e^{- \expo{e:intersections2rw} m^{\frac{3}{5}} }.
	\]
\end{cor}

\subsection{Chemical distance in the range of a random walk}\label{subsec"chem_random_walk}

In dimensions $d\geq5$, the high-dimensional setting, our understanding of the chemical distance in the range of a simple random walk is obtained by studying its cut-points. The arguments in this section are an adaptation to dimensions $d\geq 5$ of results obtained in \cite[Section 7.7]{Lawb} for dimension $d = 4$.

Let $X  =  (X (t))_{t \in \mathbb{Z}} $ be a doubly-infinite random walk. We call $t_0\in\mathbb{Z}$ a \emph{global cut-time} if
\begin{equation} \label{def:cut_point}
	X (-\infty, t_0] \cap X [t_0 + 1, \infty)  = \emptyset.
\end{equation}
In this case, we call $X(t_0)$  a \emph{global cut-point}.
The existence of global cut-points corresponds to the non-intersection event of two independent random walks.

The next lemma is a localization result, showing that it suffices to verify that the non-intersection property within a large, yet finite, number of steps.
For $-\infty\leq n<m\leq \infty$, we call $t_0\in [n,m-1]_\mathbb{Z}$  a \emph{local cut-time of $X$ of the interval $[n, m]_\mathbb{Z}$} if
\begin{equation}  \label{eq:def_local_cut}
	X [n, t_0] \cap X [t_0 + 1, m ] = \emptyset ,
\end{equation}
in which case we say that $X(t_0)$ is a \emph{local cut-point}.

\begin{lemma}[{c.f.~\cite[Lemma 7.7.3]{Lawb}}] \label{lemma:locality}
	Let  $d \geq 5$ and let $X$ and $X'$ denote independent simple random walks on $\mathbb{Z}^d$ starting at the origin. Then
	\begin{equation}\label{localization}
	\prw \left( X[0,k] \cap X'[1,k] = \emptyset \right)=
	\prw \left( X[0,+\infty) \cap X'[1,+\infty) = \emptyset \right)
	\cdot \left(1+O( k^\frac{4-d}{2}) \right),
	\end{equation}
	as $k\to\infty$.
\end{lemma}

\begin{proof}
	First note that, since $d \geq 5$, by \cite[Theorem 3.5.1]{Lawb} 
	\begin{equation}\label{dim_five_cut_point}
	\prw \left( X[0,+\infty) \cap X'[1,+\infty) = \emptyset \right)>0,
	\end{equation}
	and thus it is enough to show that the difference of the probabilities is $O( k^\frac{4-d}{2}) $ as $k\to\infty$.

	Since
	\begin{equation}\label{local_triv_ineq}
		\prw \left( X[0,k] \cap X'[1,k] = \emptyset \right)\geq
		\\
		\prw \left( X[0,+\infty) \cap X'[1,+\infty) = \emptyset \right),
	\end{equation}
	it suffices to focus on the reverse inequality. To this end note that 
	\begin{equation}\label{loc_union_bound}
		\begin{split}
			&\prw\left( X[0,+\infty) \cap X'[1,+\infty) \neq \emptyset \right) 
			\\
			 \leq &\prw\left( X[0,k] \cap X'[1,k] \neq \emptyset \right)+
			\prw\left( X[0,k] \cap X'[k,+\infty) \neq \emptyset \right)\\
			 + &\prw\left( X[k,+\infty) \cap X'[1,k] \neq \emptyset \right)+
			\prw\left( X[k,+\infty) \cap X'[k,+\infty) \neq \emptyset \right) \\
		  \leq  &\prw\left( X[0,k] \cap X'[1,k] \neq \emptyset \right)+
			3 \prw\left( X[k, \infty) \cap X'[0,+\infty) \neq \emptyset \right).
		\end{split}
	\end{equation}
	Thus, in order to deduce \eqref{localization} from \eqref{local_triv_ineq} and \eqref{loc_union_bound},  we only need to show
	\begin{equation}
		\prw \left( X[k,+\infty) \cap X'[0,+\infty) \neq \emptyset \right) = 
		O( k^\frac{4-d}{2})
	\end{equation}
	as $k\to\infty$.

	Turning to prove the last estimation, note that reversibility, the Markov property of simple random walk and heat kernel estimations \cite[Section 4, Example 1]{Barlow2017} imply
	\begin{equation}\label{one_meets_whole}
		\begin{split}
		\prw \left( X (i) \in X'[0,+\infty) \right)
		&\leq \sum_{t=0}^\infty \prw \left( X (i)  = X'(t)   \right)
		= \sum_{t=0}^\infty \prw \left( X(i + t)  = 0 \right)
		  \\
		& =
		\sum_{j=i}^\infty \prw \left( X (j)=0 \right) \asymp i^{\frac{2-d}{2}},
		\end{split}
	\end{equation}
	and hence
	\begin{equation}
		\prw \left( X[k,+\infty) \cap X'[0,+\infty) \neq \emptyset \right) 
		\leq 
		\sum_{i=k}^\infty \prw \left( X (i) \in X'[0,+\infty) \right)
		\leq
		 C \sum_{i=k}^\infty i^{\frac{2-d}{2}} \asymp k^{\frac{4-d}{2}},  
	\end{equation}
	where the second inequality follows from~\eqref{one_meets_whole}.
\end{proof}


\begin{lemma}[{c.f.~\cite[Lemma 7.7.4]{Lawb}}]  \label{lemma:global_cut}
	Let $X^1$ and $X^2$ be two independent simple random walks on $\mathbb{Z}^d$ starting at the origin and let $X $ be the doubly-infinite simple random walk defined by $X^1$ and $X^2$, that is $X(t)=X_1(t)$ for $t\geq 0$ and $X(t)=X_2(-t)$ for $t\leq 0$. 
	There exist  $ \newexpo \label{cut} = \expo{cut} (d)  \in (0,\infty)$ and $C=C(d)\in (0,\infty)$ such that for every $N$
	\[
	\prw \left(   
				\begin{array}{c}
					\text{ There exists } j \in \{ 1, \ldots , N \} \text{ such that } \\
					X(-\infty, j] \cap X[j + 1, \infty) = \emptyset
				\end{array}
			 \right)
			 \geq 
			 1 - C(\log N)^{ \expo{cut} } N^{\frac{4-d}{2}}.
	\]
\end{lemma}

\begin{proof}
	Let us choose some $1\leq k \leq N$ (to be specified later) and define 
	\[M=\frac{N}{2k}.\]

	Let us  define the random variables $Y$ and $Y^*$ by
	\begin{align}
	Y &= \sum_{j=1}^M \ind { \{ 2jk \text{ is a global cut-time}  \} },\\
	Y^* &= \sum_{j=1}^M \ind {\{  2jk \text{ is a local cut-time in } X[2jk-k,2jk+k] \}}.
	\end{align}
	Thus $Y$ counts the number of global cut-points among the points $ 2k, 4k, 6k, \dots, 2Mk$ and $Y^*$
	is a local version of $Y$. Now note that $Y^*$ has binomial distribution and that $Y^* \geq Y$. 
	Hence, we bound the probability of the event $Y=0$ in terms of $Y^*$:
	\begin{equation} \label{Y_star_break}
		\prw(Y=0)\le	
		\prw \left( Y < \frac {\mathbf{E} (Y^*)}{4} \right) \leq 
		\prw  \left( Y^* \leq \frac{\mathbf{E}(Y^*)}{2} \right)
		+
		\prw \left( Y^*-Y  \geq \frac{\mathbf{E}(Y^*)}{4} \right).
	\end{equation}

	We now turn to estimate each of the terms on the right hand side of \eqref{Y_star_break}. Starting with the former, since $Y^*$ is a binomial random variable with expectation $CM$ for some $C>0$ (c.f.\ \eqref{dim_five_cut_point}), it follows that 
	\begin{equation} \label{eq:Y_estimate_1}
		\prw \left( Y^* \leq \frac{\mathbf{E}(Y^*)}{2} \right) \leq e^{-cM}.
	\end{equation}

	As for the second term, by Markov's inequality and Lemma~\ref{lemma:locality}:
	\begin{equation} \label{eq:Y_estimate_2}
		\prw \left( Y^*-Y  \geq \frac{\mathbf{E}(Y^*)}{4} \right)
		\leq 
		\frac{ \mathbf{E}\left( Y^*-Y  \right) }{ \mathbf{E}(Y^*)/4 } 
		\leq 
		 C \, \frac{ M \cdot k^\frac{4-d}{2} }{ M  }= C k^\frac{4-d}{2}.
	\end{equation}
	The bounds~\eqref{eq:Y_estimate_1} and~\eqref{eq:Y_estimate_2}, over the terms of the right-hand side of \eqref{Y_star_break}, give
	\[
		\prw \left( Y < \frac {\mathbf{E} (Y^*)}{4} \right)
		\leq 
		e^{ -c M }+  C \left( \frac{N}{2M} \right)^\frac{4-d}{2}.
	\]
	Choosing $M= \frac{ d-4 }{ 2 c'} \log(N)$  gives
	\begin{equation}
	\begin{aligned}
		&\prw \left( 
		\; \not \exists \, j \in \{1,\dots,N\} \; : \;  X(-\infty,j] \cap X[j+1,\infty) =\emptyset 
		\right) \\
		\leq	&
	 	\prw \left( Y=0 \right) 
	 	\leq 
	  C (\log(N))^{\expo{cut} } \cdot N^{\frac{4-d}{2}}
	 \end{aligned}
	\end{equation}
	for some $\expo{cut}\in (0,\infty)$ and $C\in (0, \infty)$.
\end{proof}

\begin{prop}\label{prop:cutdensity}
	Let $X$ be a simple random walk on $\mathbb{Z}^d$. For $n\geq 3$, let $Z_n$ denote the number of global cut-points in $X[0,n]$.  There exist constants $\newcnt \label{c:cutdensity} = \cnt{c:cutdensity} (d) > 0$ and $\newexpo \label{dense} = \expo{dense}  \in (0,\infty)$  such that
	\begin{equation} \label{eq: linear_number_of_cut_points}
		\prw \left( Z_n \leq \cnt{c:cutdensity} \cdot n \right) \leq (\log (n) )^{\expo{dense}} \cdot n^{\frac{4 - d}{2} },\qquad \text{for every }n\in\mathbb{N}.
	\end{equation}
\end{prop}

\begin{proof}
	The idea of the proof is to divide the time-interval $[0, n]$ into $O(\log(n))$ segments. 
	On each segment, we guarantee the existence of two global cut-points, and between them we count local cut-points. This structure will guarantee that the local cut-points are in fact global cut-points.

	For each $n \geq 3$, let $ N \leq n/3 $ be the length in which we partition $[0,n]_{\mathbb{Z}}$ following the next construction. First we partition into segments of length $3N$. 
	For each of the $i$-th segments of length $3N$, we further sub-divide them into three sub-segments of length $N$: $ l_i$, $m_i$ and $r_i$, which we call left, middle and right segments, respectively (see Figure~\ref{fig:cut_point}).
	With 
	\begin{equation} \label{eq:M_choice}
		M = \left\lfloor \frac{n}{3N}  \right\rfloor, 
	\end{equation} 
	we obtain the partition of $[0,n]_{\mathbb{Z}}$ into the segments $ l_1, m_1, r_1, \ldots, l_M, m_M$, and $ r_M  $.

	\begin{figure}[h]
		\begin{center}
		
		\begin{tikzpicture}[scale=0.8]
		\draw [,thick](0,0) -- (12,0);
		
		\draw [,thick](0.0 , .40) -- (0.0 , -.40);
		\draw [,thick](1.5 , .25) -- (1.5 , -.25);
		\draw [,thick](3.0 , .25) -- (3.0 , -.25);
		\draw [,thick](4.5 , .40) -- (4.5 , -.40);

		\draw (.75,0) node[above] {$l_1$};
		\draw (2.25,0) node[above] {$m_1$};
		\draw (3.75,0) node[above] {$r_1$};

		\draw (6,0) node[above] {\Large$\ldots$};

		\draw ( 8.25,0) node[above] {$l_M$};
		\draw ( 9.75,0) node[above] {$m_M$};
		\draw (11.25,0) node[above] {$r_M$};

		\draw [,thick]( 7.5,.40) -- ( 7.5,-.40);
		\draw [,thick]( 9.0,.25) -- ( 9.0,-.25);
		\draw [,thick](10.5,.25) -- (10.5,-.25);
		\draw [,thick](12.0,.40) -- (12.0,-.40);

		\draw [decorate, thick, decoration = {calligraphic brace, mirror,raise=5pt, amplitude=10pt}] (0.0,-.40) --  (4.5,-.40) node[pos=.5,below=15pt,black]{$N$};
		
		\end{tikzpicture}

		\caption{Partition of the time interval $[0, n]$ in the proof of Proposition~\ref{prop:cutdensity}.}
		\label{fig:cut_point}
		\end{center}
	\end{figure}
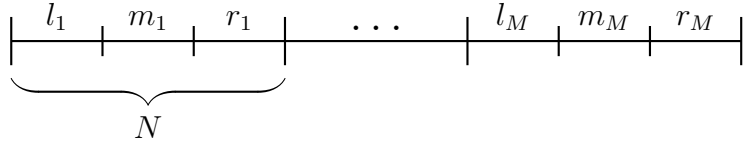

	Let 
	$	
		\mathfs{A} 
	$
	be the event that each of the $2M$ left and right sub-segments contain a global cut-point.
	By Lemma~\ref{lemma:global_cut} and a union bound 
	\begin{equation} \label{eq:prob_A}
		\prw \left(  \mathfs{A}  \right) \geq  1 - CM(\log N)^{ \expo{cut} } N^{\frac{4-d}{2}} .
	\end{equation}

	Let $\tilde{Z}_i$ be the number of local cut-points at $m_i$ of the interval $l_i \cup m_i \cup r_i$. 
	Recall that~\cite[Equation (3.2)]{Lawb} 
	\[
		a = \prw \left(  \text{ the origin is a global cut-point of a doubly-infinite random walk}  \right) > 0,
	\]
	and note that $ \mathbf{E} (\tilde{Z}_i) \geq aN $. Since we also have the bound $\prw \left( \tilde{Z}_i \leq N \right) = 1$, it follows that
	\[
		aN \leq \mathbf{E} \left(\tilde{Z}_i \right) \leq   \frac{a}{2} N \cdot \prw \left( \tilde{Z}_i < \frac{a}{2}N \right) 	
																	  + N \cdot \prw \left( \tilde{Z}_i > \frac{a}{2} N \right)
															 \leq \frac{a}{2} N + N \cdot \prw \left( \tilde{Z}_i > \frac{a}{2} N \right).
	\]
and hence the inequality above implies that
	\[
		\prw \left(B_i \right) \geq \frac{a}{2},
	\]
	where $B_i \coloneqq \left\{ \tilde{Z}_i \geq aN / 2 \right\}$.
	
	By the definition of local cut-point in~\eqref{eq:def_local_cut}, the random variables $(\tilde{Z}_i)_{i=1}^M$ are i.i.d. Thus, there exists some $b \in (0,1)$ such that 
\begin{equation} \label{eq:local_cut_est}
		\prw \left( \sum_{i = 1}^M \tilde{Z}_i  \geq  \frac{a^2}{32}\,  n \right)\geq \prw \left( \sum_{i = 1}^M \ind{B_i }  \geq M \, \frac{a}{4} \right)
		\geq  1 -  b^M.
\end{equation}
	
	On the event $\mathfs{A}$, for each $1  \leq i \leq M $ the left segment $l_i$ as well as the right segment $r_i$ contain a global cut-point. Then, if there is a point  in the middle segment $j \in m_i$ which a local cut-point of the interval $ l_i \cup m_i \cup r_i $ then  $j$ is actually a global cut-point. In particular, on the event $\mathfs{A}$ we have 
	\[
		Z_n \geq \sum_{i = 1}^M \tilde{Z}_i.
	\]
	Therefore, by the estimates in~\eqref{eq:prob_A} and~\eqref{eq:local_cut_est} we get:
	\[
		\prw \left( Z_n \le  \frac{a^2}{32}\,  n \right) \leq 2M(\log N)^{ \expo{cut} } N^{\frac{4-d}{2}} + b^M.
	\]
	Now we choose $N$ by setting $ M = C \log (n) $, for some constant $ C\in (0, \infty)$ large enough (depending only on the dimension). Then the conclusion of the proposition follows.
\end{proof}

\section{Upper bound on the chemical distance in random interlacements} \label{sec:upper_bound}

In this section we prove Theorem~\ref{thm:main_upper}. The key ingredient in the proof is an upper-bound when the direction of the time constant is $x = e_1 \coloneqq (1, 0 , \ldots , 0) \in \mathbb{Z}^d$. 

\begin{prop} \label{prop: upper bound}
	Consider random interlacements on $\mathbb{Z}^d$ with $d \geq 5$.
	There exist $ \newuBound \label{u:propUpper} = \uBound{u:propUpper}(d) \in (0,1) $ and a constant  $ C  = C(d, \uBound{u:propUpper}) \in (0,\infty) $ such that for all $u \in (0 ,  \uBound{u:propUpper}]$
	\begin{equation} \label{eq: upper bound}
		\rho_u (e_1) \leq  \frac{C}{\sqrt{u}}.
	\end{equation}
\end{prop}

The proof of Proposition~\ref{prop: upper bound} is based on a renormalization argument, and is partially inspired by an argument in \cite[Section 4]{ErhardPoisat}. 
The main idea is to construct (for a given $n\in\mathbb{N}$) a path between $0$ and $n e_1$, whose length is bounded by $Cnu^{-1/2}$, provided $n$ is large enough.
Before proceeding with the arguments for constructing such a path, see Subsections~\ref{subsec:upper-renormalization} and~\ref{subsec: proof of upper bound - case 1} below, let us show how Proposition~\ref{prop: upper bound} implies Theorem~\ref{thm:main_upper}. 

\begin{proof}[{Proof of Theorem~\ref{thm:main_upper} given Proposition~\ref{prop: upper bound}}]
	Let $\{   e_1, \ldots, e_{d} \} $ be the canonical basis of $\mathbb{R}^d$. 
	Since the distribution of the simple random walk is invariant under rotations of $\mathbb{Z}^d$, the distribution of random interlacements is invariant under these rotations as well. Consequently, Proposition~\ref{prop: upper bound} implies that for any $u \in (0, \uBound{u:propUpper}]$
	\[
			\rho_u \left( \pm e_j \right) \leq  \frac{C}{\sqrt{u}}, \qquad \text{ for } j = 1, \ldots, d.
	\]
Furthermore, if $u \in (\uBound{u:propUpper}, 1]$ then by monotonicity of random interlacements in $u$, for each $j = 1, \ldots, d$ 
	\[  
		\rho_u ( \pm e_j) \leq \rho_{\uBound{u:propUpper}} (\pm e_j) \leq \frac{C}{\sqrt{\uBound{u:propUpper}}}
		=
		\frac{C}{\sqrt{\uBound{u:propUpper}}} \cdot \frac{\sqrt{u}}{\sqrt{u}}
		\leq \frac{C}{\sqrt{\uBound{u:propUpper}}} \cdot \frac{1}{\sqrt{u}}=\frac{\hat{C}}{\sqrt{u}},
	\]
	where $\hat{C}=\hat{C}(d,\uBound{u:propUpper})\in (0,\infty)$ is a constant.
	
	Finally, since $\rho_u$ is a norm, by the triangle inequality, there exists a constant $C'=C'(d,\uBound{u:propUpper})\in (0,\infty)$ such that for every $x \in \mathbb{R}^d$
	\[
		\rho_u (x) \leq \sum_{i = 1}^d \vert x_i \vert \cdot \rho_u (e_i)
		\leq  \frac{\hat{C}}{\sqrt{u}} \,  \vert x \vert_1 \leq   \frac{C'}{\sqrt{u}}  \, \vert x \vert_{2} , 
 	\]
 	where the last inequality follows from the equivalence of norms on $\mathbb{R}^d$. The theorem follows with $ \Cnt{c:up} =  C'$.
\end{proof}

\subsection{Renormalization} \label{subsec:upper-renormalization}

We begin by introducing the setting for the renormalization. 
For a density parameter $u\in (0,1)$ and a constant $\alpha > 1$, we set 
\begin{equation} \label{eq: R upper bound}
		R_u \equiv R_u^{(\alpha)}\coloneqq \left\lceil \frac{\alpha}{\sqrt{u}} \right\rceil .
\end{equation}
The constant $\alpha$ is a large number that is chosen in the proof of Proposition~\ref{prop: upper bound}.

Let 
\[
	\mathbb{F}^{R_u}   \coloneqq \{  (z_1, \ldots, z_d) \in \mathbb{Z}^d  \st z_1, z_2 \in (2R_u ) \cdot \mathbb{Z}, \text{ and } z_j = 0 \text{ for } j = 3, \ldots ,d \}
\]
be a renormalization of a two-dimensional subgraph of $\mathbb{Z}^d$:
\[
	\mathbb{F}  \coloneqq \{  (z_1, \ldots, z_d) \in \mathbb{Z}^d  \st z_1, z_2 \in  \mathbb{Z}, \text{ and } z_j = 0 \text{ for } j = 3, \ldots ,d \}.
\] 
We make $\mathbb{F}^{R_u}$ into a graph by defining a set of un-oriented edges (including diagonal directions):
\[
 \mathbb{E}^{R_u} \coloneqq \{  (z, y) \in  \mathbb{F}^{R_u} \st \vert z - y \vert_\infty = 2R_u \}.
\]
The pair $(\mathbb{F}^{R_u}, \mathbb{E}^{R_u})$ defines a $2$-dimensional lattice (sometimes called  the $*$-connected square grid).

Consider the union of boxes, as defined in~\eqref{eq: box}, centred at vertices in $\mathbb{F}^{R_u}$: 
\[
	{S}^{R_u} \coloneqq \bigcup_{z \in \mathbb{F}^{R_u}} B_z(R_u).
\]
The set ${S}^{R_u}$ is contained in a slab of two-dimensional subgraphs of $\mathbb{Z}^d$. 
This slab consists of $\mathbb{F}$ and its integer translations $ t + \mathbb{F}  \subset \mathbb{Z}^d $ where $ t = (0,0,t_3, \ldots , t_d)  $ with $t_i \in \{  -(R_u - 1), \ldots  , R_u - 1 \}$ for all $3\leq i\leq d$.

We first need an estimate on the escape probability of a random walk from $S^{R_u}$.  Because $d \geq 5$, a random walk on $\mathbb{Z}^{d-2}$ is transient and thus it suffices to verify escape from $0$ in the last $(d-2)$-coordinates by comparing with a continuous time random walk (see~\cite{Barlow2017} for further background). Hence we have the following bound.

\begin{lemma} \label{lemma: prob F-escape}
	Let $\alpha > 1$ and let $R_u$ as in~\eqref{eq: R upper bound}.
	Let $z \in \mathbb{F}^{R_u}$ and consider a simple random walk $X$ starting at  $x \in  \partial_i B_z  (R_u) \cap \partial_i S^{R_u}$. Then there exists a constant $\newet\label{et:1}=\et{et:1}(d)>0$ such that
	\begin{equation}\label{eq:escapeprob}
		\inf_{ 0 < u < 1 } \prw_x \left( H_{ S^{R_u} \setminus B_z (R_u)}  = \infty  \right) \geq \inf_{ 0 < u < 1 } \prw_x \left( H_{ S^{R_u}}  = \infty  \right) \geq\et{et:1} 
		.
	\end{equation}
\end{lemma}

We will also make use of a conditional form of Lemma \ref{lemma: prob F-escape}, which follows immediately from \eqref{eq:escapeprob}.
\begin{cor}\label{cor:escapecoditioned}
Let $\alpha > 1$ and let $R_u$ as in~\eqref{eq: R upper bound}.
	Let $z \in \mathbb{F}^{R_u}$ and consider a simple random walk $X$ starting at  $x \in  \partial_i B_z  (R_u) \cap \partial_i S^{R_u}$. Then
		\begin{equation}\label{eq:escapeprobcinditioned}
		\inf_{ 0 < u < 1 } \prw_x \left( H_{ S^{R_u} \setminus B_z (R_u)}  = \infty  \Big|H_{ S^{R_u}}  = \infty \right) \geq\et{et:1} 
		.
	\end{equation}
\end{cor}

For each vertex $z \in \mathbb{F}^{R_u}$, consider the collection of trajectories 
\begin{equation} \label{eq: trajectories z}
	\left\{ w^* \in \supp (\RI^{u}) \st  H_{ B_z(R_u) } (w^*) < \infty \right\}.
\end{equation}
Each trajectory in the set above defines a path
$\tilde{X} = w^* \vert_{\ge H_{ B_z(R_u) } }$, where the initial point of $\tilde{X}$ is the first hitting time of $w^*$ in $ B_z(R_u) $. The collection of paths defined in~\eqref{eq: trajectories z} corresponds to $\supp \vec{\RI}^u_{ B_z(R_u)}$, defined in~\eqref{eq: measure paths at K}.
The path $\tilde{X} \in \supp \vec{\RI}^u_{ B_z(R_u)}$  has law $ \prw_{\tilde{e}_{B_z(R_u)}} $, which corresponds to a simple random walk with initial distribution $\tilde{e}_{B_z(R_u)}$ (see~\eqref{eq: normalized equilibrium measure} for the definition of normalized equilibrium measure).

\begin{definition} \label{def: good vertices}
	For $d \geq 5$, fix $\gamma,M>0$ such that 
	\begin{enumerate}
		\item $\gamma<\gamma_d$, where $ \gamma_d$ is as in Theorem~\ref{thm: capacity_lln}.
		\item $\prw (  \range (\tilde{ X} [0, \frac{1}{2}R_u^2] )  \subset B_z(M R_u)) > 1 - \et{et:1}/2$, where $\eta_2$ is as in Lemma \ref{lemma: prob F-escape}. 
	\end{enumerate}
	Note that such a choice for $M$ exists due to Proposition \ref{prop: RW escape}.
	
	We say that a {finite path $\tilde{X}[0,R_u^2] $} is a \emph{good random walk} based at $z \in \mathbb{F}^{R_u}$ (or more precisely $(\gamma,M)$-good), if it satisfies the following conditions (see Figure~\ref{fig: good boxes} for an illustration): 
	\begin{enumerate}[label=(\roman*)]
		\item \label{item: good rw 0}  $\tilde{X}=\vec{s}^{\, *}_{B_z(R_u)}(w^{*})\in \supp \vec{\RI}^u_{ B_z(R_u)}$ for some $w^*\in{\RI}^u_{ B_z(R_u)}$.
		\item \label{item: good rw hit} $H_{  S^{R_u}  \setminus { B_z(R_u) }  } (w^*)>H_{ { B_z(R_u) }  } (w^*)$,
		\item \label{item: good rw 1} $\range (\tilde{ X} [0, \frac{1}{2}R_u^2] )  \subset B_z(M R_u)$,  
		\item \label{item: good rw 2} $\capa (\range ( \tilde{X} [0, \frac{1}{2}R_u^2] ) \geq \frac{\gamma}{2} R_u^2$.
	\end{enumerate}
	We denote the collection of good random walks based at $z$ by $\mathcal{G}^{u}_z$.
	We say that $z \in \mathbb{F}^{R_u}$ is an \emph{$R_u$-good vertex} if the collection of good random walks $\mathcal{G}^{u}_{z}$ is non-empty. In this case we call $B_z(R_u)$ an $R_u$-good box. Otherwise, $z$ is called an \emph{$R_u$-bad vertex} and $B_z(R_u)$ is an $R_u$-bad box.
\end{definition}

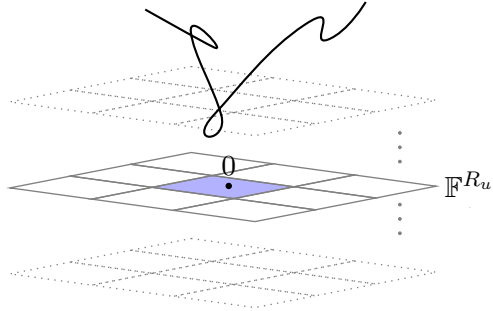
\begin{figure}[h]
	\centering
		\begin{tikzpicture} 
	  	\begin{axis}[view = {40}{10},
			  zmax = 12, 
			  zmin = -12,
			  ymax = 21, 
			  ymin = -21,
			  xmax = 18.5, 
			  xmin = -18.5,
			  xtick = {0},
			  ytick = {0},
			  ztick = {0}, 
			  hide axis, 
			  draw = gray, 
			  fill = gray
			  ]


				\addplot3[mark=*,black,point meta=explicit symbolic,nodes near coords, mark size=1pt] 
					coordinates {(0,1,0)[\small{$0$}]};

				\addplot3[mark=*,point meta=explicit symbolic,nodes near coords, mark size=0pt] 
					coordinates {(18.5,18.5,-1.5)[\small{$\mathbb{F}^{R_u}$}]};
			
				\addplot3 [
			    domain=-5:5,
			    domain y = -5:5,
			    samples = 2,
			    surf,
			    fill = blue!30, faceted color = blue!30] {0};

				\addplot3 [
			    domain=-16:16,
			    domain y = -16:16,
			    samples = 4,
			    mesh,
			    draw = gray,
			    line width = .5pt] {0};

				\addplot3 [
			    domain=-16:16,
			    domain y = -16:16,
			    samples = 4,
			    mesh,
			    draw = black!50!white,
			    line width = 0.5pt,dotted] {6};

				\addplot3 [
			    domain=-16:16,
			    domain y = -16:16,
			    samples = 4,
			    mesh,
			    draw = black!50!white,
			    line width = 0.5pt,dotted] {-6};

				\addplot3 [only marks, mark size=.5pt]table 
				{
					x y     z
					12 15   3.5
					12 15   2.5
					12 15   1.5
					12 15   -3.5
					12 15   -2.5
					12 15   -1.5
				};
			\end{axis} 

			\draw [black,thick] plot [smooth, tension=1.2] 
			coordinates {
				(5.3,5.3) (4.8,4.8) (4.5,5)   
				(3.3,3.8) (3.4,3.8) 
				(3,4.9)  (3.3,4.7)  (2.4,5.2) 
			};

		\end{tikzpicture}

	\caption{ 
	The vertex $0 \in \mathbb{F}^{R_u}$ is $R_u$-good if a trajectory in the interlacements hits $B_z(R_u)$, avoids $ S^{R_u} \setminus B_z(R_u) $ before hitting $B_z(R_u)$, and has a range of typical length and capacity 
	(see conditions~\ref{item: good rw 0}, ~\ref{item: good rw hit},~\ref{item: good rw 1} and~\ref{item: good rw 2} in Definition~\ref{def: good vertices}, respectively). 
	This figure represents a good vertex with a coloured region around it.  
	}
	\label{fig: good boxes}
\end{figure}

Our next step is a lower bound on the probability that a box is good.

\begin{prop} \label{prop: good high prob}
	There exists $R^* > 0$ and $\newet\label{et:2}=\et{et:2}(d)>0$ such that the following holds whenever $R_u > R^*$ and $z \in \mathbb{F}^{R_u}$. Then 
	$ \vert  \mathcal{G}^{u}_{z} \vert \geq_{\ST} \mathrm{Pois}(\et{et:2} \cdot u \cdot {R^{d-2}_u})$.
	In particular, 
	\[
		\prob^u(\text{The box }B_z(2R_u)\text{ is good})\geq 1-	e^{-\et{et:2}u R_u^{d-2}} .
	\]
\end{prop}

\begin{proof}
	We consider an event based on conditions~\ref{item: good rw 0},~\ref{item: good rw hit},~\ref{item: good rw 1} and~\ref{item: good rw 2} in Definition~\ref{def: good vertices}. 
	For a doubly-infinite simple random walk $\tilde{X}$ starting at  $x \in  \partial_i B_z  (R_u) \cap \partial_i S^{R_u}$ (with the future being a simple random walk and the past being an independent simple random walk conditioned to avoid $B_z(R_u)$), let 
 	\[
 			A_z (\tilde{X}) = 
 			\left\{  
 			\begin{array}{c}
 			\range (\tilde{ X} [0, \frac{1}{2}R_u^2] )  \subset B_z(MR_u), \\ \capa (\range ( \tilde{X} [0, \frac{1}{2}R_u^2] ) \geq \frac{\gamma}{2} R_u^2, \\ \tilde{X}(-\infty,0)\cap S^{R_u} \setminus  B_z(R_u) =\emptyset
 			\end{array}
 			\right\}.
 	\]
 	By Theorem~\ref{thm: capacity_lln}, Proposition \ref{prop: RW escape} and Corollary \ref{cor:escapecoditioned}, there exists $R^* > 0$ and an $\newet\label{et:3}=\et{et:3}(d,M,\gamma)>0$ such that whenever $R_u > R^*$ 
 	\[
 		\prw_x\otimes\prw_x \left(A_z (\tilde{X})\Big|\tilde{X}(-\infty,0)\cap B_z(R_u)=\emptyset\right) > \et{et:3}.
 	\]
 	Recall that $\vert \mathcal{G}^{u}_{z} \vert $ is  the number of good random walks based at $z$, and that $$ \vert \supp\vec{\RI}^u_{B_z(R_u)}\vert  = N_{B_z(R_u)}  $$ has Poisson distribution with parameter $ u \cdot \capa (B_z(R_u))$.
  The Poisson thinning property implies that $\vert \mathcal{G}^{u}_{z} \vert $ dominates a Poisson distribution with parameter $ \et{et:3} \cdot u \cdot \capa (B_z(R_u))$.
	Proposition~\ref{prop:capa_ball} implies that
	\[
		\et{et:3} \cdot u \cdot \capa (B_z (R_u)) \geq  \et{et:3} \cdot u  \cdot \left( \cnt{c:capL}  \cdot R_u^{d-2} \right) . 
	\]
	Let  $\et{et:2}=\et{et:3}\cdot \cnt{c:capL} > 0$, and note that it only depends on $d$. 
	Therefore,  Proposition~\ref{prop: Poisson ST} implies that $\vert \mathcal{G}^{u}_{z} \vert$ stochastically dominates a Poisson random variable with parameter $\et{et:2} \cdot u \cdot R_u^{d-2} $.
\end{proof}

The basic estimate, that will allow us to construct the path between 0 and $ne_1$, is the following.  
\begin{lemma}\label{lem:blackboxbound}
	Let $R^*>0$ be as in Proposition~\ref{prop: good high prob} and assume that $R_u>R^* $.
	There exists a constant $ \newet \label{et:4} = \et{et:4} (d,M,\gamma) > 0$ satisfying the following. 
	If $z \in \mathbb{F}^{R_u}$ is an {$R_u$-good vertex} with an arbitrarily chosen good random walk $\hat{X}_z$, and $(z,y) \in \mathbb{E}^{R_u}$, then
	\[
		\prob^u \big(\text{There exists a good random walk based at } y \text{ intersecting } \hat{X}_z \big) \geq 1 -  e^{ - \et{et:4} \cdot \alpha^2}.
	\]
\end{lemma}

\begin{proof}
	By definition, the collections of good random walks $\mathcal{G}^{u}_z$ and $\mathcal{G}^{u}_y$ are independent when $y \neq z$. 
 	Therefore, we get that the number of good random walks based at $z$, namely $\vert \mathcal{G}_z^{u} \vert$, and the number of good random walks based at $y$, namely $\vert \mathcal{G}_y^{u} \vert$, are independent random variables.

 	Let $\tilde{\prw}_{y}$ be the law of a good random walk $\tilde{X}_y$ based at $y$. Note that for any good random walk $\tilde{X}_y$, the law of $\tilde{X}_y[\frac{1}{2}R_u^2,\infty)$ is that of a simple random walk $\prw_{\tilde{X}(\frac{1}{2}R_u^2)}$. Furthermore since $\tilde{X}$ is good random walk based at $y$, we know that $\tilde{X}(\frac{1}{2}R_u^2)\in B_y(MR_u)$.
 	Proposition~\ref{prop: RW and set} and Condition~\ref{item: good rw 2} in Definition~\ref{def: good vertices} implies that for all the good random walks based at $y$ in $\mathcal{G}_y^{u }$: 
 	\begin{align} \label{eq: intersection bond perc} 
 		\tilde{\prw}_{y}\left(  \range \left( \hat{X}_z \left[0, \frac{1}{2}R_u^2\right] \right)\cap \range \left( \tilde{X}_y \left[\frac{1}{2}R_u^2, R_u^2\right] \right)\neq \emptyset  \right) 
 		&\geq \frac{\et{et:rws} \capa (\range ( \hat{X}_z [0, \frac{1}{2}R_u^2] ))}{R_u^{d-2}} \\
 		 &\geq \frac{\et{et:rws} \gamma}{2R_u^{d-4}}.
 	\end{align}
 	By the Poisson thinning property, Proposition~\ref{prop: good high prob} and Proposition~\ref{prop: Poisson ST}, the number of random walks in  $\mathcal{G}_y^{u}$ intersecting $\hat{X}_z [0, R_u^2]$ within the first $R_u^2$ steps, stochastically dominates a Poisson random variable with parameter $ \et{et:2} \cdot u \cdot \et{et:rws} \cdot \frac{\gamma}{2} \cdot R_u^2  = \et{et:4} \cdot \alpha^2 $, where $\et{et:4}=\et{et:2} \cdot \et{et:rws} \cdot \frac{\gamma}{2}$.
 	In particular, the probability that there exists a random walk in  $\mathcal{G}_y^{u}$ intersecting $\hat{X}_z [0, R_u^2]$ within the first $R_u^2$ steps is at least  $1 -  \exp \left( - \et{et:4} \cdot \alpha^2 \right) $.
\end{proof}

Next, we define a path in the renormalized $*$-connected square grid $(\mathbb{F}^{R_u}, \mathbb{E}^{R_u})$ using an exploration process. For a fixed $n\in\mathbb{N}$, we wish to construct a path between $0$ and $z=(z_1, 0,0,\ldots, 0)\in\mathbb{F}^{R_u}$ satisfying $z_1\geq 2nR_u$. In each step of the exploration process we attempt to connect the current box with a new box by verifying that a good random walk in the new box is connected to a previously specified good random walk in the current box. Each successful exploration step on the renormalised lattice will be denoted by an arrow 
\begin{tikzpicture}
	\draw [-latex,thick,white](0,0) -- (0.75,0);\draw [-latex,thick,blue](0,0.1) -- (0.75,0.1);
\end{tikzpicture}
between the centres of the boxes. These arrows define a path on the renormalized graph $\mathbb{F}^{R_u}$, to which we always refer as an  \emph{arrow path}. A direction in which the exploration failed to create a connection are marked with a {\large $\textcolor{red}{\times}$}. In this case, we also colour the box we failed to connect to by painting it in black 
\begin{tikzpicture} 
	\draw [fill=black!20] (0,0)  rectangle (.4,.4); 
\end{tikzpicture}. 
The following algorithm explores the boundary of the \emph{black box region} (see Figures~\ref{fig:pathalgo} and~\ref{fig:restart} for examples of realizations of the process).  

Denote by $\mathcal{E}=\{\pm e_1,~\pm e_2,~\pm (e_1\pm e_2)\}$, the set of $*$-directions in the graph $(\mathbb{F}^{R_u},\mathbb{E}^{R_u})$ and by $\ell_0=\{(x, 0,0,\ldots, 0):x\in2R_u\mathbb{Z}\} \subset \mathbb{F}^{R_u}$. Besides black boxes, along the algorithm we declare boxes along the exploration process live or dead. \emph{Live boxes} are eventually used to construct the final arrow path whereas \emph{dead boxes} are explored but are not used eventually.

\begin{figure}[H]
\begin{tikzpicture}[scale=1.5]
\usetikzlibrary{patterns}
	
	\draw [fill=black, opacity=.2] (0,-0.5) rectangle (8.5,0);

	\draw[step=.5cm,black,thin] (0,-0.5) grid (8.5,3);

	\node at (0.25,.25)[circle,fill=black,inner sep=1pt]{};
	\node at (0.25,.10) (c) {\small $0$};
	\draw [-latex,thick,blue](0.25,0.25) -- (0.75,0.25);
	\draw [-latex,thick,blue](0.75,0.25) -- (1.25,0.25);
	\draw [-latex,thick,blue, dashed](1.25,0.25) -- (1.75,0.25); 
	
	\draw [fill=black, opacity=.2] (2,0) rectangle (2.5,0.5); 
	\draw [fill=black, opacity=.2] (2,0.5) rectangle (2.5,1); 
	\draw [fill=black, opacity=.2] (1.5,0.5) rectangle (2,1); 
	\draw [fill=black, opacity=.2] (1,0.5) rectangle (1.5,1); 
	
	\draw [-latex,thick,blue](1.25,0.25) -- (0.75,0.75);
	\draw [-latex,thick,blue](0.75,0.75) -- (1.25,1.25);

	\draw [fill=black, opacity=.2] (1.5,1) rectangle (2,1.5);	
	\draw [fill=black, opacity=.2] (1.5,1.5) rectangle (2,2);

	\draw [-latex,thick,blue](1.25,1.25) -- (1.25,1.75);
	\draw [-latex,thick,blue](1.25,1.75) -- (1.75,2.25);

	\draw [fill=black, opacity=.2] (2,1.5) rectangle (2.5,2);

	\draw [-latex,thick,blue](1.75,2.25) -- (2.25,2.25);
	\draw [-latex,thick,blue](2.25,2.25) -- (2.75,1.75);
	\draw [-latex,thick,blue,dashed](2.75,1.75) -- (2.25,1.25); 

	\draw [fill=black, opacity=.2] (2.5,0.5) rectangle (3,1);
	\draw [fill=black, opacity=.2] (2.5,1) rectangle (3,1.5);

	\draw [-latex,thick,blue](2.75,1.75) -- (3.25,1.25); 
	
	\draw [fill=black, opacity=.2] (3,1) rectangle (3.5,0.5);

	\draw [-latex,thick,blue](3.25,1.25) -- (3.75,0.75); 
	
	\draw [fill=black, opacity=.2] (3,0.5) rectangle (3.5,0);

	\node[red] at (2,0.25)  (c)     {\large $\times$};
	\node[red] at (2,0.5)  (c)     {\large $\times$};
	\node[red] at (1.75,0.5)  (c)     {\large $\times$};
	\node[red] at (1.5,0.5)  (c)     {\large $\times$};
	\node[red] at (1.5,1.25)  (c)     {\large $\times$};
	\node[red] at (1.5,1.5)  (c)     {\large $\times$};
	\node[red] at (2,2)  (c)     {\large $\times$};
	\node[red] at (2.5,1)  (c)     {\large $\times$};
	\node[red] at (2.5,1.25)  (c)     {\large $\times$};
	\node[red] at (3.25,1)  (c)     {\large $\times$};
	\node[red] at (3.25,0.5)  (c)     {\large $\times$};

	\draw [-latex,thick,blue,dashed](3.75,0.75) -- (3.75,0.25); 


	\draw [-latex,thick,blue,dashed](3.75,0.25) -- (4.25,0.25); 
	\draw [-latex,thick,blue,dashed](4.25,0.25) -- (4.75,0.25); 
	
	\draw [fill=black, opacity=.2] (5,0) rectangle (5.5,0.5);
	\draw [fill=black, opacity=.2] (5,0.5) rectangle (5.5,1);
	\draw [fill=black, opacity=.2] (4.5,0.5) rectangle (5,1);
	\draw [fill=black, opacity=.2] (4,0.5) rectangle (4.5,1);
	
	\node[red] at (5,0.25)  (c)     {\large $\times$};
	\node[red] at (5,0.5)  (c)     {\large $\times$};
	\node[red] at (4.75,0.5)  (c)     {\large $\times$};
	\node[red] at (4.5,0.5)  (c)     {\large $\times$};

	\draw [-latex,thick,blue](3.75,0.75) -- (4.25,1.25);
	\draw [-latex,thick,blue](4.25,1.25) -- (4.75,1.25);
	\draw [-latex,thick,blue](4.75,1.25) -- (5.25,1.25);
	\draw [-latex,thick,blue](5.25,1.25) -- (5.75,0.75);
	\draw [-latex,thick,blue](5.75,0.75) -- (5.75,0.25);
	\draw [-latex,thick,blue](5.75,0.25) -- (6.25,0.25);
	\draw [-latex,thick,blue](6.25,0.25) -- (6.75,0.25);

	\draw [fill=black, opacity=.2] (7,0) rectangle (7.5,0.5);
	\node[red] at (7,0.25)  (c)     {\large $\times$};

	\draw [-latex,thick,blue](6.75,0.25) -- (7.25,0.75);
	\draw [-latex,thick,blue](7.25,0.75) -- (7.75,0.25);
	\draw [-latex,thick,blue](7.75,0.25) -- (8.25,0.25);
	\node at (8.25,0.25)[circle,fill=black,inner sep=1pt]{};
	\node at (8.25,.10) (c) {\small $z$};
\end{tikzpicture}
\caption{Path finding algorithm. The lattice $\mathbb{F}^{R_u}$ is the lattice of centre of boxes in the figure. The path of solid arrows represent the final output of the algorithm. Boxes at the head of dashed arrows are dead.}
\label{fig:pathalgo}
\end{figure}
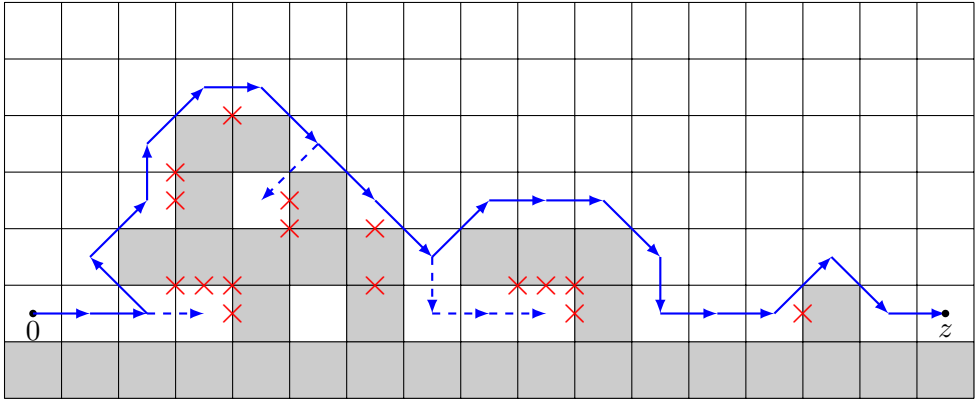

\begin{algorithm}\label{alg:exploration}
	\begin{itemize}$~$
	\item \emph{Initialization:} 
	\begin{enumerate}[label=(\arabic*)]
	\item \label{step:1}  Paint in black all boxes ${\bf{x}}=(x_1, x_2,0,\ldots, 0)\in\mathbb{F}^{R_u}$ satisfying $x_2=-2R_u$.
	\item Denote by ${\bf{w_0}}=(w_0,0,0,\ldots,0)\in\ell_0$ the box with the largest $w_0\leq 0$ such that ${\bf{w_0}}$ is a good box.
	\item Choose an arbitrary good (random walk) path $X_{{\bf{w}_0}}$ in the box ${\bf{w_0}}$, declare ${\bf{w}_0}$ to be the current box as well as a live box, and colour all boxes in $\ell_0$ between $w_0$ (not including) and $0$ (including) black. 
	\end{enumerate}
	\item \emph{Iteration:}
	\begin{enumerate}[label=(\arabic*)]
		\setcounter{enumi}{2}
		\item If the current box is of the form $(z,0,0,\ldots,0)\in\ell_0$ for some $z\geq 2 R_u n$, jump to Output.
		\item Observe the current box ${\bf{v}}$. If ${\bf{v}}\in\ell_0$ declare the direction $e$ to be the first direction in the counter clockwise direction, starting from $e_1$, such that ${\bf{v}}+e$ is not a black box or a dead box. If ${\bf{v}}\notin \ell_0$, observe the last drawn arrow $e'$ and denote by $e\in\mathcal{E}$ the first direction in the counter clockwise direction, starting from $-e'$, such that ${\bf{v}}+e$ is not a black box or a dead box (this choice allows us to follow the boundary of the currently existing black-box cluster). If no such direction exist, define the box ${\bf{v}}$ to be dead and jump to Restart.
		\item If ${\bf{v}}+e$ is a live box, declare the box at ${\bf{v}}$ to be dead and redefine the current box to be ${\bf{v}}+e$ (this corresponds to backtracking along the arrow path).
		\item If ${\bf{v}}+e$ is not a live box (which means it was not explored until this point), check whether the box ${\bf{v}}+e$ contains a good random walk that intersects the chosen good random walk in the box ${\bf{v}}$. 
		\begin{enumerate}
			\item If such a good random walk exists, choose one of them (arbitrarily), draw an arrow from the centre of the box at ${\bf{v}}$ to the centre of the box at ${\bf{v}}+e$, declare the box at ${\bf{v}}+e$ to be a live box and redefine the current box to be ${\bf{v}}+e$. 
			\item If such a good random walk does not exist, denote {\large $\textcolor{red}{\times}$} on the edge shared by the boxes at ${\bf{v}}$ and ${\bf{v}}+e$ and define the box ${\bf{v}}+e$ to be black.
		\end{enumerate}	
					\item Return to (3).
	\end{enumerate}
	\item \emph{Restart:}
		\begin{enumerate}[label=(\arabic*)]
		\setcounter{enumi}{5}
		\item Define ${\bf{w}_0}=(w_0,0,0,\ldots,0)\in\ell_0$ with $w_0\in \mathbb{Z}$ the largest value such that all dead or live boxes along the line $\ell_0$ are strictly to the right of the box ${\bf{w}_0}$ and ${\bf{w}_0}$ is a good box. Define all boxes to the right of ${\bf{w}_0}$ and to the left of the left most explored box (dead or alive) to be black. Choose an arbitrary good random walk in the box ${\bf{w}_0}$ and declare the box ${\bf{w}_0}$ to be the current box as well as a live box. Return to (3). 
		\end{enumerate}
	\item \emph{Output:}
		Denote the (last) current box by ${\bf{w}_1}$. 
		Return the arrow path from ${\bf{w}_0}$ to ${\bf{w}_1}$  consisting of all live boxes and the arrows connecting them. 
	\end{itemize}
\end{algorithm}

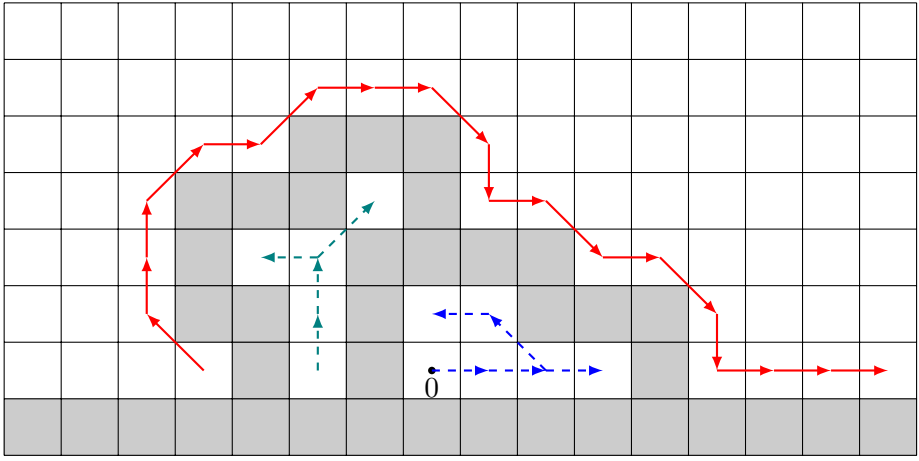
\begin{figure}[H] 
\begin{tikzpicture}[scale=1.5]
\usetikzlibrary{patterns}
\usetikzlibrary{math} 

	\tikzmath{ \t = 3.5;} 
	
	\draw [fill=black, opacity=.2] (0,-0.5) rectangle (8,0);

	\draw[step=.5cm,black,thin] (0,-0.5) grid (8,3.5);

	\node at (0.25 + \t ,.25)[circle,fill=black,inner sep=1pt]{};
	\node at (0.25+ \t ,.10) (c) {\small $0$};
	

	\draw [-latex,thick,blue,dashed]( 0.25 + \t , 0.25) -- ( 0.75 + \t , 0.25);
	\draw [-latex,thick,blue,dashed]( 0.75 + \t , 0.25) -- ( 1.25 + \t , 0.25);
	\draw [-latex,thick,blue,dashed]( 1.25 + \t , 0.25) -- ( 1.75 + \t , 0.25);
	\draw [-latex,thick,blue,dashed]( 1.25 + \t , 0.25) -- ( 0.75 + \t , 0.75);
	\draw [-latex,thick,blue,dashed]( 0.75 + \t , 0.75) -- ( 0.25 + \t , 0.75);
	
	\draw [fill=black, opacity=.2] ( 2.0 + \t , 0.0) rectangle ( 2.5 + \t , 0.5); 
	\draw [fill=black, opacity=.2] ( 2.0 + \t , 0.5) rectangle ( 2.5 + \t , 1.0); 
	\draw [fill=black, opacity=.2] ( 1.5 + \t , 0.5) rectangle ( 2.0 + \t , 1.0); 
	\draw [fill=black, opacity=.2] ( 1.0 + \t , 0.5) rectangle ( 1.5 + \t , 1.0); 
	
	\draw [fill=black, opacity=.2] ( 1.0 + \t , 1.0) rectangle ( 1.5 + \t , 1.5);
	\draw [fill=black, opacity=.2] ( 0.5 + \t , 1.0) rectangle ( 1.0 + \t , 1.5);
	\draw [fill=black, opacity=.2] ( 0.0 + \t , 1.0) rectangle ( 0.5 + \t , 1.5); 
	\draw [fill=black, opacity=.2] (-0.5 + \t , 1.0) rectangle ( 0.0 + \t , 1.5); 
	\draw [fill=black, opacity=.2] (-0.5 + \t , 0.5) rectangle ( 0.0 + \t , 1.0); 
	\draw [fill=black, opacity=.2] (-0.5 + \t , 0.0) rectangle ( 0.0 + \t , 0.5); 


	\draw [fill=black, opacity=.2] ( 0.0 + \t , 1.5) rectangle ( 0.5 + \t , 2.0);
	\draw [fill=black, opacity=.2] ( 0.0 + \t , 2.0) rectangle ( 0.5 + \t , 2.5);
	\draw [fill=black, opacity=.2] (-0.5 + \t , 2.0) rectangle ( 0.0 + \t , 2.5); 
	\draw [fill=black, opacity=.2] (-1.0 + \t , 2.0) rectangle (-0.5 + \t , 2.5); 

	\draw [-latex,thick,teal,dashed](-0.75 + \t , 0.25) -- (-0.75 + \t , 0.75);
	\draw [-latex,thick,teal,dashed](-0.75 + \t , 0.75) -- (-0.75 + \t , 1.25);
	\draw [-latex,thick,teal,dashed](-0.75 + \t , 1.25) -- (-0.25 + \t , 1.75);
	\draw [-latex,thick,teal,dashed](-0.75 + \t , 1.25) -- (-1.25 + \t , 1.25);


	\draw [fill=black, opacity=.2] (-1.0 + \t , 1.5) rectangle (-0.5 + \t , 2.0); 
	\draw [fill=black, opacity=.2] (-1.5 + \t , 1.5) rectangle (-1.0 + \t , 2.0); 
	\draw [fill=black, opacity=.2] (-2.0 + \t , 1.5) rectangle (-1.5 + \t , 2.0); 
	\draw [fill=black, opacity=.2] (-2.0 + \t , 1.0) rectangle (-1.5 + \t , 1.5); 
	\draw [fill=black, opacity=.2] (-2.0 + \t , 0.5) rectangle (-1.5 + \t , 1.0);
	\draw [fill=black, opacity=.2] (-1.5 + \t , 0.5) rectangle (-1.0 + \t , 1.0); 
	\draw [fill=black, opacity=.2] (-1.5 + \t , 0.0) rectangle (-1.0 + \t , 0.5); 

	\draw [-latex,thick,red](-1.75 + \t , 0.25) -- (-2.25 + \t , 0.75);
	\draw [-latex,thick,red](-2.25 + \t , 0.75) -- (-2.25 + \t , 1.25);
	\draw [-latex,thick,red](-2.25 + \t , 1.25) -- (-2.25 + \t , 1.75);
	\draw [-latex,thick,red](-2.25 + \t , 1.75) -- (-1.75 + \t , 2.25);
	\draw [-latex,thick,red](-1.75 + \t , 2.25) -- (-1.25 + \t , 2.25);
	\draw [-latex,thick,red](-1.25 + \t , 2.25) -- (-0.75 + \t , 2.75);
	\draw [-latex,thick,red](-0.75 + \t , 2.75) -- (-0.25 + \t , 2.75);
	\draw [-latex,thick,red](-0.25 + \t , 2.75) -- ( 0.25 + \t , 2.75);
	\draw [-latex,thick,red]( 0.25 + \t , 2.75) -- ( 0.75 + \t , 2.25);
	\draw [-latex,thick,red]( 0.75 + \t , 2.25) -- ( 0.75 + \t , 1.75);
	\draw [-latex,thick,red]( 0.75 + \t , 1.75) -- ( 1.25 + \t , 1.75);
	\draw [-latex,thick,red]( 1.25 + \t , 1.75) -- ( 1.75 + \t , 1.25);
	\draw [-latex,thick,red]( 1.75 + \t , 1.25) -- ( 2.25 + \t , 1.25);
	\draw [-latex,thick,red]( 2.25 + \t , 1.25) -- ( 2.75 + \t , 0.75);
	\draw [-latex,thick,red]( 2.75 + \t , 0.75) -- ( 2.75 + \t , 0.25);

	\draw [-latex,thick,red]( 2.75 + \t , 0.25) -- ( 3.25 + \t , 0.25);
	\draw [-latex,thick,red]( 3.25 + \t , 0.25) -- ( 3.75 + \t , 0.25);
	\draw [-latex,thick,red]( 3.75 + \t , 0.25) -- ( 4.25 + \t , 0.25);

\end{tikzpicture}
\caption{Path finding algorithm. The lattice $\mathbb{F}^{R_u}$ is the lattice of centre of boxes in the figure. Different colours represent different iterations of the algorithm.}
\label{fig:restart}
\end{figure}

Next we wish to bound the cardinality of the black region. 
We start by introducing some notation. 
Denote by $\mathcal{BR}^n=\mathcal{BR}^n(\alpha,u)$ the set of black boxes discovered in Algorithm~\ref{alg:exploration} which are above or on $\ell_0$ (for input $n$). 

\begin{lemma}\label{lem:blackregionbound}
	There is an $\alpha>0$ such that for any $n\in\mathbb{N}$ and every $u\in (0,1)$ 
	$$\prob^u\left(|\mathcal{BR}^n(\alpha,u)|<8n\right)>1/2.$$
\end{lemma}
\begin{proof}
The argument follows by a Peierls argument and a first moment estimate.

Let $l$ be the random number of black box $*$-clusters in $\mathcal{BR}^n$ and enumerate them as $\mathcal{BR}_1,\ldots\mathcal{BR}_l$ according to the order of their appearance along the exploration process.

First we claim  that we can choose $\alpha > 0$ large enough, such that for all $1\le i\le l$, 
\begin{equation}\label{eq:expecbound}
\mathbb{E}[|\mathcal{BR}_i|]\le 2.
\end{equation} 
Since a black-box cluster is a lattice animal in $\mathbb{F}^{2R_u}$, it follows from \cite[(4.24)]{grimmett1999percolation}, that the number of possible animals of size $k$ is smaller than $49^k$.
A union bound and Proposition \ref{prop: good high prob} and Lemma~\ref{lem:blackboxbound} gives
\[ 
	\prob^u(|\mathcal{BR}_i|>k)\le  49^{k}\cdot (e^{-\et{et:4} \cdot k \cdot \alpha^2}+e^{-\eta_3uR_u^{d-2}k})\leq (49[e^{-\eta_5\alpha^2}+e^{-\eta_3 \alpha^{d-2}}])^k
.\]
For an $\alpha>0$ large enough and $u\in (0,1)$ we can thus guarantee that $\mathbb{E}[|\mathcal{BR}_i|]\le 2$. Since each connected component (perhaps except for the last one) must contain a box in $\ell_0$ whose first coordinate is between $0$ and $2nR_u$, it follows that $l\le n+1$.
Now by Markov's inequality and \eqref{eq:expecbound}
\begin{equation}
\begin{aligned}
\prob^u\left(|\mathcal{BR}^n|>8n\right)&\le \frac{\mathbb{E}[|\mathcal{BR}^n|]}{8n}\le\frac{\mathbb{E}\Big[\mathbb{E}\Big[|\mathcal{BR}^n|\Big|l\Big]\Big]}{8n}\\
&=\frac{\mathbb{E}\Big[\sum_{i=1}^l\mathbb{E}\Big[|\mathcal{BR}_i|\Big]\Big]}{8n}=\frac{2 \mathbb{E}[l]}{8n}\le \frac{1}{2}
\end{aligned}
.\end{equation}
\end{proof}

\subsection{Proof of Proposition~\ref{prop: upper bound}} \label{subsec: proof of upper bound - case 1}


We apply Algorithm \ref{alg:exploration} with input $n \geq 1$. The boxes along the path of arrows between ${\bf{w}_0}$ and ${\bf{w}_1}$  in the output are on the outer $*$-boundary of (see Figure \ref{fig:pathalgo})
\[ 
	\mathcal{BR}^n\cup\{(x_1, x_2,0,\ldots, 0)\in\mathbb{F}^{R_u}, 0\le x_1\le 2nR_u, x_2=-2R_u\}.
\] 
Since every black box has at most $8$ $*$-neighbouring boxes, it follows from Lemma \ref{lem:blackregionbound} that the output path consists of at most $(64+1)n$ many $R_u$-good boxes with probability greater than $1/2$. Each such arrow represents a connection of length $R_u^2$ (by Condition~\ref{item: good rw 1} in Definition~\ref{def: good vertices}), thus we obtain that
\begin{equation}
\prob^u 
\left( 
\begin{array}{c} 
	\text{There exists } x \in B_{{\bf{w}_0}} (2R_u), \, y \in B_{{\bf{w}_1}} (2R_u)   \\
	\text{such that } d_u (x,y) \leq 65 n R_u^2
\end{array} 
\right) \geq \frac 1 2 .
\end{equation}
Next we bound the distance between $[0]_u$, $[ne_1]_u$ and $x$, $y$ satisfying the event above, respectively. 
To this end, we first bound the distance between  ${\bf{w}_0}$ and $0$, and between ${\bf{w}_1}$ and $n\cdot  (2R_ue_1)$.
Now if $x\neq B_{\bf{w}_0}(R_u)$, then by \eqref{eq:expecbound} and the fact that $x$ is on the boundary of a black-box component connected, Markov's inequality implies that
\begin{equation}
	\prob^u(|{\bf{w}_0}| \leq 16 \cdot 2R_u) \geq 1 - \prob(\vert \mathcal{BR}_1 \vert>32 R_u) >1- \frac{1}{8}.
\end{equation}
Similarly, for the endpoint of the output path we have
\begin{equation}
	\prob^u(|{\bf{w}_1} - n \cdot (2R_u e_1) | \leq 16 \cdot 2R_u) \geq 1 - \prob(\vert \mathcal{BR}_l \vert>32 R_u) >1- \frac{1}{8}.
\end{equation}
Additionally, for any $K > 0$
\[
	\prob^u \left( 
		\Big\vert [0]_u \Big\vert > K, \text{ or } 
		\Big\vert [n\cdot(2R_ue_1)]_u -   n\cdot(2R_ue_1) \Big\vert  > K 
		\right)
	\leq 2 e^{-u K^{d-2}}.
\]
Fix $K$ such that bound above is smaller than $1/100$.
Then, assuming the last three events hold 
\[
	\begin{split}
	[ 0 ]_u \cup B_{\bf{w}_0}(R_u) &\subset B_0 \left( \max \{ 34 R_u, K \} \right), \\
	[2R_u e_1]_u \cup B_{\bf{w}_1}(R_u) &\subset B_{2R_u e_1} \left( \max \{ 34 R_u, K \} \right).
	\end{split}
\]
Then, by the bound on the chemical distance in \cite[Theorem 1.3]{vcerny2012internal} there exists $C = C(u,d) > 0$ such that we have that
\[
	\prob^u \left( 
	\begin{array}{c}
		\text{There exists } x_1, x_2 \in B_0 \left( \max \{ 34 R_u, K \} \right) \cap \mathcal{I}^u \\
		d_u \left( x_1, x_2 \right) > C \max \{34 R_u, K  \}
	\end{array}
	 \right) <  \frac{1}{16}, 
\]
and
\[
	\prob^u \left( 
	\begin{array}{c}
		\text{There exists } y_1, y_2 \in B_{n(2R_ue_1)} \left( \max \{ 34 R_u, K \} \right) \cap \mathcal{I}^u \\
		d_u \left( y_1, y_2 \right) > C \max \{34 R_u, K  \}
	\end{array} 
	 \right)
	 < \frac{1}{16}.
\]
Then, all the bounds above imply that the intersection of these events has positive probability:
\[
	\prob^u \left( 
	\begin{array}{c}
	\text{There exists } x \in B_{\bf{w}_0} (2R_u)  \cap \mathcal{I}^u, \, y \in  B_{\bf{w}_1} (2R_u)  \cap \mathcal{I}^u \text { such that }  \\[3pt]
	 d_u \left( [0]_u, x \right) \leq C \max \{34 R_u, K  \}, 
	 \\[3pt]
	 d_u \left( y, [n(2R_ue_1)] \right) \leq C \max \{34 R_u, K \},   \\[3pt] 
	\text{ and } d_u \left( x , y \right) \leq 65 n R_u^2
	\end{array}
	\right) > c > 0. 
\]
Finally,  since
 \[ \rho_u(e_1) = \lim_{n \to \infty} \frac{1}{2nR_u} d_u \left( [0]_u, [n(2R_ue_1)]_u \right) \]
is an almost sure constant, the event above implies that  $\rho_u (e_1) \leq  33 R_u $ almost surely.
\qed


\section{A local lower bound on the chemical distance} \label{sec:local_lower_bound}

In preparation for the proof of Theorem~\ref{thm:main_lower}, we present a local bound on the chemical distance in random interlacements, in a carefully chosen scale.
We thus work with a box centred around $0$ and radius
\begin{equation} \label{eq: defLu}
	L_u := \frac{u^{4\varepsilon}}{\sqrt{u}}.
\end{equation}
We bound its $d_u$-distance to $B(2L_u)$.
For any two subsets $A,B\subset \mathbb{Z}^d$ 
\[
\hat{d}_{u} (A,B):=\inf\{d_{u} (x,y):x\in A\cap \mathcal{I}^{u}, y\in B\cap \mathcal{I}^{u}\}
.\]

\begin{prop}\label{prop:annulus}
	There exists a constant $ C_{\const{prop:annulus}{\ell} } = C_{\const{prop:annulus}{\ell}}(d) \in (0,\infty)$ satisfying the following. 
	For every $\varepsilon \in (0,1/24)$ there exists $ \newuBound \label{u:propAnnulus} = \uBound{u:propAnnulus} (\varepsilon) > 0 $ such that for every $u \in (0, \uBound{u:propAnnulus})$
	\[ \lim_{u \to 0} \prob^u
		\left( \hat{d}_u \left( \partial B (L_u), \partial B (2L_u) \right) 
		\leq C_{\const{prop:annulus}{\ell} }  \cdot u^{3\varepsilon}L_u^2\right) = 0.
	\]
\end{prop}

The next subsection establishes some properties necessary for the proof of Proposition~\ref{prop:annulus}. We complete the proof in Subsection~\ref{subsec: proof of prop annulus}. 

\subsection{Chemical distance in the ring} \label{subsec: chemDist ring}

Throughout this section, we refer to the set $ B (2L_u) \setminus B (L_u)  $ as the $L_u$-\emph{ring}.
Our goal here is to give a lower bound on the chemical distance in random interlacements between the two disjoint boundary components, $\partial B(2L_u)$ and $\partial B (L_u)$ of the $L_u$-ring.

In the low-intensity regime, the expected number of trajectories intersecting the $L_u$-ring is small, i.e., of order $uL_u^{d-2}$. These trajectories create clusters crossing from $\partial B(2L_u)$ to $\partial B (L_u)$. The argument in the proof of Proposition~\ref{prop:annulus} depends on a uniform lower bound on the chemical distance inside all of these clusters.
For random walks in dimension $d \geq 5$, we can control the number of intersections between trajectories and the number of self-intersections of each trajectory. Our argument depends on these high-dimensional properties and it is divided into two main steps.
\begin{enumerate}
	\item We partition the $L_u$-ring into smaller concentric regions, which we call annuli. We prove that for each cluster, there exists an annulus $A_u^*$ such that none of the trajectories associated to the cluster intersect within $A_u^*$.
	\item Inside $A_u^*$, we can find a sub-annulus, such that the chemical distance of all the paths is proportional to the number of steps required to cross the sub-annulus. We obtain this  by showing that, with positive probability, there is a positive density of global cut-points in all the paths. 
\end{enumerate}

We start by verifying certain properties for trajectories and clusters in the ring, these are established below in Sections \ref{subsubsec: control trajectories} and \ref{subsubsec: control clusters}, respectively.


\subsubsection{Trajectories in the $L_u$-ring} \label{subsubsec: control trajectories}
First we give an estimate on the last exit time of a simple random walk from a box. For a simple random walk $X = (X(n))_{n \geq 0}$ and a finite set $A\subset\mathbb{Z}^d$,  let
\begin{equation} \label{eq:lastexit}
	\tau_A \coloneqq \sup \left\{n>0:X(n)\in A \right\}.
\end{equation}

\begin{lemma}\label{lem:rangeboundann}
	Let $X$ be a simple random walk starting from $x \in \partial B(L_u)$.
	There exist $c',\newexpo \label{e:rangeboundann} >0$ such that for any $\beta>0$ and any $ u \in  (0,1)$
	\begin{equation}
	\prw_x \left( \tau_{B(2L_u)}>u^{-\beta }L_u^2\right)\le c'e^{- \expo{e:rangeboundann} u^{-\frac{\beta}{2} }}.
	\end{equation}
	\end{lemma}
\begin{proof}

Let $\sigma_{s}^1 \coloneqq H_{\partial B(4L_u)}$, $\sigma_{e}^1 \coloneqq \inf \left\{ t \geq \sigma_{s}^1 \st X (t) \in \partial B(2L_u) \right\}$, and for $i\geq 2$,
\[		
\sigma_{s}^i \coloneqq \inf \left\{ t \geq \sigma_{e}^{i-1} \st X (t) \in \partial B(4L_u) \right\}, 
		 \; \sigma_{e}^i \coloneqq \inf \left\{ t \geq \sigma_{s}^i \st X (t) \in \partial B(2L_u) \right\}.  
\]

Furthermore, define $Z'=\inf\{i\ge1: \sigma_{e}^i=\infty\}$. By Doob's invariance principle and the strong Markov property, there is a constant $\kappa\in (0,1)$ such that $Z'$ is dominated by a random variable $Z\sim \Geo (\kappa)$. 

By Proposition \ref{prop: RW escape} and random walk reversibility, for $u \in (0,1)$ and any $\beta > 0 $ we have, for every $i\geq 1$,
\[
\prw_x \left( \sigma_{e}^i-\sigma_{s}^{i}>u^{-\frac{\beta}{2} }L_u^2,  \; \sigma_{e}^i <\infty \right)
	\le 
	ce^{-\frac{1}{2}  \expo{e:RWescape} u^{-\frac{\beta}{2}}}    
\]
and
\[
	\prw_x \left(\sigma_{s}^i-\sigma_{e}^{i-1} > u^{-\frac{\beta}{2} }L_u^2 \Big| \sigma_{e}^{i-1}<\infty   \right)
	\le 
	ce^{-\frac{ 1 }{2} \expo{e:RWescape} u^{-\frac{\beta}{2}}}. 
\]

Thus by a union bound there exist $c',\expo{e:rangeboundann} = \expo{e:rangeboundann} (d) >0$ such that
\begin{align}
\prw_x (\tau_{B(2L_u)}&>u^{-\beta}L_u^2)\leq 
\prw_x \left(\sigma_{s}^{Z'}>u^{-\beta}L_u^2\right)
\leq \prw_x \left(\sigma_{s}^{Z}>u^{-\beta}L_u^2\right) \\
&\le \prw_x \left(\sum_{i=1}^{\lceil u^{-\frac{\beta}{2}}\rceil}(\sigma_s^i-\sigma_s^{i-1})>u^{-\beta}L_u^2\right)+\prw_x(Z>u^{-\frac{\beta}{2} })\\
& \leq \prw_x \Big(\exists i\in \{1,\ldots,\lceil u^{-\frac{\beta}{2}}\rceil\} \text{ s.t. }\sigma_s^i-\sigma_s^{i-1}>u^{-\frac{\beta}{2}} {L_u^2}\Big)+\prw_x (Z>u^{-\frac{\beta}{2}})\\
&\le c'e^{- \expo{e:rangeboundann} u^{-\frac{\beta}{2}}},
\end{align}
as required.
\end{proof}

Next, we wish to obtain control in high probability on the number of trajectories in $\RI^u$ intersecting the box $B(2L_u)$, and on the capacity of each of them. 
Recall that Equation~\eqref{eq: measure paths at K} defines this collection of paths as the point measure $\vec{\RI}^u_{B(2L_u)}$. In particular the number of paths in $\supp \vec{\RI}^u_{B(2L_u)}$ is the Poisson random variable $N_{B(2 L_u)}$ and each path starts at $\partial B(2L_u)$ according to the normalized equilibrium measure.
 
For $\varepsilon > 0 $ and $u > 0$, consider the events

\begin{equation} \label{eq:event_bound_paths}
	\mathfs{N}_u = \mathfs{N}_u (\varepsilon)  \coloneqq \left\{ N_{B(2 L_u)} < u^{1-\varepsilon}  L_u^{d-2} \right\} ,
\end{equation}

and
\begin{equation} \label{eq: event bound capacity}
 	\mathfs{C}_u = \mathfs{C}_u (\varepsilon) \coloneqq  	
	\left\{
	 \begin{array}{c} 
	 \text{ For every } w \in\supp\vec{\RI}^u_{B(2L_u)} \\[3pt]  \capa(\range(w)\cap B(2L_u))<u^{-4\varepsilon}L_u^2
 	 \end{array}
	\right\}.
\end{equation}

\begin{lemma} \label{lemma:numberbound}
	For $\varepsilon\in (0,{1}/{22})$ and $ u \in (0,1)$, consider the event $\mathfs{N}_u$ above and set
	$ \newbet \label{bet:eps} = \bet{bet:eps} (\varepsilon) =   \frac{1}{2} (d-4) -  \varepsilon (4d - 9)  > 0$. 
	There exists a constant $c = c(d) > 0$ such that
	\begin{equation}\label{eq:numberbound}
		\prob^u(\mathfs{N}_u)  = \prob^u \left(N_{B(2L_u)} < u^{-\bet{bet:eps}} \right) 
		\geq 1 -  cu^{\varepsilon u^{- \bet{bet:eps} }/2}.
	\end{equation}
\end{lemma}

\begin{proof}
  We first note that the choice of the exponent $\bet{bet:eps}$ satisfies
  $ u^{1-\varepsilon}  L_u^{d-2} =  u^{-\bet{bet:eps}}$.
  Moreover, the assumption $ \varepsilon < {1}/{22} $ implies $\bet{bet:eps} > 0 $, provided $d\ge 5$.

	Recall that $N_{B(2L_u)}\sim\text{Poisson}\left(u \cdot \capa(B(2L_u)) \right)$, and by Proposition~\ref{prop:capa_ball}, we have that $ \capa (B(2L_u)) \leq c \cdot L_u^{d-2} $ with $c = 2^{d-2} \cnt{c:capU} > 0$. 
	Therefore, Proposition~\ref{prop: Poisson ST} implies that a Poisson random variable $N \sim \Poi (c \cdot u  L_u^{d-2})$ stochastically dominates $N_{B(2L_u)}$, and thus 
		\[
			\prob^u \left(\mathfs{N}_u^c\right) = \prob^u \left(N_{B(2 L_u)} \geq u^{1-\varepsilon}  L_u^{d-2}  \right) 
			\leq 	\prob^u \left(N \geq u^{1-\varepsilon}  L_u^{d-2} \right).
		\]
	
	  An application of the Poisson approximation to the sum of Bernoulli and the Chernoff inequality (see e.g.~\cite[Section 2.3]{Vershynin2018}) gives
	 	 \[
	 	 	\prob^u \left( N \geq u^{1-\varepsilon}  L_u^{d-2}  \right) 
	 	 	\leq 
	 	 		e^{- (c \cdot u  L_u^{d-2})} 
	 	 		\left( \frac{c  e \cdot  u  L_u^{d-2}}{u^{1-\varepsilon}  L_u^{d-2}} \right)^{u^{1-\varepsilon}  L_u^{d-2}}
	 	 	\preceq
	 	 		u^{\varepsilon u^{-\bet{bet:eps}}/2}  ,
	 	 \]
	 	 proving the lemma.
\end{proof}

We now show that, with high probability, all the paths in $\vec{\RI}^u_{B(2L_u)}$ have bounded capacity.

\begin{lemma}\label{lemma:unioncapa}
	For every $\varepsilon\in (0,1/24)$ there exists $ \newuBound \label{u:proof1} = \uBound{u:proof1} (\varepsilon) > 0$ such that the following holds for all $ u \in (0, \uBound{u:proof1} ) $. 
	There exist constants $\newexpo \label{e:unioncapa} >0$ and $C > 0$  such that, for the event $\mathfs{C}_u$ above,
	$$
	\prob^u\left(\mathfs{C}_u\right)\ge 1-Ce^{- \expo{e:unioncapa}  u^{-\varepsilon}}.
	$$
\end{lemma}

\begin{proof}
	Let $X[0,\tau_{B(2L_u)}]$ be a path of a simple random walk starting at $x \in \partial B(L_u)$, where $\tau_{B(2L_u)}$  is the last exit time from $B(2L_u)$, as defined in~\eqref{eq:lastexit}. 
 	
	Let $\varepsilon \in (0,1/24)$ and assume that $u\in (0,(\gamma_d^+)^{-\frac{1}{2\varepsilon}})$, where $\gamma_d^+$ is the constant of Theorem~\ref{thm:cap_ldp}. With this choice $ u^{-2\varepsilon} > \gamma_d^+$ and thus from Lemma~\ref{lem:rangeboundann} and Theorem~\ref{thm:cap_ldp} we deduce 
	\begin{equation} \label{eq:unioncapa1}
		\begin{split}
			\prw_x &\left(
							\capa(X[0,\tau_{B(2L_u)}])>u^{-4 \varepsilon }L_u^2
						\right) \\
			&\le 
			\prw_x \left(
					\tau_{B(2L_u)}>u^{- 2\varepsilon }L_u^2
			\right)
			+\prw_x \left(
						\capa( X[0,u^{- 2\varepsilon }L_u^2]) > u^{-4\varepsilon } L_u^2
			\right) \\
			&\leq \prw_x \left(
					\tau_{B(2L_u)}>u^{- 2\varepsilon }L_u^2
			\right)+\prw_x \left(
						\capa( X[0,u^{- 2\varepsilon }L_u^2]) > \gamma_d^+ u^{-2\varepsilon} L_u^2
			\right) \\
			&\preceq  e^{- \expo{e:rangeboundann}  u^{ - \varepsilon } }+ e^{-  \expo{e:cap_ldp} u^{-2\varepsilon}L_u^2}
			\preceq   e^{- \tilde{\xi} u^{-\varepsilon}},
		\end{split}
	\end{equation}
	for  $\tilde{\xi} = \min \{  \expo{e:rangeboundann},  \expo{e:cap_ldp} \} $.
	
	We  now consider  the event $\mathfs{N}_u$ above, and recall $\bet{bet:eps} =  \frac{1}{2} (d-4) -  \varepsilon (4d - 9)$ from Lemma~\ref{lemma:numberbound}. 
	An application of Lemma~\ref{lemma:numberbound}, a union bound on the number of paths in $\vec{\RI}^u_{B(2L_u)}$ and~\eqref{eq:unioncapa1} yield:
	\begin{equation}
		\prob^u\left(\mathfs{C}_u^c \right) 
		\le \prob^u\left( \mathfs{N}_u^c \right)  + \prob^u\left(\mathfs{C}_u^c , \, \mathfs{N}_u \right)
		\preceq  u^{\varepsilon u^{-\bet{bet:eps}}/2} + u^{- \bet{bet:eps}}  e^{- \tilde{\xi} u^{-  \varepsilon } }.
	\end{equation}
	Fixing $0<\expo{e:unioncapa}<\tilde{\xi}=\min \{  \expo{e:rangeboundann},  \expo{e:cap_ldp} \}$ and since $\varepsilon \in (0,1/24)$ implies $\bet{bet:eps} > \varepsilon  $, one can find $\tilde{u} = \tilde{u} (\varepsilon) >0$ such that 
	\[
		u^{\varepsilon u^{-\bet{bet:eps}}/2} + u^{- \bet{bet:eps}}  e^{- \tilde{\xi} u^{-  \varepsilon } }\leq e^{- \expo{e:unioncapa} u^{-\varepsilon}},
	\]
	for all $u\in (0,\tilde{u})$, and thus the result follows with $\uBound{u:proof1}=\min\{\tilde{u},(\gamma_d^+)^{-\frac{1}{2\varepsilon}}\}>0$.
\end{proof}


\subsubsection{Clusters in the $L_u$-ring}  \label{subsubsec: control clusters}

We define an equivalence relation within the paths in $\supp \vec{\RI}^u_{B(2L_u)}$. Each equivalence class contains the paths in $\vec{\RI}^u_{B(2L_u)}$ corresponding to a connected component inside $ B(2L_u) $.

\begin{definition}\label{def:classes}
	We say that two paths $w_1,w_2\in\supp\vec{\RI}^u_{B(2L_u)}$ are equivalent if there exist paths $\eta_1,\ldots,\eta_m\in\supp\vec{\RI}^u_{B(2L_u)}$ for some $m\geq 1$ such that
	\[
		\mathcal{R}(w_1)\cap\mathcal{R}(\eta_1)\cap B(2L_u)\neq\emptyset, \qquad \mathcal{R}(w_2)\cap\mathcal{R}(\eta_m)\cap B(2L_u)\neq\emptyset,
	\]
	and for all $1\le i<m$ we have
	\[
		\mathcal{R}(\eta_i)\cap\mathcal{R}(\eta_{i+1})\cap B(2L_u)\neq\emptyset.
	\]
	We refer to these equivalence classes as clusters. Let $l$ be the random number of clusters, and denote them by $\mathcal{C}_1,\ldots,\mathcal{C}_l$.
\end{definition}

Our next goal is to control the number of paths in each of the clusters. We do this by controlling the capacity of the paths.

We consider a variation of $\vec{\RI}^u_{B(2L_u)}$ where the random walks satisfy the event $\mathfs{C}_u$ defined in~\eqref{eq: event bound capacity} above:
$$
\vec{\RI}^u(\mathfs{C}_u) \coloneqq \sum_{w \in\supp\vec{\RI}^u_{B(2L_u)}}\delta_w\ind{\{\capa( \range (w) \cap B(2L_u))<u^{-4\varepsilon}L_u^2\}}
.$$
Similarly to Definition \ref{def:classes}, we denote the equivalence classes of paths under the restricted point measure $\vec{\RI}^u(\mathfs{C}_u)$ as $\mathcal{C}_1(\mathfs{C}_u), \ldots , \mathcal{C}_{m}(\mathfs{C}_u)$, where $m$ denotes the random number of such clusters. 

\begin{lemma}\label{lem:dominatebybranching}
	Let $\varepsilon >0$ and $u>0$. The size of each of the clusters $\mathcal{C}_i(\mathfs{C}_u)$ is stochastically dominated by the total population size of a Galton-Watson branching process with mean $u^{4\varepsilon}$.
\end{lemma}

\begin{proof}
	Consider a path $w \in \supp\vec{\RI}(\mathfs{C}_u)$. 
	We proceed to construct sets of paths by induction. 
	Let $W_1$ be the paths in $\supp\vec{\RI}(\mathfs{C}_u)$ that intersect $w$ inside $B(2L_u)$. 
	Now assume that we have constructed $W_{i-1}$. Let $W_i$ be the paths in $\supp\vec{\RI}(\mathfs{C}_u)$ that intersect $W_{i-1}$ inside $B(2L_u)$ but avoid, inside $B(2L_u)$, the path $w$ and the paths in $W_1,\ldots,W_{i-2}$. By \cite[Lemma 7.2]{ProcacciaTykesson} and the fact that all paths in $\supp\vec{\RI}(\mathfs{C}_u)$ have capacity smaller than $u^{-4\varepsilon} L_u^2$, we obtain that for all $i$, $|W_i|$ is stochastically dominated by independent Poisson variables with mean $u^{1-4\varepsilon} L_u^2|W_{i-1}|= u^{4\varepsilon}|W_{i-1}|$.
\end{proof}

Next we state a well known result for subcritical branching processes.
\begin{lemma}[{\cite[Theorem 3.16]{VanDerHofstad2017}}] \label{lem:vanderhofshtad}
	Let $G_i$ be the $i$-th generation of a subcritical branching process with mean offspring distribution $\lambda\in (0,1)$. Then there exists  $\newexpo \label{e:vanderhofshtad}=\lambda-1-\log\lambda >0$ such that for any $k>1$
	\[ 
		P\left(\sum_{i=1}^\infty G_i>k\right)<e^{- \expo{e:vanderhofshtad} k}
	.\]
\end{lemma}

\begin{lemma}\label{lem:maxclustsizebound}
	Let $\varepsilon\in (0,1/24)$. There exists a constant $ \newuBound \label{u:proof2} = \uBound{u:proof2}(\varepsilon) >0$ such that for any $u \in (0 ,\uBound{u:proof2})$ and any $k>1$
	\[  
		\prob^u\left(\max \left\{\vert \mathcal{C}_1(\mathfs{C}_u) \vert,\ldots,\vert \mathcal{C}_m(\mathfs{C}_u) \vert \right\}>k\right)
		\preceq 
		u^{\varepsilon u^{-\bet{bet:eps}}/2} +   u^{-\bet{bet:eps}} e^{- \expo{e:vanderhofshtad} k},
	\]	
	where $\bet{bet:eps} > 0$ is as in Lemma~\ref{lemma:numberbound}. 

\end{lemma}

\begin{proof}
	First we can bound the number of equivalence classes $m$ with the total number of paths in  $\vec{\RI}^u(\mathfs{C}_u)$. 
	Recall that the later  is the random variable $N_{B(2L_u)}$, and from Lemma~\ref{lemma:numberbound} we have that
	$\prob^u(N_{B(2L_u)} \geq u^{1-\varepsilon}  L_u^{d-2} )\le cu^{\varepsilon u^{-\bet{bet:eps}}/2}$.
		The last estimate, together with a union bound, gives
		\begin{equation} \label{eq:maxclust1}
		\begin{split}
			&\prob^u \left( \max \left\{\vert \mathcal{C}_1(\mathfs{C}_u) \vert,\ldots,\vert \mathcal{C}_m(\mathfs{C}_u) \vert \right\}>k \right)\\
			&\le\prob^u\Big(m>u^{1-\varepsilon} L_u^{d-2} \Big) 
				+ \prob^u\Big( m<u^{1-\varepsilon} L_u^{d-2}, \, \max \left\{\vert \mathcal{C}_1(\mathfs{C}_u) \vert,\ldots,\vert \mathcal{C}_m(\mathfs{C}_u)  \vert \right\}>k \Big) \\
			&\preceq u^{\varepsilon u^{-\bet{bet:eps}}/2}+u^{1-\varepsilon} L_u^{d-2} \max_{1 \leq i \leq m}  \prob^u\left( \vert \mathcal{C}_i(\mathfs{C}_u) \vert>k \right) \\
			&= u^{\varepsilon u^{-\bet{bet:eps}}/2}+u^{- \bet{bet:eps}}  \max_{1 \leq i \leq m}  \prob^u\left( \vert \mathcal{C}_i(\mathfs{C}_u) \vert>k \right).
		\end{split}
		\end{equation}
		Lemmas \ref{lem:dominatebybranching} and \ref{lem:vanderhofshtad} imply tightness on the size of each cluster. Fixing $\uBound{u:proof2}(\varepsilon)=\left(\frac{1}{2}\right)^\frac{1}{4\varepsilon}$ gives for $\expo{e:vanderhofshtad} = \log(2)-1/2>0$:
		\[
			  \prob^u\left( \vert \mathcal{C}_i(\mathfs{C}_u) \vert>k \right) \leq  e^{- \expo{e:vanderhofshtad} k}, \qquad  \text{ for all }1 \leq i \leq m. 
		\]
		The bound above and~\eqref{eq:maxclust1} imply the the result.
\end{proof}

Lemma~\ref{lem:maxclustsizebound} bounds the number of clusters in the point measure $\vec{\RI}^u(\mathfs{C}_u)$. This variation of random interlacements is equal to the original process  $\vec{\RI}^u$ on the event $\mathfs{C}_u$. In particular, on  $\mathfs{C}_u$,  we have that $l = m $ and $(\mathcal{C}_i(\mathfs{C}_u))_{i=1}^m$ and $(\mathcal{C}_i)_{i=1}^l$ have the same distribution. 
Lemma~\ref{lemma:unioncapa} shows that the event  $ \mathfs{C}_u  $ has probability at least $ 1-ce^{- \expo{e:unioncapa}  u^{-\epsilon}} $ for  any $u \in (0 ,u_1)$,  and  we thus immediately get the next corollary.

\begin{cor}\label{cor:realclustbound}
For any $\varepsilon\in (0,1/24)$ there exists
$\newuBound \label{u:proof3} = \uBound{u:proof3}  (\varepsilon) = \min\{\uBound{u:proof1},\uBound{u:proof2}\}>0 $ 
such that for any $u \in (0 ,\uBound{u:proof3})$  and $ k > 1$
\[  
 	\prob^u\left(\max\{\vert \mathcal{C}_1 \vert,\ldots,\vert \mathcal{C}_l \vert \}>k\right) 
 	\preceq u^{\varepsilon u^{-\bet{bet:eps}}/2} + u^{-\bet{bet:eps}} e^{- \expo{e:vanderhofshtad} k} + e^{- \expo{e:unioncapa}  u^{-\varepsilon}}.
\]
\end{cor}

For each $1\le i \le l$, let $\mathcal{K}_i$ be the number of intersections between pairs of random walk paths in the cluster $\mathcal{C}_i$ (not including self-intersections). 

For $u > 0$ and $\varepsilon > 0 $ define the event
\begin{equation} \label{eq:maxintbound}
	\mathfs{K}_u = \mathfs{K}_u (\varepsilon)  \coloneqq \left\{ \max\{\mathcal{K}_1,\ldots,\mathcal{K}_l\} < u^{-\varepsilon/2} \right\} .
\end{equation}

\begin{lemma}\label{lem:maxintbound} 
For every $\varepsilon\in (0,1/24)$ there exist $ \newuBound \label{u:proof4} = \uBound{u:proof4} (\varepsilon) > 0 $ and $ c = c(d) > 0$ such that for any $u \in (0, \uBound{u:proof4})$ 
	\[  
		\prob^u \left( \mathfs{K}_u \right) \geq 1 - c u^{  \bet{bet:eps} }.
	\]
\end{lemma}

\begin{proof} 
	Let $\lambda > 0$ and $\newbet \label{bet:2} > 0$, to be chosen along the proof.
  Set \[ k(u)\coloneqq \log(u^{-\bet{bet:2}}),\] which we use below as a  bound on the size of all clusters.
  The number of clusters $l$ is always bounded by $N_{B(2L_u)}$, which by Lemma~\ref{lemma:numberbound} can be bounded by $N_{B(2L_u)}< u^{1-\varepsilon}L_u^{d-2} = u^{-\bet{bet:eps}}$, with probability greater than
  $1-cu^{\varepsilon u^{-\bet{bet:eps}}/2}$. Then
  \begin{equation}\label{eq:clustermaxbounding1}
  \begin{split}
  	\prob^u &\left(\max\{\mathcal{K}_1,\ldots,\mathcal{K}_l\}>\lambda k(u)^2 \right)  \\
  	&\preceq
  	u^{\varepsilon u^{-\bet{bet:eps}}/2} +  \prob^u \left(\max\{\mathcal{K}_1,\ldots,\mathcal{K}_l\}>\lambda k(u)^2, \, l <  u^{-\bet{bet:eps}}\right) .
  \end{split}
  \end{equation}
	Now, by the law of total probability followed by a union bound,
	\begin{equation}\label{eq:clustermaxbounding2}
	\begin{split}
		\prob^u &\left(\max\{\mathcal{K}_1,\ldots,\mathcal{K}_l\}>\lambda k(u)^2 , l <  u^{-\bet{bet:eps}}\right) \\
		&\le \prob^u\Big(\max\{|\mathcal{C}_1|,\ldots,|\mathcal{C}_l|\}>k(u)\Big)\\
		& +\prob^u\Big(\max\{\mathcal{K}_1,\ldots,\mathcal{K}_l\}>\lambda k(u)^2 , \,l <  u^{-\bet{bet:eps}} \Big| \max\{|\mathcal{C}_1|,\ldots,|\mathcal{C}_l|\}\le k(u) \Big) \\
		&\le \prob^u\Big(\max\{|\mathcal{C}_1|,\ldots,|\mathcal{C}_l|\}>k(u)\Big)\\
		& + u^{-\bet{bet:eps}}  \max_{1 \leq  i \leq l} \,\prob^u\Big( \mathcal{K}_i >\lambda k(u)^2  \Big| \max\{|\mathcal{C}_1|,\ldots,|\mathcal{C}_l|\}\le k(u) \Big) . 
	\end{split}
	\end{equation}
	
	Thus by~\eqref{eq:clustermaxbounding1},~\eqref{eq:clustermaxbounding2}, and applications of Corollary~\ref{cor:realclustbound} and Corollary~\ref{cor:numofintersections}, we get that for any $u \in (0,\uBound{u:proof3})$
	\begin{equation}
	\begin{split}
		\prob^u &\left( \max\{\mathcal{K}_1,\ldots,\mathcal{K}_l\}>\lambda k(u)^2 \right) \\
					& 
					\preceq u^{\varepsilon u^{-\bet{bet:eps}}/2} + u^{-\bet{bet:eps}}e^{-\expo{e:vanderhofshtad}k(u)} + e^{- \expo{e:unioncapa}  u^{-\varepsilon}}
					+ u^{-\bet{bet:eps}} k(u)^2e^{- \expo{e:intersections2rw} \lambda^{\frac{3}{5}}}.
	\end{split}
	\end{equation}
	Our choice of $k(u)$ yields 
	\begin{equation}
	\begin{split}
		\prob^u &\left( \max\{\mathcal{K}_1,\ldots,\mathcal{K}_l\}>\lambda k(u)^2 \right) \\
					& 
					\preceq u^{\varepsilon u^{-\bet{bet:eps}}/2} + u^{-\bet{bet:eps}} u^{ \expo{e:vanderhofshtad} \bet{bet:2} } + e^{- \expo{e:unioncapa}  u^{-\varepsilon}}
					+ u^{-\bet{bet:eps}} \left( \log\left( u^{-\bet{bet:2}} \right) \right)^2e^{- \expo{e:intersections2rw} \lambda^{\frac{3}{5}}}.
	\end{split}
	\end{equation}
	Choose $\bet{bet:2} = 2 \bet{bet:eps} /  \expo{e:vanderhofshtad} $ and $\lambda=u^{-\varepsilon/2}/k(u)^2$. Substituting our choice into the last bound, it follows for some $\uBound{u:proof4}  \in (0, \uBound{u:proof3} )$ small enough that for all $u\in (0,\uBound{u:proof4})$
	\begin{equation}
	\begin{split}
		\prob^u &\left(\max\{\mathcal{K}_1,\ldots,\mathcal{K}_l\}>u^{-\varepsilon/2} \right) \\
					&\preceq u^{\varepsilon u^{-\bet{bet:eps}}/2} + u^{\bet{bet:eps}}  + e^{- \expo{e:unioncapa}  u^{-\varepsilon}}
					+ u^{-\bet{bet:eps}} \left( \log\left( u^{-\bet{bet:2}} \right) \right)^2e^{- \expo{e:intersections2rw} (  \log (u^{-\bet{bet:2}}) )^{-6/5} u^{-3\varepsilon / 10}  } \\
					&\preceq u^{\bet{bet:eps}},
	\end{split}
	\end{equation}
	as required. 
\end{proof}

\subsection{Proof of~Proposition \ref{prop:annulus}}  \label{subsec: proof of prop annulus}

The results above set the necessary properties to obtain a probabilistic bound over the seed event of the multi-scale renormalization described below.

\begin{figure} 
\begin{tikzpicture}[scale=0.8]
	\coordinate (O) at (0,0);
	
	\draw[ultra thick, fill=white] (-3, -3) rectangle (3, 3) node [anchor=west]    {\tiny $ B(2L_u) $};
	
  \draw[fill=white ] (-5/2, -5/2) rectangle (5/2, 5/2);
 	\draw[black,  fill=blue!40] (-4/2, -4/2) rectangle (4/2, 4/2);
	\draw[thin, blue, fill=blue!20, dashed] (-11/6, -11/6) rectangle (11/6, 11/6);
	\draw[thin, blue, fill=blue!40, dashed] (-10/6, -10/6) rectangle (10/6, 10/6);

	\draw[black, fill=white] (-3/2, -3/2) rectangle (3/2, 3/2);

	\draw[black] (-3/2, -3/2) rectangle (3/2, 3/2);
	
	\draw[ultra thick, fill=white] (-1, -1) rectangle (1, 1) node[pos=.5] {\tiny $B(L_u)$};
	\node[blue] at (17/8, 9/4) (c)    {\tiny $ A^{j(i)}_u $};

\end{tikzpicture}
\caption{The region between the boxes $B(L_u)$ and $B(2L_u)$ is the $L_u$-ring. The $L_u$-ring is partitioned into disjoint (squared) annuli of equal width. The annulus $A^{j(i)}_u$ (coloured) is associated to the cluster $\mathcal{C}_i$, meaning that it satisfies Property~\ref{cond:C1}. We further divide $A^{j(i)}_u$ into three disjoint sub-annuli (indicated with different shades). 
For the paths in $\mathcal{C}_i$, their crossings of the middle sub-annuli satisfy Property~\ref{cond:C2}. \label{fig:annulipartition}}
\end{figure}
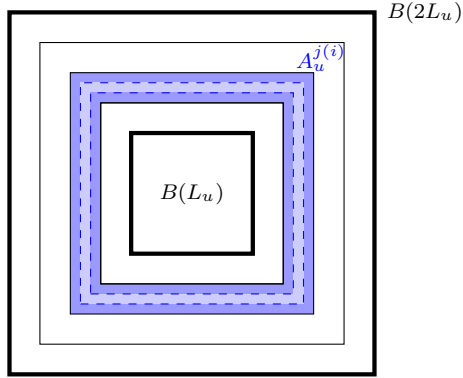

\begin{proof}[Proof of Proposition \ref{prop:annulus}]
	A \emph{crossing of the $L_u$-ring} is any path in the random interlacements graph $\cI^{u}$ that starts in $\partial B(L_u)$, ends in $\partial B(2L_u)$ and stays inside $B(2L_u)$.
	The aim of this proof is to show that the following property holds with high probability:
	\begin{enumerate}[label=(\Roman*)]
		\item  \label{aim} Every crossing of the $L_u$-ring must contain a path $w\in \supp\vec{\RI}^{u}_{B(2L_u)}$ that includes a sub-path of chemical distance at least $C_{\const{prop:annulus}{\ell}} u^{2\varepsilon} L_u^2 $ that does not intersect any of the other paths in $\supp\vec{\RI}^{u}_{B(2L_u)}$. 
	\end{enumerate}
	When this property holds, any crossing of the $L_u$-ring must use at least one of these sub-paths, providing us with the required lower bound on the chemical distance. Recalling that $\mathcal{I}^u$ is infinite and connected, it follows that if $B(L_u)\cap \mathcal{I}^u\neq\emptyset$ then there exists a crossing of the $L_u$-ring. Consequently, using the fact that $\prw(B(L_u)\cap \mathcal{I}^u\neq\emptyset)\geq 1-e^{-c_1uL_u^{d-2}}$, Property~\ref{aim} completes the proof.

	Fix $\varepsilon > 0$, set $ \uBound{u:propAnnulus} = \uBound{u:proof4}$ and  let $ u \in (0, \uBound{u:propAnnulus})$.
	Let us recall the events:
	\[
		\mathfs{K}_u = \{\max\{\mathcal{K}_1,\ldots,\mathcal{K}_l\}<u^{-\varepsilon/2}\}, \qquad
		\mathfs{N}_u = \{ N_{B(2 L_u)} < u^{1 - \varepsilon} L_u^{d-2} \}.
	\]
	By Lemma~\ref{lem:maxintbound}, 
	$
		\prob^u(\mathfs{K}_u)>1-u^{\bet{bet:eps}}
	$,
	and by Lemma~\ref{lemma:numberbound},
	$  
		\prob^u \left( \mathfs{N}_u \right) \geq 1 -  c u^{\varepsilon u^{-\bet{bet:eps}}/2}
	$, and hence 
	\begin{equation} \label{eq:event_1}
		\prob^u \left( \mathfs{K}_u \cap \mathfs{N}_u \right) \geq 1 - u^{\bet{bet:eps}} - c u^{\varepsilon u^{- \bet{bet:eps}}/2}\,.
	\end{equation}
	Since the right hand side in \eqref{eq:event_1} goes to $1$ as $u\to 0$, we can assume throughout the rest of the proof that the event $ \mathfs{K}_u \cap \mathfs{N}_u $ occurs.

  Partition the $L_u$-ring $B(2L_u) \setminus B(L_u) $ to $u^{-\varepsilon}$ annuli of width $u^{\varepsilon}L_u$.  
 	Note that the upper bound on $\varepsilon > 0$ guarantees that $\lim_{u\to 0}u^{\varepsilon}L_u =\infty$. We denote these annuli by $A_u^j$, for $  j = 1, \ldots ,  u^{-\varepsilon}$ (see Figure \ref{fig:annulipartition}).
	By the event $\mathfs{K}_u$ and the  pigeonhole principle, for each one of the clusters $\mathcal{C}_i$ there must be at least one annulus $A^{j (i)}_u$ of width $u^{\varepsilon}L_u$ satisfying the following property:
	\begin{enumerate}[label=(\Roman*)]
		\setcounter{enumi}{1}
	 	\item \label{cond:C1} Every path $w \in \supp \vec{\RI}^u_{B(2L_u)}$ in the cluster $\mathcal{C}_i$ avoids intersections within $A^{j (i)}_u$  with other paths.
	\end{enumerate}
	We have obtained that, with high probability, there are no intersections between paths in $\mathcal{C}_{i}$ while they traverse the annulus $A_u^{j(i)}$. 
	To obtain Property~\ref{aim}, our next goal is to (lower) bound the chemical distance of these paths while crossing $A_u^{j(i)}$. Note that we need to account for different returns of a single path $w \in \mathcal{C}_i$ to the same annulus, and the self-intersections of the path $w$ during each return.
	To control these returns as well as self-intersections, we further partition $A^{j (i)}_u$ into 3 sub-annuli, each of width $1/3\cdot u^{\varepsilon}L_u$. We refer to them as the  inner, middle and outer  annuli of $A^{j (i)}_u$ (see Figure \ref{fig:annulipartition}).  A connected sub-path of $w \in \mathcal{C}_i $ is called a \emph{crossing} of $\mathcal{C}_i$ if it stays inside the sub-annulus, and it starts and ends in the two disjoint boundaries of an annulus or sub-annulus.
	Our last goal is to show the following property.
	\begin{enumerate}[label=(\Roman*)]
	\setcounter{enumi}{2}
	 \item \label{cond:C2}   For every $i$ and every $w \in \mathcal{C}_i$, all its crossings of the middle sub-annulus of $A_u^{j(i)}$ are disjoint and have sufficiently large chemical distance.
	\end{enumerate}
	Every path $w \in \supp \vec{\RI}^{u}_{B (2L_u)} $ crossing the $L_u$-ring belongs to a cluster, then Property~\ref{cond:C1} and Property~\ref{cond:C2} imply Property~\ref{aim}.

	\begin{figure}[H] 
	\begin{tikzpicture}[scale=0.7]
		\coordinate (O) at (0,0);
		
		\draw[ultra thick] (0, 4) -- (0, 0);
		\draw[ultra thick] (12, 4) -- (12, 0);

		\draw[thick, dashed, blue] (4, 4) -- (4, 0);
		\draw[thick, dashed, blue] (8, 4) -- (8, 0);

				\draw [black,thick] plot [smooth, tension=1.5] 
				 coordinates { 
					(-1  , 3.8)   
					(0  , 3.5) 
					( 2.2  , 3.4)
					( 2.6 ,2.3)
					( 2.6  , 2.7)
					( 2.3  , 3.4)
					( 4 , 3.8) 
					( 5 , 2.8) 
					( 8 , 3.8)
					( 9 , 4.4)
					( 8 , 3.8)
					( 12 , 3.4)
					( 12 , 3.4)
					( 3,  1.2 )
					( 2  , 1.5)
					( 8.5  , .8)   
					( 9  , 2)
					( 9.5  , .8)
					( 10  , .7)
					( 13  , .5)
				};

		\node[blue] at (13, 0) (c)    { $ A^{j(i)}_u $};

	\end{tikzpicture}
	\caption{ This path has three crossings of the middle sub-annulus.}
	\end{figure}
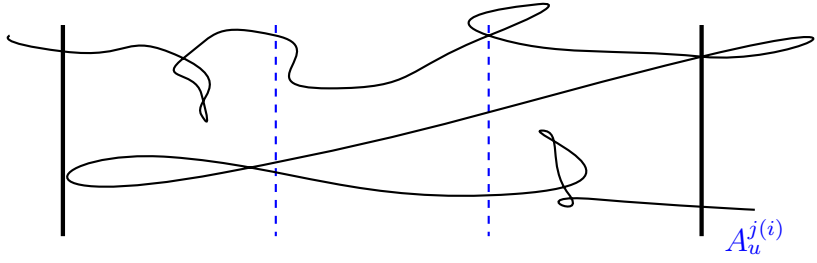
	
{In the following analysis, to avoid conditioning random walks, we will take a union bound on crossing any annulus instead of considering the special annulus $A_u^{j(i)}$. Let $X\in\supp \vec{\RI}^u_{B(2L_u)}$,  and let $\mathfs{Y}^j_X$ be the number of crossings of the inner, middle and outer sub-annuli of $A_u^{j}$ traversed by $X$. Let $  \newbet \label{1}, \newbet \label{hell}  > 0 $ such that $\bet{1}-\bet{hell}-2\varepsilon>0$ (whose value will be specified later).
Then 
\begin{equation}\label{eq:numberofcrossingsbound_one_path}
		\prob^u \Big(\mathfs{Y}^j_X \geq u^{-\bet{1}}\Big)\preceq u^{-\bet{1}}e^{-\expo{e:RWescape}u^{-\bet{hell}}}+e^{-\expo{e:rangeboundann}u^\frac{\bet{hell}-\bet{1}+2\varepsilon}{2}}\,.
\end{equation}
Indeed splitting the event according to the existence of a crossing of length $<u^{\bet{hell}+2\varepsilon}L_u^2/9$ among the first $u^{-\bet{1}}$ crossings traversed by the path $X$ we get: On the one hand, the probability that such a crossing indeed exists, by a union bound and Proposition \ref{prop: RW escape}, is bounded by 
	$$
	c \cdot u^{-\bet{1}}e^{-\expo{e:RWescape}u^{-\bet{hell}}}
	.$$
	On the other hand, if all crossings are of length at least $u^{\bet{hell}+{2\varepsilon}}L_u^2/9$, then the total amount of time spent by the random walk inside the box $B(2L_u)$ is at least $u^{-\bet{1}+\bet{hell}+{2\varepsilon}}L_u^2/9$. By Lemma \ref{lem:rangeboundann} and the fact that $-\bet{1}+\bet{hell}+2\varepsilon<0$ this event happens with a probability bounded by $c' \cdot e^{-\expo{e:rangeboundann}u^\frac{\bet{hell}-\bet{1}+2\varepsilon}{2}}$.
	Thus the bound in~\eqref{eq:numberofcrossingsbound_one_path} follows. 
	
	Now, let 
\[ \mathfs{Y} = \sum_{X\in \supp \vec{\RI}^u_{B(2L_u)}}\sum_{j=1}^{u^{-\varepsilon}} \mathfs{Y}^j_X \] be the total number of possible crossings. By union bound on the number of paths (on the event $\mathfs{N}_u = \{ N_{B(2 L_u)} < u^{-\bet{bet:eps}} \}$) and annuli, we obtain that
\begin{equation}\label{eq:numberofcrossingsbound}
		\prob^u \Big(\mathfs{Y} \leq u^{-\bet{1} - \bet{bet:eps}-\varepsilon}\Big) \geq 1 - c \cdot u^{-\varepsilon- \bet{bet:eps}}\left(u^{-\bet{1}}e^{-\expo{e:RWescape}u^{-\bet{hell}}}+e^{-\expo{e:rangeboundann}u^\frac{\bet{hell}-\bet{1}+2\varepsilon}{2}}\right)\,.
	\end{equation}

	If we certify, with high probability, that for every crossing of any annulus $A_u^{j}$ there are global cut-points of the crossings in the inner and outer sub-annuli, then different crossings of the middle sub-annulus are disjoint. Additionally, if we prove that all crossings of the middle sub-annulus have enough global cut-points, we obtain the desired lower bound on the chemical distance. By \eqref{eq:numberofcrossingsbound}, the number of such crossings is bounded by $u^{-\bet{1} - \bet{bet:eps}-\varepsilon}$ with high probability. 

	Let $ \newbet \label{4}  > 0$, whose value is chosen below.
	Let $w$ be a crossing of a sub-annulus of $A^{j}_u$, let $\sigma$ be the crossing time, let $x$ be its starting point, and let $Z_w$ be the number of global cut-points in $w[0,\sigma]$. By Proposition~\ref{prop: RW escape} and Proposition~\ref{prop:cutdensity}, 
	\begin{equation} \label{eq:cut_density} 
	\begin{split}
		\prw_x \left(Z_w< \cnt{c:cutdensity} u^{\bet{4} + 2\varepsilon}L_u^2\right) 
		&\le \prw_ x\left(\sigma<u^{\bet{4} + 2\varepsilon}L_u^2\right)+ \prw_x \left(Z_w<  \cnt{c:cutdensity} u^{\bet{4}+ 2\varepsilon}L_u^2,~\sigma>u^{\bet{4}+ 2\varepsilon}L_u^2\right) \\ 
		&\preceq  \log(1/u)^{\expo{dense}} \left(u^{\bet{4}+ 2\varepsilon}L_u^2\right)^{\frac{4-d}{2}}.
	\end{split}
	\end{equation}

	Now, since $\bet{1}>0$, we assume the event above, and consider all $ \mathfs{Y}$ sub-annuli crossings $w_i$ (each one starting at $x_i$).
	By~\eqref{eq:cut_density}, all these crossings have a positive density of global cut-points with probability bounded from below by:
	\begin{equation} \label{eq:allsteps}
	\begin{split}
	 \prob^{u} \left(\bigcap_{i = 1}^\mathfs{Y} \left \{  Z_{w_i} \ge \cnt{c:cutdensity} u^{\bet{4}+ 2\varepsilon}L_u^2\right\}\right)
		&	\ge 1-\sum_{i=1}^{u^{-\bet{1} - \bet{bet:eps}-\varepsilon}} \prw_{x_i} \left(Z_{w_i}< \cnt{c:cutdensity}u^{\bet{4}+ 2\varepsilon}L_u^2 \right) \\
		& \ge 1- c \cdot u^{-\bet{1}-\bet{bet:eps}-\varepsilon}  \log(1/u)^{\expo{dense}} \left(u^{\bet{4}+ 2\varepsilon}L_u^2\right)^{\frac{4-d}{2}} \\
		& = 1 - 
		c \cdot \log(1/u)^{\expo{dense}} \cdot u^{-\bet{1}-\bet{bet:eps}-\varepsilon + \left(\frac{4-d}{2}  \right) \left(\bet{4} + 8 \varepsilon - 1 \right)}  
	\end{split}
	\end{equation}
	Substituting the value of $\bet{bet:eps}$ in the polynomial factor, we get that the exponent of $u$ above simplifies to 
	\[
	\begin{split}
	&-\bet{1}-\bet{bet:eps}-\varepsilon + \left(\frac{4-d}{2}  \right) \left(\bet{4} + 8 \varepsilon - 1 \right)\\
	=&	
		-\bet{1} - \frac{d-4}{2} + \varepsilon (4d - 10) + \left(\frac{4-d}{2}  \right) \left(\bet{4} + 8 \varepsilon - 1 \right)  \\
	=& 
		-\bet{1}  + \varepsilon (4d - 10) + \left(\frac{4-d}{2}  \right) \left(\bet{4} + 8 \varepsilon  \right)\\
	=& 
		-\bet{1}  + 6\varepsilon - \left(\frac{d-4}{2}\right) \bet{4} 
	\end{split}
	\]
	We choose $\bet{1}=3\varepsilon$, $\bet{hell}=\frac{\varepsilon}{2}$ and $ 0 < \bet{4} < \frac{2}{d-4} \cdot \varepsilon $, and then the equation above is positive and $\bet{1}-\bet{hell}-2\varepsilon>0$. 
	This bounds the exponent in~\eqref{eq:allsteps} and we obtain some $\newbet \label{5}=\bet{5}(\varepsilon)>0$ such that
	\begin{equation} \label{eq:event_3}
	\prob^{u} \left(\bigcap_{i=1}^\mathfs{Y}\{Z_{w_i}\ge   \cnt{c:cutdensity} u^{\bet{4} + 2\varepsilon}L_u^2\}\right)\ge 1-u^{\bet{5}}.
	\end{equation}

	Since $\bet{4}<\varepsilon$, we let $ 0<C_{\const{prop:annulus}{\ell}}<\cnt{c:cutdensity}$ in~\eqref{eq:event_3}. With this choice of constant, we conclude the proposition by consider the intersection of the events in~\eqref{eq:event_1},~\eqref{eq:numberofcrossingsbound} and~\eqref{eq:event_3}.
}\end{proof}

\section{Lower bound on the chemical distance in random interlacements} \label{sec: lower bound}

In this section we prove Theorem~\ref{thm:main_lower}. Due to monotonicity, we can assume without loss of generality that $\varepsilon\in (0,1/24)$. Similarly to the proof of Theorem~\ref{thm:main_upper}, it is enough to work with $x = e_1 \coloneqq (1, 0 , \ldots , 0)$. To see this consider any $x=(x_1,x_2,\ldots,x_d)\in\mathbb{Z}^d$, and without loss of generality assume that $|x_1|\ge |x|_2/\sqrt{d}$ (note that there must be a direction for which this holds). Denote $\bar{x}=(x_1,-x_2,\ldots,-x_d)$ to be its reflection in the $e_1$ direction. Then by the reflection symmetry of the model, $2\frac{|x|_2}{\sqrt{d}}\rho_u(e_1)\le \rho_u(x)+\rho_u(\bar{x})=2\rho_u(x)$, yielding the desired lower bound for the direction $x$.

As a result Theorem~\ref{thm:main_lower} reduces to the following result.

\begin{prop} \label{prop: lower bound}
	Consider random interlacements on $\mathbb{Z}^d$ with $d \geq 5$. 
	There exists a constant  $   \Cnt{c:low}  =  \Cnt{c:low} (d)\in (0, \infty) $ such that the following holds: For every $\varepsilon > 0$, there exists $ \newuBound \label{u:propLowerBound} = \uBound{u:propLowerBound}(\varepsilon) > 0$   such that for any $u \in (0, \uBound{u:propLowerBound})$
	\begin{equation} \label{eq: lower bound}
		\rho_u (e_1) \geq \Cnt{c:low}   u^{-\frac{1}{2}+\varepsilon}.
	\end{equation}
\end{prop}

Proposition~\ref{prop: lower bound} follows from the following proposition applying it with $\varepsilon/4$ instead of $\varepsilon$ from the previous Proposition.

\begin{prop} \label{prop: asymp-bound}
	There exists a constant $ \Cnt{c:low}  =  \Cnt{c:low}  (d)\in (0, \infty)$ such that the following holds. 
	For every  $ \varepsilon \in (0, 1/24)$, there exist $\uBound{u:propLowerBound}  = \uBound{u:propLowerBound} (\varepsilon) > 0$,  $ \newCnt \label{c:asymp1}  =  \Cnt{c:asymp1} (\uBound{u:propLowerBound})\in (0, \infty) $, $ \newCnt \label{c:asymp2}  =  \Cnt{c:asymp2} (\uBound{u:propLowerBound})\in (0, \infty) $, $ \newexpo \label{e:asymp} =  \expo{e:asymp} \in (0,1)  $ and $N = N(\uBound{u:propLowerBound})\in\mathbb{N}$ so that 
	\begin{equation} \label{eq:asymp-bound}
		\prob^u \Big( d_u \left([0]_u, \left[ \left( n L_u \right) \cdot e_1 \right]_u \right) \geq   \Cnt{c:low} u^{-\frac{1}{2}+4\varepsilon} (n L_u)\Big) >1 - \Cnt{c:asymp1} e^{ - \Cnt{c:asymp2} e^{(\log n)^{\expo{e:asymp}}} }
	\end{equation}
	for every $u\in (0,\uBound{u:propLowerBound})$ and $n\geq N$, where 
$	L_u = u^{-\frac{1}{2}+4\varepsilon}$. 
\end{prop}

Our strategy follows a multi-scale renormalization argument. 
The first step of the argument is based on the results in Section~\ref{sec:local_lower_bound} and in particular on Proposition~\ref{prop:annulus}, which gives a  lower bound on the chemical distance between a box of radius $L_u$ and a box with radius $2 L_u$.
Such event will correspond to level $0$ of the renormalization introduced below.
Subsection~\ref{subsec: multiscale-renormalization} below presents the framework of the multi-scale renormalization, after which we prove Proposition~\ref{prop: asymp-bound} in Subsection \ref{subsec:proof_of_proposition}.


\subsection{Multi-scale renormalization} \label{subsec: multiscale-renormalization}

The multi-scale renormalization setup we used is based on a similar construction in~\cite[Section IV]{DrewitzRathSapozhnikov}.

Consider the vector $\Theta =  ( \theta_{\text{sc}},  L_0, l_0, r_0)\in (0,\infty)\times\mathbb{N}^3$ to which we refer as \emph{renormalization scales}.
Let $\varepsilon \in (0,  1/24)$ (we fix this value for the rest of the section), and let $\newuBound \label{u:fixedEps} \coloneqq  \uBound{u:propAnnulus} (\varepsilon) > 0 $ be the corresponding upper bound on the intensity from Proposition~\ref{prop:annulus}.
Let $u \in (0, \uBound{u:fixedEps})$ and set $L_0\in \Theta$ to be 
\begin{equation} \label{eq:L0_scale}
		L_0  \coloneqq L_{u} = u^{-\frac{1}{2}+4\varepsilon}.
\end{equation}
For the other scales in $\Theta$:  $\theta_{\text{sc}}$, $l_0$, and $r_0$ are positive integers. We assume that $l_0$, $r_0$ are large and satisfy $ r_0 / l_0 < 1/4 $.
Then we define increasing sequences of positive integers $ (l_k)_{k \geq 0}$, $(r_k)_{k \geq 0}$, and $ (L_k)_{k \geq 0} $ by
\begin{equation} \label{eq:renormalization-seq}
	l_k = l_0 \cdot 4^{k^{\theta_{\text{sc}}}} , \quad r_k = r_0 \cdot 2^{k^{\theta_{\text{sc}}}}, \quad L_k = l_{k-1} \cdot L_{k-1}, \quad k \geq 1.
\end{equation}
Note that the assumptions above imply $ l_k > r_k $, for all $k \geq 1$.

The sequences $(l_k)_{k\geq 0}$, $(r_k)_{k\geq 0}$ and $(L_k)_{k\geq 0}$ define exponentially growing scales, which we will use below to define the size of boxes in the  multi-scale renormalization.
The integer $ L_k  $ sets the size of renormalized boxes at level $k$,  while $r_{k-1} L_{k-1}$ bounds the side length of the regions containing ``bad boxes'' at level $k-1$ (see Definition~\ref{def:multiscale_0-good} below).
Next, we begin with the construction of the renormalization and define the aforementioned boxes and events. Then~Proposition~\ref{prop:multi_renormalization} gives values for $\theta_{\text{sc}}$,  $l_0$ and $r_0$ that are useful in our setting.

\begin{figure}[H]
\begin{tikzpicture}
	\draw[step=.05cm,gray,ultra thin] (0,0) grid (5.4,5.4);

	\draw [fill=lightgreen,fill opacity=0.8]  (3.60,4.95) rectangle (4.50,5.40);			
	\draw [fill=lightgreen,fill opacity=0.8]  (3.15,4.50) rectangle (4.05,4.95);		  
	\draw [fill=lightgreen,fill opacity=0.8]  (4.50,4.50) rectangle (4.95,4.95);			  
	\draw [fill=lightgreen,fill opacity=0.8]  (4.05,4.05) rectangle (4.50,4.50); 			


			\draw [fill=blue!60,fill opacity=0.8,draw=none] (0.25,5.15) rectangle (0.35,5.25); 

			\draw [fill=blue!60,fill opacity=0.8,draw=none] (0.50,5.05) rectangle (0.55,5.10); 
			\draw [fill=blue!60,fill opacity=0.8,draw=none] (0.55,5.05) rectangle (0.60,5.10);
			\draw [fill=blue!60,fill opacity=0.8,draw=none] (0.55,5.00) rectangle (0.60,5.05);

			\draw [fill=blue!60,fill opacity=0.8,draw=none] (0.95,5.25) rectangle (1.05,5.35); 

			\draw [fill=blue!60,fill opacity=0.8,draw=none] (1.35,5.30) rectangle (1.40,5.35);
			\draw [fill=blue!60,fill opacity=0.8,draw=none] (1.45,5.30) rectangle (1.50,5.35);
			\draw [fill=blue!60,fill opacity=0.8,draw=none] (1.35,5.25) rectangle (1.40,5.30);
			\draw [fill=blue!60,fill opacity=0.8,draw=none] (1.40,5.25) rectangle (1.45,5.30);

			\draw [fill=blue!60,fill opacity=0.8,draw=none] (1.90,5.15) rectangle (2.00,5.25);

			\draw [fill=blue!60,fill opacity=0.8,draw=none] (2.40,5.15) rectangle (2.50,5.25); 

			\draw [fill=blue!60,fill opacity=0.8,draw=none] (2.95,5.15) rectangle (3.05,5.25); 
			\draw [fill=blue!60,fill opacity=0.8,draw=none] (3.05,5.15) rectangle (3.10,5.25); 

			\draw [fill=blue!60,fill opacity=0.8,draw=none] (3.30,5.00) rectangle (3.45,5.05); 
			\draw [fill=blue!60,fill opacity=0.8,draw=none] (3.35,4.95) rectangle (3.40,5.00);

			\draw [fill=blue!60,fill opacity=0.8,draw=none] (4.60,5.10) rectangle (4.70,5.20); 

			\draw [fill=blue!60,fill opacity=0.8,draw=none] (5.15,5.35) rectangle (5.20,5.40); 
			\draw [fill=blue!60,fill opacity=0.8,draw=none] (5.10,5.30) rectangle (5.15,5.35);	
	

			\draw [fill=blue!60,fill opacity=0.8,draw=none] (0.35,4.90) rectangle (0.45,4.95);  
			\draw [fill=blue!60,fill opacity=0.8,draw=none] (0.30,4.85) rectangle (0.35,4.90);

			\draw [fill=blue!60,fill opacity=0.8,draw=none] (0.80,4.50) rectangle (0.90,4.60);
	
			\draw [fill=blue!60,fill opacity=0.8,draw=none] (1.30,4.60) rectangle (1.35,4.70);

			\draw [fill=blue!60,fill opacity=0.8,draw=none] (1.55,4.65) rectangle (1.65,4.75);  

			\draw [fill=blue!60,fill opacity=0.8,draw=none] (1.80,4.50) rectangle (1.90,4.60);

			\draw [fill=blue!60,fill opacity=0.8,draw=none] (2.25,4.50) rectangle (2.35,4.60);  

			\draw [fill=blue!60,fill opacity=0.8,draw=none] (3.00,4.60) rectangle (3.10,4.70);

			\draw [fill=blue!60,fill opacity=0.8,draw=none] (4.25,4.75) rectangle (4.35,4.85);

			\draw [fill=blue!60,fill opacity=0.8,draw=none] (5.15,4.80) rectangle (5.20,4.90); 


			\draw [fill=blue!60,fill opacity=0.8,draw=none] (0.40,4.40) rectangle (0.45,4.45);
			\draw [fill=blue!60,fill opacity=0.8,draw=none] (0.35,4.35) rectangle (0.45,4.40);

			\draw [fill=blue!60,fill opacity=0.8,draw=none] (0.75,4.45) rectangle (0.80,4.50); 
			\draw [fill=blue!60,fill opacity=0.8,draw=none] (0.85,4.45) rectangle (0.90,4.50);
			\draw [fill=blue!60,fill opacity=0.8,draw=none] (0.75,4.40) rectangle (0.85,4.45);

			\draw [fill=blue!60,fill opacity=0.8,draw=none] (1.00,4.35) rectangle (1.10,4.40); 
			\draw [fill=blue!60,fill opacity=0.8,draw=none] (0.95,4.30) rectangle (1.00,4.35);

			\draw [fill=blue!60,fill opacity=0.8,draw=none] (1.75,4.05) rectangle (1.80,4.15);

			\draw [fill=blue!60,fill opacity=0.8,draw=none] (1.80,4.45) rectangle (1.90,4.50); 

			\draw [fill=blue!60,fill opacity=0.8,draw=none] (2.25,4.40) rectangle (2.35,4.50); 

			\draw [fill=blue!60,fill opacity=0.8,draw=none] (2.85,4.35) rectangle (2.95,4.40); 
			\draw [fill=blue!60,fill opacity=0.8,draw=none] (2.85,4.25) rectangle (2.95,4.30);

			\draw [fill=blue!60,fill opacity=0.8,draw=none] (3.35,4.30) rectangle (3.45,4.20); 

			\draw [fill=blue!60,fill opacity=0.8,draw=none] (4.00,4.40) rectangle (4.05,4.45);
			\draw [fill=blue!60,fill opacity=0.8,draw=none] (3.95,4.35) rectangle (4.00,4.40);
			\draw [fill=blue!60,fill opacity=0.8,draw=none] (4.00,4.35) rectangle (4.05,4.40);

			\draw [fill=blue!60,fill opacity=0.8,draw=none] (4.80,4.25) rectangle (4.90,4.30); 
			\draw [fill=blue!60,fill opacity=0.8,draw=none] (4.85,4.30) rectangle (4.90,4.35);

			\draw [fill=blue!60,fill opacity=0.8,draw=none] (5.10,4.20) rectangle (5.25,4.25); 
			\draw [fill=blue!60,fill opacity=0.8,draw=none] (5.20,4.15) rectangle (5.25,4.20);

			\draw [fill=blue!60,fill opacity=0.8,draw=none] (0.40,3.65) rectangle (0.45,3.70); 
			\draw [fill=blue!60,fill opacity=0.8,draw=none] (0.30,3.60) rectangle (0.40,3.70);

			\draw [fill=blue!60,fill opacity=0.8,draw=none] (0.45,3.65) rectangle (0.50,3.75); 
			\draw [fill=blue!60,fill opacity=0.8,draw=none] (0.50,3.70) rectangle (0.55,3.75);

			\draw [fill=blue!60,fill opacity=0.8,draw=none] (1.05,3.70) rectangle (1.15,3.80);

			\draw [fill=blue!60,fill opacity=0.8,draw=none] (1.45,3.70) rectangle (1.55,3.80); 

			\draw [fill=blue!60,fill opacity=0.8,draw=none] (2.00,3.65) rectangle (2.10,3.75);

			\draw [fill=blue!60,fill opacity=0.8,draw=none] (2.35,3.90) rectangle (2.45,4.00); 

			\draw [fill=blue!60,fill opacity=0.8,draw=none] (2.95,3.60) rectangle (3.05,3.70);

			\draw [fill=blue!60,fill opacity=0.8,draw=none] (3.20,3.90) rectangle (3.30,3.95);

			\draw [fill=blue!60,fill opacity=0.8,draw=none] (3.90,3.80) rectangle (4.00,3.90); 

			\draw [fill=blue!60,fill opacity=0.8,draw=none] (4.30,3.65) rectangle (4.50,3.70); 

			\draw [fill=blue!60,fill opacity=0.8,draw=none] (4.50,3.60) rectangle (4.60,3.65); 
			\draw [fill=blue!60,fill opacity=0.8,draw=none] (4.50,3.65) rectangle (4.55,3.70);

			\draw [fill=blue!60,fill opacity=0.8,draw=none] (5.15,3.75) rectangle (5.20,3.90); 
	
			\draw [fill=blue!60,fill opacity=0.8,draw=none] (0.30,3.55) rectangle (0.35,3.50); 
			\draw [fill=blue!60,fill opacity=0.8,draw=none] (0.35,3.45) rectangle (0.45,3.55);

			\draw [fill=blue!60,fill opacity=0.8,draw=none] (0.40,3.55) rectangle (0.50,3.60); 

			\draw [fill=blue!60,fill opacity=0.8,draw=none] (0.90,3.50) rectangle (1.00,3.60); 

			\draw [fill=blue!60,fill opacity=0.8,draw=none] (1.65,3.25) rectangle (1.70,3.30); 
			\draw [fill=blue!60,fill opacity=0.8,draw=none] (1.60,3.15) rectangle (1.70,3.25); 

			\draw [fill=blue!60,fill opacity=0.8,draw=none] (1.85,3.20) rectangle (1.90,3.25); 
			\draw [fill=blue!60,fill opacity=0.8,draw=none] (1.90,3.15) rectangle (2.00,3.25);

			\draw [fill=blue!60,fill opacity=0.8,draw=none] (2.55,3.15) rectangle (2.65,3.25); 

			\draw [fill=blue!60,fill opacity=0.8,draw=none] (2.85,3.50) rectangle (2.95,3.60); 

			\draw [fill=blue!60,fill opacity=0.8,draw=none] (3.45,3.40) rectangle (3.55,3.50);

			\draw [fill=blue!60,fill opacity=0.8,draw=none] (3.65,3.55) rectangle (3.75,3.60); 

			\draw [fill=blue!60,fill opacity=0.8,draw=none] (4.30,3.15) rectangle (4.40,3.25); 

			\draw [fill=blue!60,fill opacity=0.8,draw=none] (4.50,3.45) rectangle (4.60,3.55); 

			\draw [fill=blue!60,fill opacity=0.8,draw=none] (5.15,3.55) rectangle (5.25,3.60); 
			\draw [fill=blue!60,fill opacity=0.8,draw=none] (5.20,3.50) rectangle (5.30,3.55);

			\draw [fill=blue!60,fill opacity=0.8,draw=none] (0.15,2.75) rectangle (0.25,2.85); 

			\draw [fill=blue!60,fill opacity=0.8,draw=none] (0.60,2.95) rectangle (0.70,3.05); 

			\draw [fill=blue!60,fill opacity=0.8,draw=none] (1.00,3.05) rectangle (1.10,3.10); 
			\draw [fill=blue!60,fill opacity=0.8,draw=none] (1.00,3.00) rectangle (1.05,3.05);

			\draw [fill=blue!60,fill opacity=0.8,draw=none] (1.65,3.05) rectangle (1.75,3.15); 
			\draw [fill=blue!60,fill opacity=0.8,draw=none] (1.75,3.05) rectangle (1.80,3.10); 

			\draw [fill=blue!60,fill opacity=0.8,draw=none] (2.10,3.00) rectangle (2.20,3.05); 
			\draw [fill=blue!60,fill opacity=0.8,draw=none] (2.05,2.95) rectangle (2.10,3.00);

			\draw [fill=blue!60,fill opacity=0.8,draw=none] (2.60,2.70) rectangle (2.70,2.80);

			\draw [fill=blue!60,fill opacity=0.8,draw=none] (3.10,3.10) rectangle (3.15,3.15);

			\draw [fill=blue!60,fill opacity=0.8,draw=none] (3.40,2.95) rectangle (3.50,3.00); 
			\draw [fill=blue!60,fill opacity=0.8,draw=none] (3.40,2.90) rectangle (3.45,2.95);

			\draw [fill=blue!60,fill opacity=0.8,draw=none] (3.60,2.70) rectangle (3.65,2.80); 
			\draw [fill=blue!60,fill opacity=0.8,draw=none] (3.70,2.70) rectangle (3.75,2.80);

			\draw [fill=blue!60,fill opacity=0.8,draw=none] (4.20,3.00) rectangle (4.30,3.10);

			\draw [fill=blue!60,fill opacity=0.8,draw=none] (4.60,2.70) rectangle (4.70,2.80);

			\draw [fill=blue!60,fill opacity=0.8,draw=none] (4.95,2.95) rectangle (5.05,3.05);

			\draw [fill=blue!60,fill opacity=0.8,draw=none] (0.30,2.25) rectangle (0.35,2.35); 
			\draw [fill=blue!60,fill opacity=0.8,draw=none] (0.40,2.30) rectangle (0.45,2.35);

			\draw [fill=blue!60,fill opacity=0.8,draw=none] (0.45,2.30) rectangle (0.55,2.40); 

			\draw [fill=blue!60,fill opacity=0.8,draw=none] (1.25,2.50) rectangle (1.35,2.65); 
		
			\draw [fill=blue!60,fill opacity=0.8,draw=none] (1.35,2.45) rectangle (1.45,2.60); 

			\draw [fill=blue!60,fill opacity=0.8,draw=none] (1.90,2.50) rectangle (2.00,2.60); 

			\draw [fill=blue!60,fill opacity=0.8,draw=none] (2.35,2.25) rectangle (2.45,2.35);

			\draw [fill=blue!60,fill opacity=0.8,draw=none] (2.90,2.60) rectangle (3.00,2.65); 

			\draw [fill=blue!60,fill opacity=0.8,draw=none] (3.50,2.35) rectangle (3.60,2.45); 

			\draw [fill=blue!60,fill opacity=0.8,draw=none] (3.70,2.55) rectangle (3.75,2.65); 
			\draw [fill=blue!60,fill opacity=0.8,draw=none] (3.80,2.55) rectangle (3.85,2.65);

			\draw [fill=blue!60,fill opacity=0.8,draw=none] (4.25,2.45) rectangle (4.35,2.55); 

			\draw [fill=blue!60,fill opacity=0.8,draw=none] (4.60,2.60) rectangle (4.65,2.65);
			\draw [fill=blue!60,fill opacity=0.8,draw=none] (4.55,2.65) rectangle (4.65,2.70);

			\draw [fill=blue!60,fill opacity=0.8,draw=none] (4.95,2.25) rectangle (5.05,2.35);
	
			\draw [fill=blue!60,fill opacity=0.8,draw=none] (0.25,2.15) rectangle (0.35,2.25); 

			\draw [fill=blue!60,fill opacity=0.8,draw=none] (0.65,2.00) rectangle (0.75,2.10); 

			\draw [fill=blue!60,fill opacity=0.8,draw=none] (1.25,2.00) rectangle (1.30,2.05); 

			\draw [fill=blue!60,fill opacity=0.8,draw=none] (1.75,2.20) rectangle (1.80,2.25);

			\draw [fill=blue!60,fill opacity=0.8,draw=none] (1.85,2.15) rectangle (1.90,2.20);
			\draw [fill=blue!60,fill opacity=0.8,draw=none] (1.85,2.10) rectangle (1.95,2.15);

			\draw [fill=blue!60,fill opacity=0.8,draw=none] (2.50,2.10) rectangle (2.60,2.20); 

			\draw [fill=blue!60,fill opacity=0.8,draw=none] (2.70,2.05) rectangle (2.80,2.15); 

			\draw [fill=blue!60,fill opacity=0.8,draw=none] (3.15,2.05) rectangle (3.25,2.15); 

			\draw [fill=blue!60,fill opacity=0.8,draw=none] (3.70,1.85) rectangle (3.80,1.95);

			\draw [fill=blue!60,fill opacity=0.8,draw=none] (4.05,2.20) rectangle (4.15,2.25); 

			\draw [fill=blue!60,fill opacity=0.8,draw=none] (4.55,2.20) rectangle (4.65,2.25); 
			\draw [fill=blue!60,fill opacity=0.8,draw=none] (4.50,2.15) rectangle (4.55,2.20);

			\draw [fill=blue!60,fill opacity=0.8,draw=none] (4.95,2.05) rectangle (5.05,2.15);
	
			\draw [fill=blue!60,fill opacity=0.8,draw=none] (0.35,1.60) rectangle (0.45,1.65); 
			\draw [fill=blue!60,fill opacity=0.8,draw=none] (0.35,1.55) rectangle (0.40,1.60);
			\draw [fill=blue!60,fill opacity=0.8,draw=none] (0.40,1.50) rectangle (0.45,1.55);

			\draw [fill=blue!60,fill opacity=0.8,draw=none] (0.75,1.75) rectangle (0.80,1.80); 
			\draw [fill=blue!60,fill opacity=0.8,draw=none] (0.70,1.65) rectangle (0.80,1.75);

			\draw [fill=blue!60,fill opacity=0.8,draw=none] (1.15,1.70) rectangle (1.25,1.80); 
			\draw [fill=blue!60,fill opacity=0.8,draw=none] (1.20,1.65) rectangle (1.25,1.70);

			\draw [fill=blue!60,fill opacity=0.8,draw=none] (1.75,1.70) rectangle (1.80,1.75);
			\draw [fill=blue!60,fill opacity=0.8,draw=none] (1.70,1.65) rectangle (1.80,1.70); 

			\draw [fill=blue!60,fill opacity=0.8,draw=none] (1.85,1.70) rectangle (1.95,1.80); 

			\draw [fill=blue!60,fill opacity=0.8,draw=none] (2.50,1.55) rectangle (2.60,1.60);

			\draw [fill=blue!60,fill opacity=0.8,draw=none] (2.95,1.60) rectangle (3.05,1.70); 

			\draw [fill=blue!60,fill opacity=0.8,draw=none] (3.50,1.75) rectangle (3.55,1.80); 
			\draw [fill=blue!60,fill opacity=0.8,draw=none] (3.40,1.70) rectangle (3.50,1.80); 

			\draw [fill=blue!60,fill opacity=0.8,draw=none] (4.00,1.40) rectangle (4.05,1.50); 
			\draw [fill=blue!60,fill opacity=0.8,draw=none] (3.95,1.45) rectangle (4.00,1.50);

			\draw [fill=blue!60,fill opacity=0.8,draw=none] (4.05,1.35) rectangle (4.15,1.45); 

			\draw [fill=blue!60,fill opacity=0.8,draw=none] (4.85,1.60) rectangle (4.95,1.75); 

			\draw [fill=blue!60,fill opacity=0.8,draw=none] (4.95,1.65) rectangle (5.00,1.75); 
			\draw [fill=blue!60,fill opacity=0.8,draw=none] (5.00,1.65) rectangle (5.05,1.75);
	
			\draw [fill=blue!60,fill opacity=0.8,draw=none] (0.20,1.20) rectangle (0.30,1.30); 

			\draw [fill=blue!60,fill opacity=0.8,draw=none] (0.50,1.20) rectangle (0.60,1.35); 

			\draw [fill=blue!60,fill opacity=0.8,draw=none] (0.90,0.90) rectangle (1.05,1.00); 

			\draw [fill=blue!60,fill opacity=0.8,draw=none] (1.70,1.20) rectangle (1.80,1.30); 

			\draw [fill=blue!60,fill opacity=0.8,draw=none] (2.05,1.20) rectangle (2.15,1.30);

			\draw [fill=blue!60,fill opacity=0.8,draw=none] (2.50,1.00) rectangle (2.60,1.10);

			\draw [fill=blue!60,fill opacity=0.8,draw=none] (3.00,1.25) rectangle (3.10,1.30); 
			\draw [fill=blue!60,fill opacity=0.8,draw=none] (3.00,1.15) rectangle (3.10,1.20);

			\draw [fill=blue!60,fill opacity=0.8,draw=none] (3.25,1.25) rectangle (3.30,1.30);
			\draw [fill=blue!60,fill opacity=0.8,draw=none] (3.15,1.20) rectangle (3.25,1.25);

			\draw [fill=blue!60,fill opacity=0.8,draw=none] (3.95,1.25) rectangle (4.05,1.35); 

			\draw [fill=blue!60,fill opacity=0.8,draw=none] (4.05,1.25) rectangle (4.15,1.35); 

			\draw [fill=blue!60,fill opacity=0.8,draw=none] (4.85,1.10) rectangle (4.95,1.20);

			\draw [fill=blue!60,fill opacity=0.8,draw=none] (4.95,1.15) rectangle (5.05,1.25); 

			\draw [fill=blue!60,fill opacity=0.8,draw=none] (0.40,0.65) rectangle (0.45,0.70);
			\draw [fill=blue!60,fill opacity=0.8,draw=none] (0.35,0.60) rectangle (0.40,0.65);
			\draw [fill=blue!60,fill opacity=0.8,draw=none] (0.40,0.60) rectangle (0.45,0.65);

			\draw [fill=blue!60,fill opacity=0.8,draw=none] (0.80,0.70) rectangle (0.90,0.80);

			\draw [fill=blue!60,fill opacity=0.8,draw=none] (1.10,0.80) rectangle (1.20,0.90); 

			\draw [fill=blue!60,fill opacity=0.8,draw=none] (1.60,0.70) rectangle (1.75,0.75); 

			\draw [fill=blue!60,fill opacity=0.8,draw=none] (1.85,0.70) rectangle (1.95,0.80);

			\draw [fill=blue!60,fill opacity=0.8,draw=none] (2.50,0.75) rectangle (2.60,0.85); 

			\draw [fill=blue!60,fill opacity=0.8,draw=none] (2.95,0.80) rectangle (3.05,0.90); 

			\draw [fill=blue!60,fill opacity=0.8,draw=none] (3.15,0.70) rectangle (3.25,0.80);

			\draw [fill=blue!60,fill opacity=0.8,draw=none] (4.00,0.70) rectangle (4.05,0.75); 
			\draw [fill=blue!60,fill opacity=0.8,draw=none] (3.95,0.65) rectangle (4.00,0.70);
			\draw [fill=blue!60,fill opacity=0.8,draw=none] (3.95,0.60) rectangle (4.00,0.65);
			\draw [fill=blue!60,fill opacity=0.8,draw=none] (4.00,0.60) rectangle (4.05,0.65);

			\draw [fill=blue!60,fill opacity=0.8,draw=none] (4.40,0.85) rectangle (4.45,0.90);
			\draw [fill=blue!60,fill opacity=0.8,draw=none] (4.35,0.80) rectangle (4.40,0.85);
			\draw [fill=blue!60,fill opacity=0.8,draw=none] (4.45,0.80) rectangle (4.50,0.85);
			\draw [fill=blue!60,fill opacity=0.8,draw=none] (4.40,0.75) rectangle (4.45,0.80);
			\draw [fill=blue!60,fill opacity=0.8,draw=none] (4.45,0.75) rectangle (4.50,0.80);

			\draw [fill=blue!60,fill opacity=0.8,draw=none] (4.80,0.80) rectangle (4.90,0.90); 

			\draw [fill=blue!60,fill opacity=0.8,draw=none] (5.30,0.80) rectangle (5.40,0.90);
	
			\draw [fill=blue!60,fill opacity=0.8,draw=none] (0.20,0.30) rectangle (0.25,0.35);
			\draw [fill=blue!60,fill opacity=0.8,draw=none] (0.30,0.30) rectangle (0.35,0.35);
			\draw [fill=blue!60,fill opacity=0.8,draw=none] (0.20,0.25) rectangle (0.25,0.30);
			\draw [fill=blue!60,fill opacity=0.8,draw=none] (0.25,0.25) rectangle (0.30,0.30);

			\draw [fill=blue!60,fill opacity=0.8,draw=none] (0.80,0.05) rectangle (0.90,0.15); 

			\draw [fill=blue!60,fill opacity=0.8,draw=none] (0.90,0.00) rectangle (1.00,0.10);

			\draw [fill=blue!60,fill opacity=0.8,draw=none] (1.70,0.25) rectangle (1.80,0.35);

			\draw [fill=blue!60,fill opacity=0.8,draw=none] (1.80,0.30) rectangle (1.90,0.45); 

			\draw [fill=blue!60,fill opacity=0.8,draw=none] (2.55,0.10) rectangle (2.65,0.20);

			\draw [fill=blue!60,fill opacity=0.8,draw=none] (2.80,0.05) rectangle (2.90,0.15);

			\draw [fill=blue!60,fill opacity=0.8,draw=none] (3.25,0.35) rectangle (3.35,0.40); 
			\draw [fill=blue!60,fill opacity=0.8,draw=none] (3.25,0.25) rectangle (3.35,0.30);

			\draw [fill=blue!60,fill opacity=0.8,draw=none] (3.95,0.25) rectangle (4.05,0.35);

			\draw [fill=blue!60,fill opacity=0.8,draw=none] (4.10,0.10) rectangle (4.15,0.15);
			\draw [fill=blue!60,fill opacity=0.8,draw=none] (4.05,0.05) rectangle (4.10,0.10); 
			\draw [fill=blue!60,fill opacity=0.8,draw=none] (4.15,0.05) rectangle (4.20,0.10);
			\draw [fill=blue!60,fill opacity=0.8,draw=none] (4.10,0.00) rectangle (4.15,0.05);
			\draw [fill=blue!60,fill opacity=0.8,draw=none] (4.15,0.00) rectangle (4.20,0.05); 

			\draw [fill=blue!60,fill opacity=0.8,draw=none] (4.85,0.25) rectangle (4.95,0.35);

			\draw [fill=blue!60,fill opacity=0.8,draw=none] (5.20,0.00) rectangle (5.30,0.10);
	\draw[step=.45cm,black,thin] (0,0) grid (5.4,5.4);
\end{tikzpicture}
\caption{(C.f.~\cite[Fig. 2]{ProcacciaRosenthalSapozhnikov}.) The scales in the figure are $l_1 = 12$, $r_1 = 3$, $l_0 = 9$, and $r_0 = 2$. The small white boxes have size $L_0$. Then the figure shows $12 \times 12$ boxes of size $L_1$, which constitute one box of size $L_2$.  In the figure bad boxes are coloured.} 

\end{figure}
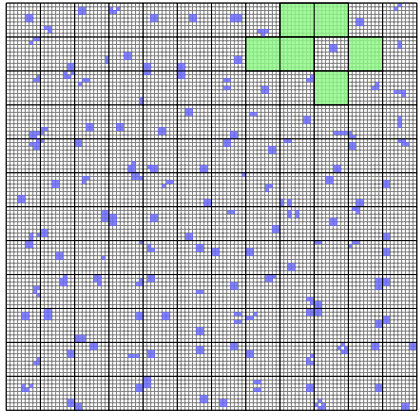

For $ k \geq 0$, define the renormalized graph $\mathbb{G}_k$ to be the graph whose vertex set is
\[
	\mathbb{G}_k \coloneqq L_k \mathbb{Z}^d = \{ L_k \cdot x \: : \: x \in \mathbb{Z}^d \}.
\]
The corresponding edge set is the collection of pairs of vertices at distance $L_k$, in other words, $\mathbb{G}_k$ has edges between its nearest neighbours  (when $\mathbb{G}_k$ is considered as an isomorphic graph to $\mathbb{Z}^d$). For $k \geq 1$ and $x \in \mathbb{Z}^d$, write
\[
	V (x, L_k) := x + [0, L_k)^d
\]
for the (half-closed) 
box with edge-corner $x$ and side length $L_k$.

Let us define conditions for a vertex in $\mathbb{G}_0$ to be good at level $0$. 
Recall the constant $C_{\const{prop:annulus}{\ell} } > 0$ from Proposition~\ref{prop:annulus}, and based on the event considered in that proposition, for $x \in \mathbb{G}_0$ and $ u, v > 0 $ we define
\begin{equation} \label{eq:multiscale_0-good}
		G^{u,v}_{x, 0} \coloneqq \left\{   \hat{d}_v \Big( \partial B_x (L_u), \partial B_x (2L_u) \Big) > C_{\const{prop:annulus}{\ell} } u^{-\frac{1}{2}+8\varepsilon} L_u \right\} .
\end{equation}

\begin{definition} \label{def:multiscale_0-good}
	The vertex $x \in  \mathbb{G}_0$ is \emph{$0$-good} if the event
	\[ 
		G^{u,u}_{x, 0} = \left\{   \hat{d}_u \Big( \partial B_x (L_u), \partial B_x (2L_u) \Big) > C_{\const{prop:annulus}{\ell} } u^{-\frac{1}{2}+8\varepsilon} L_u \right\}
	\]
	happens.
	In this case, we also call $ V (x, L_0) $ a $0$-good box. 
	Otherwise, $x$ is a $0$-bad vertex and $ V (x, L_0) $ is a $0$-bad box.
\end{definition}

The favourable event $G_{x, 0}^{u,v}  \in \sigma \left( \Psi_{(w,z)} \st w, z \in B_x (2L_0) \right)$ in~\eqref{eq:multiscale_0-good} is monotonic decreasing in $v$ (recall that $(\Psi_e)_{e \in E}$ are the canonical edge maps in~\eqref{eq:coordmaps}). 
Indeed, the complement of \eqref{eq:multiscale_0-good} is the event
\begin{equation} \label{eq:0bad-event}
	\overline{G}^{u,v}_{x, 0} \coloneqq  \left( G^{u,v}_{x,0} \right)^c = \left\{  
		\begin{array}{c} 
		\text{There exists a path } \gamma \text{ in } \cI^v \text{ connecting} \\ 
		 \partial B_x(L_u) \text{ and } \partial B_x(2 L_u) \text{ such that}  \\
			 \len (\gamma) <  C_{\const{prop:annulus}{\ell}} u^{-\frac{1}{2}+8\varepsilon} L_u
		\end{array}  
	\right\}.
\end{equation}
which is monotonic increasing in $v$.
We consider $ \overline{G}^{u,v}_{x, 0}$ as an unfavourable event.

\begin{remark}
	The half-closed boxes $V (x, L_k)$ refer to the state of the vertex $x$ on the renormalization. We follow the convention in these type of arguments, and set $x$ as the left-bottom corner of $V (x, L_k)$. We remark, however, that the favourable event $G_{x,0}^{u,v}$ is measurable on the $\sigma$-algebra generated by the coordinate maps of $B_x (2L_0)$; where in this later box, $x$ is at the centre.  
\end{remark}

In the framework of the multi-scale renormalization,
Proposition~\ref{prop:annulus} shows that a vertex $x \in \mathbb{G}_0$ is $0$-good with high probability as $u \to 0$, or equivalently, as $L_0 \to \infty$ (here we recall that $\varepsilon > 0 $ in \eqref{eq:L0_scale} is a fixed value).
Let us present this statement as a corollary of Proposition~\ref{prop:annulus}.
\begin{lemma} \label{lemma:high_0-good}
	There exists a $  \newet \label{uvfix} \in (0,1) $ such that 
	\[ 
		\adjustlimits\lim_{u \to 0} \sup_{x \in \mathbb{G}_0}  \prob^{u^{1-\et{uvfix}}} \left(  \overline{G}^{u,u^{1-\et{uvfix}}}_{x,0} \right) = 0.
	\] 
	Furthermore, the convergence is uniform on compacts.
\end{lemma}
Note that the $8\varepsilon$ instead of $7\varepsilon$ in the definition of $\overline{G}^{u,v}_{x, 0} $, accounts for the scale change one obtains from $u^{1-\et{uvfix}}$.

In view of Lemma~\ref{lemma:high_0-good}, the monotonicity in $v$ and the definition of $L_u$, we thus define a sequence $(b_k)_{k \geq 1}$  satisfying that for each $k \geq 1$  
\begin{equation} \label{eq:b_k}
		\sup_{x \in \mathbb{G}_0}  \prob^{u} \left(  \overline{G}^{u,v}_{x,0} \right) \leq 2^{-2^{k}}, \qquad \text{ for all } u \in (0 , b_k)\text{ and }v\le u^{1-\et{uvfix}}.
\end{equation}
In particular, since $ u < u^{1-\et{uvfix}} $ whenever $u \in (0,1) $, $\sup_{x \in \mathbb{G}_0}  \prob^{u} \left(  \overline{G}^{u,u}_{x,0} \right) \leq 2^{-2^{k}}$ for $u \in (0 , b_k)$.

We define recursively the conditions for a $k$-good vertex in $\mathbb{G}_k$.

\begin{definition} \label{def:k-good}
	
	For $k \geq 1$ and $x \in \mathbb{G}_k$, the \emph{complement} of the $k$-good event at $x$ is
	\begin{equation} \label{eq:kbad_event}
		\overline{G}^{u,v}_{x , k} \coloneqq \bigcup_{ \substack{x_1, x_2 \in \mathbb{G}_{k-1} \cap V (x, L_k)  \\ \vert x_1 - x_2 \vert_{\infty} > r_{k-1} \cdot L_{k-1} }  } \overline{G}^{u,v}_{x_1, k-1} \cap \overline{G}^{u,v}_{x_2, k-1}.
	\end{equation}
	On the event $\overline{G}^{u,v}_{x, k}$, the vertex $x$ and the box $V (x, L_k)$ are $k$-bad since there are two $(k-1)$-bad boxes inside $V (x, L_k)$ which are apart of each other.
	Otherwise, if the favourable event $G^{u,v}_{x,k}$ happens, then $x$ and  $V (x, L_k)$ are $k$-good.
\end{definition}

Next we prove that a  box is $k$-good with high probability. This proof follows the argument in~\cite[Theorem 4.1]{DrewitzRathSapozhnikov}, but we  remark on two critical differences between the later theorem and Proposition~\ref{prop:multi_renormalization}. In~\cite{DrewitzRathSapozhnikov}, Theorem 4.1 requires a hypothesis on $L_0$ as $L_0 \to \infty$, but in our case $L_0$ has been determined in~\eqref{eq:L0_scale}. The proof shows that our choice of $L_0$ satisfies the requirements of~\cite[Theorem 4.1]{DrewitzRathSapozhnikov}.
The second difference is that our scheme holds uniformly for $u$ small enough. For this uniformity, we require the decoupling inequality in Theorem~\ref{thm:RI_decorrelation}.
We provide the details of the proof for clarity.

\begin{prop} \label{prop:multi_renormalization}
	Let $  \uBound{u:fixedEps} > 0 $ as in the beginning of this section, take  $ u  \in (0, \uBound{u:fixedEps} )$ and $L_0 = L_u$.
	There  exists $l_0 \geq 1$ and $C = C(\uBound{u:fixedEps} , l_0) < \infty$  such that for any choice of $l_0/4>r_0 \geq C $ and $\theta_{\text{sc}} = \theta_{\text{sc}} (\uBound{u:fixedEps}, l_0) > 0$  the following holds.
	Consider the renormalization scheme defined by the  scales $ \Theta = (\theta_{\text{sc}}, L_0, l_0, r_0 )$. In particular, recall the sequences in~\eqref{eq:renormalization-seq} defined by
	\begin{equation} \label{eq:escalesprop}
		l_k = l_0 \cdot 4^{k^{\theta_{\text{sc}}}} , 
		\quad r_k = r_0 \cdot 2^{k^{\theta_{\text{sc}}}}, 
		\quad L_k = l_{k-1} \cdot L_{k-1}, \quad k \geq 1,
	\end{equation}
	and the corresponding $k$-good events.
	Then there exists $ \newuBound \label{u:renormalization}  \in (0, \uBound{u:fixedEps} )$ such that for all $k \geq 0$  
	\begin{equation} \label{eq:boundk}
		\sup_{x \in \mathbb{G}_k} \prob^u \left( \overline{G}^{u,u}_{x, k} \right) \leq 2^{-2^k}, \qquad \text{ for all }  u \in (0, \uBound{u:renormalization}).
	\end{equation}
	Moreover, the scales $(r_k)_{k \geq 0}$ and $(l_k)_{k \geq 0}$ satisfy: 
	\begin{equation} \label{eq:boundVolume}
		\prod_{k = 0}^{\infty} \left( 1  - \frac{r_k}{l_k} \right) > \frac{1}{2}.
	\end{equation}
\end{prop}

\begin{proof}[{Proof of Proposition~\ref{prop:multi_renormalization}}]
	This proposition follows from the stochastic monotonicity of random interlacements with respect to the intensity parameter, Proposition~\ref{prop:RI_monotone}, and the weak decorrelation inequality in~\cite{PopovTeixeira} (see Theorem~\ref{thm:RI_decorrelation}).

	We begin by setting values for the scales $l_0$ and $r_0$ in $\Theta  = (\theta_{\text{sc}}, L_0, l_0, r_0 )$ and the upper bound for the intensity indicated in~\eqref{eq:boundk}.

	Before we fix a value for the positive integer $l_0 > 1$, let us consider an auxiliary sequence $(\kappa_{k})_{k \geq 0}$ defined by 
	\[
		\kappa_k = 1 + l_0 \cdot \sum_{i = k +  1}^{\infty} i^{-2} = \kappa_{k + 1} + l_0 \cdot (k + 1)^{-2},
	\]
	for each $k \geq 0$.
	Recall that $  l_k = l_0 \cdot 4^{k^{\theta_{\text{sc}}}}$. Keeping in mind the definition of $(\kappa_{k})_{k \geq 0}$ above, there exists $C < \infty$ such that for all $k \geq 0$ and $l_0 \geq C$:
	\begin{equation} \label{eq:scalesineq-1}
		\log (l_k^{2d}) - \kappa_k 2^{k + 1} \leq - \kappa_{k+1} 2^{k+1} -1.
	\end{equation}
	So we choose $l_0 \geq C$ and~\eqref{eq:scalesineq-1} holds. 

	Next we choose an upper bound for the intensity. By Lemma~\ref{lemma:high_0-good} (see Equation~\eqref{eq:b_k}),  there exists $\newuBound \label{u:startrenorm}>0$ such that for any $u \in (0,   \uBound{u:startrenorm})$ and $v \in ( 0, u^{1-\et{uvfix}})$:
	\begin{equation} \label{eq:basecase}
		\sup_{x \in \mathbb{G}_0}  \prob^{v} \left( \overline{G}^{u,v}_{x,0} \right) \leq 2^{-\kappa_0}.
	\end{equation}
	We set $ \uBound{u:renormalization}  \coloneqq \min\{ \uBound{u:startrenorm}, \uBound{u:fixedEps} , 2^{-1} \}  $ as the upper bound for the intensity.

	Take $ \hat{u} \in (0, \uBound{u:renormalization} ) $. Our goal is to show~\eqref{eq:boundk} for the intensity parameter $\hat{u}$:
	\[
		\sup_{x \in \mathbb{G}_{k}} \prob^{\hat{u}} \left( \overline{G}^{\hat{u},\hat{u}}_{x, k} \right) \leq 2^{-2^k}, \qquad \text{ for all } k \geq 0.
	\]
	The proof goes through a decreasing sequence of intensity values in $(\hat{u}, \hat{u}^{1-\et{uvfix}} )$.  Since $ \hat{u} < \uBound{u:renormalization} \leq 2^{-1} $, we can choose  $\hat{\delta} = \hat{\delta} (\et{uvfix}) > 0$ be such that
	\begin{equation} \label{eq:delta}
		(1 + \hat{\delta})^{-1} \hat{u}^{1-\et{uvfix}} \in (\hat{u} , \hat{u}^{1-\et{uvfix}}).
	\end{equation}
  Next choose $r_0 = r_0  ( \et{uvfix}) > 1  $ large enough, and choose $l_0(r_0)$ large enough (satisfying \eqref{eq:scalesineq-1}), such that the sequence $(r_k)_{k \geq 0}$ defined in~\eqref{eq:renormalization-seq} satisfies
	\begin{equation} \label{eq:r0}
			\prod_{k = 0}^{\infty} \left( 1 + r_k^{- \frac{1}{4}} \right) \leq 1 + \hat{\delta} \quad \text{ and }\quad \prod_{k = 0}^{\infty} \left( 1  - \frac{r_k}{l_k} \right) > \frac{1}{2} .
	\end{equation} 
	Note that the second condition in~\eqref{eq:r0} is similar to~\cite[Equation 3.18]{ProcacciaRosenthalSapozhnikov} and gives~\eqref{eq:boundVolume}.
	We  set  
	\begin{equation} \label{eq:secV}
		\hat{u}_0 \coloneqq \hat{u}^{1-\et{uvfix}} \qquad  \text{and}  \qquad
		\hat{u}_{k+1} \coloneqq  \left(1 + r_k^{-\frac{1}{4}} \right)^{-1} \cdot \hat{u}_k \quad \text{for } k \geq 1.
	\end{equation}
	The restrictions over $\hat{\delta}$ and $r_k$ in~\eqref{eq:delta} and~\eqref{eq:r0}, respectively, and the definition of $(\hat{u}_k)_{k \geq 0}$ imply that
	\begin{equation} \label{eq:seq_u}
			0 < \hat{u}  < \hat{u}_{k+1} < \hat{u}_k \leq  \hat{u}_0 = \hat{u}^{1 -  \et{uvfix} } , \qquad \text{ for all } k \geq 0.
	\end{equation}

	Since the events $ \overline{G}^{u,v}_{x,k}$ are monotonic increasing in $v$, Proposition~\ref{prop:RI_monotone} and~\eqref{eq:seq_u} imply that
	\begin{equation} 
			\sup_{x \in \mathbb{G}_k} \prob^{ \hat{u} } \left( \overline{G}^{\hat{u},\hat{u}}_{x,k} \right)\leq \sup_{x \in \mathbb{G}_k} \prob^{\hat{u}_k} \left( \overline{G}^{\hat{u},\hat{u}_k}_{x,k} \right) , \qquad \text{ for all } k \geq 1.
	\end{equation}
	Therefore, it suffices to show
	\begin{equation} \label{eq:inductionkey}
		\sup_{x \in \mathbb{G}_k} \prob^{\hat{u}_k} \left( \overline{G}^{\hat{u},\hat{u}_k}_{x,k} \right) \leq 2^{-\kappa_k 2^k}, \qquad \text{ for all } k \geq 1.
	\end{equation}

	We proceed to prove~\eqref{eq:inductionkey} by induction on $k$. 
	The choice of $\hat{u}_0$ in~\eqref{eq:secV} implies the base case. 
	We now assume that \eqref{eq:inductionkey} holds for some $k > 0$. For each $x \in \mathbb{G}_{k+1}$, a union bound and Theorem~\ref{thm:RI_decorrelation} imply that 
\begin{align}  
	\prob^{ \hat{u}_{k+1}} &\left( \overline{G}^{\hat{u}, \hat{u}_{k+1}}_{x, k + 1} \right)
	\leq
				\sum_{ \substack{x_1, x_2 \in \mathbb{G}_{k} \cap V (x, L_{k+1})  \\ \vert x_1 - x_2 \vert_{\infty} > r_{k} \cdot L_{k} }  }  
				\prob^{\hat{u}_{k+1}} \left( \overline{G}^{\hat{u}, \hat{u}_{k+1}}_{x_1, k} \cap \overline{G}^{\hat{u}, \hat{u}_{k+1}}_{x_2, k} \right)\nonumber  \\
	&< 
				\vert \mathbb{G}_{k} \cap V (x, L_{k+1})  \vert^2 \left( \sup_{x \in \mathbb{G}_k} 
				\prob^{\hat{u}_k} \left( \overline{G}^{\hat{u}, \hat{u}_{k}}_{x,k} \right)^2 
				+ \cnt{c:RIdecoupling} ( 4L_{k+1})^d e^{ - \expo{e:RIdecoupling} r_k^{-1 / 2} \hat{u}_k {L_{k}^{d-2}}  }\right) \nonumber  \\
	&\leq l_k^{2d} \left( 2^{-\kappa_k 2^{k + 1}} +  e^{ - \expo{e:RIdecoupling} r_k^{-1 / 2} \hat{u}_k {L_{k }^{d-2}}  +  d \log (L_{k+1})+ \log (c)}  \right).  \label{eq:renorm}
\end{align}
In the second inequality, we apply Theorem~\ref{thm:RI_decorrelation} with parameters $u = \hat{u}_{k+1}$ and $(1 + \varepsilon) u = \hat{u}_{k}$; hence~\eqref{eq:secV} implies taking $\varepsilon = r_k^{-1 / 4}$. Our scale choices for the renormalization scheme imply that, in our application of Theorem~\ref{thm:RI_decorrelation}, we take $R = L_k $ and $L = r_k \cdot L_{k} < L_{k+1}$.

We now consider the last exponent in~\eqref{eq:renorm} and set 
\begin{equation}
	\begin{split}  \label{eq:fP}
	f_P (L_k) &\coloneqq \expo{e:RIdecoupling} r_k^{-1 / 2} \hat{u}_k {L_{k}^{d-2}}  -  d \log (L_{k+1}) - \log (c) \\
						&=  \expo{e:RIdecoupling} r_k^{-1 / 2} \hat{u}_k {L_{k}^{d-2}}  -  d \log (l_{k}) -d \log(L_{k}) - \log (c) .
	\end{split}
\end{equation}
Then the inequality above takes the form
\begin{equation} \label{eq:inductivestep}
	\prob^{\hat{u}_{k+1}} \left( \overline{G}^{\hat{u}, \hat{u}_{k+1}}_{x, k + 1} \right) \leq l_k^{2d} \left( 2^{-\kappa_k 2^{k + 1}} +  e^{ - f_P (L_k) } \right).
\end{equation}
From the expression in~\eqref{eq:fP}, we see that there exists $C' = C' (l_0) < \infty$ so that for $\theta_{\text{sc}} \geq C'$  we have that
\[
	f_P (L_k) \geq  \kappa_k 2^{k + 1} \qquad \text{ for all }  k \geq 0.
\]
and hence, by~\eqref{eq:scalesineq-1},
\begin{equation} \label{eq:scalesineq-2} 
	\log (l_k^{2d}) - f_P (L_k) \leq - \kappa_{k+1} 2^{k+1} - 1.
\end{equation}
Plugging~\eqref{eq:scalesineq-1} and~\eqref{eq:scalesineq-2} into~\eqref{eq:inductivestep}, we obtain that~\eqref{eq:inductionkey} holds for $k + 1 $.
Hence, by the principle of mathematical induction, we have proved~\eqref{eq:inductionkey} for all $k \geq 0$, as desired.
\end{proof}

We consider a fattened version of the good event $G^{u, u}_{0, k}$. Under this event, we will be able to consider any geodesic between $[0]_u$ and $[L_k/L_0 \cdot e_1]_u$. 
Let $z_0 = (-L_k , \ldots , -L_k )$ and set
 
\begin{equation}   \label{eq:fatgoodbox}
	\mathfs{G}^{u}_{0, k}  =   \bigcap_{  x_i \in \mathbb{G}_k \cap V(z_0, 3L_k) } G^{u,u}_{x_i, k}  
\end{equation}

\begin{figure}[h] \label{fig:path-to-boxes}
\centering
	\begin{tikzpicture}

		\draw[step=.2cm,gray,very thin] (-4.001,-4.001) grid (2,2) ;
		\draw[step=2cm, thin] (-4.001,-4.001) grid (2,2) ;

		\fill[fill=black] (-4, -4) circle[radius=2pt]  node[anchor=east]  {\small $ z_0 $};

		\draw[dashed, fill=white,fill opacity=0.8] (-2.50, -2.50) rectangle (-1.50, -1.50);
		\fill[fill=black] (-2, -2)  circle[radius=2pt] node[anchor=north]  {\small $ 0 $};

		\draw[dashed, fill=white,fill opacity=0.8] (-.5, -2.5) rectangle (.5, -1.5);
		\fill[fill=black] (0, -2)  circle[radius=2pt]  node[anchor=north]  {\small $ L_k  e_1 $};
				
		\draw [decorate,decoration={brace,amplitude=5pt,mirror,raise=.3em}]
		(-4, -4) -- (-2, -4) node[midway,yshift=-1.5em]{\small  $ V(z_0, L_k) $};

		\draw [black,thick] plot [smooth, tension=1] 
			 coordinates { 
				(-1.7, -2.30)
				(-.7, -2.7)
				(-.5, -3.5)
				(1.8, -3.3)
				(3, -3.2)
				(2.1, -1.7)
				(1.5, -3)
				(.4, -2.3) 
			};

		\fill[fill=black] (-1.7, -2.30) circle[radius=1pt];
		\fill[fill=black] (.4, -2.3)  circle[radius=1pt];
	\end{tikzpicture}
	\caption{On the event $\mathfs{G}^{u}_{0,k}$, the $k$-boxes around $0$ and $L_k \cdot e_1$ are good.} 
\end{figure}
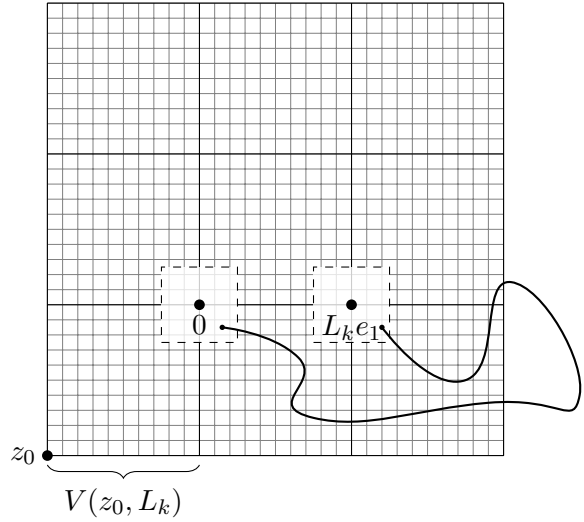

\begin{lemma} \label{lemma:box_intersection}
Assume the event $\mathfs{G}^{u}_{0, k} $ holds. Then there exists an absolute constant $\delta>0$ such that any path  on  $\mathbb{G}_0$ connecting any $0$-boxes $V_1 \subset V(0, L_k/4)$ and $V_2 \subset V(L_k/L_0 \cdot e_1, L_k/4)$ must cross at least $\delta L_k/L_0$ $0$-good boxes.
\end{lemma}
\begin{proof}
The proof is by induction. By the multi-scale renormalization construction, there is at most one region of side length $r_{k-1}$  enclosing all the $(k-1)$-bad boxes within a $k$-good box. Then any path connecting $V_1$ and $V_2$ crosses at least $ \frac{1}{2}L_k / L_0 $ $0$-boxes
inside the big box $V(z_0, 3L_k)$ considered in the event~\eqref{eq:fatgoodbox}.  Then we must intersect
\[  
	\frac{1}{2}\left(l_{k-1} - r_{k-1}\right) = \frac{1}{2} l_{k-1}\left(1- \frac{r_{k-1}}{l_{k-1}}\right) 
\] 
$(k-1)$-good boxes. An induction argument over $k$ shows that any path connecting the $0$-boxes $V_1$ and $V_2$ must intersect at least
 \[  \frac{1}{2}\prod_{i=1}^{k}l_{i-1}\left(1-\frac{r_{i-1}}{l_{i-1}}\right) ,\]
 $0$-good boxes.
Note that $\prod_{i=1}^{k}l_{i-1}=L_k / L_0$. Moreover,~\eqref{eq:boundVolume}  implies  that 
$$
 	\prod_{i=1}^{k} \left(1-\frac{r_{i-1}}{l_{i-1}}\right)\ge \prod_{k=1}^\infty\left(1-\frac{r_{k-1}}{l_{k-1}}\right) > \frac{1}{2}.$$
We obtain the conclusion by setting $ \delta =  1/4 $.
\end{proof}

\subsection{Proof of Proposition~\ref{prop: asymp-bound}}\label{subsec:proof_of_proposition}

 Now that we have established the scheme and the geometric consequences for any path between $[0]_u$ and $[L_ke_1]_u$ in Lemma~\ref{lemma:box_intersection}, we are ready to prove the main proposition of this section.  For our framework to work, we need an additional event:
\[
		\mathfs{E}_k =  \Big\{  \vert  0  - [0]_u \vert < L_k/4, \quad  \vert  L_k e_1  - [L_ke_1]_u \vert < L_k/4 \Big\}.
\] 

\begin{proof}[Proof of Proposition~\ref{prop: asymp-bound}]
	From a well-known estimate on the probability of intersection of a finite set and the random interlacements graph (see e.g. \cite[Equation (2.1.2)]{DRSbook}), we have that the probability of the event $\mathfs{E}_k$ is at least
	\[
		\prob \left( \mathfs{E}_k \right) \geq 1 -  2\exp \left( - u \cdot \capa (V (z_0, L_k/2 ) ) \right).
	\]
	We have that there is some $\newcnt \label{c:capacity} = \cnt{c:capacity} (d) > 0$ such that $\capa V (z_0, L_k/2) \geq \cnt{c:capacity} L_k^{d-2} $, and our choice of $L_k$ in Proposition~\ref{prop:multi_renormalization} implies that there is a $c>0$ such that
	\begin{equation} \label{eq:intersection}
		 u \cdot \capa V (z_0, 3L_k) \ge \sqrt{L_k}\ge c2^k
	\end{equation}
	for all $u> 0$ small enough, uniformly for all $k$. We thus get that there exists $ \newuBound \label{u:intersectionI} \in (0,1) $ such that $u \in (0, \uBound{u:intersectionI})$ satisfies~\eqref{eq:intersection}. 

	On the other hand, by Proposition~\ref{prop:multi_renormalization} and a union bound on the $3^d$ $k$-boxes considered in~\eqref{eq:fatgoodbox}, 
	\[
		\prob \left( \mathfs{G}^{u}_{0, k} \right) \geq 1-  3^d \cdot 2^{-2^k} \qquad
		\text{ for all } u \in (0, \uBound{u:renormalization}). 
	\]
	We now set $ \uBound{u:propLowerBound} \coloneqq \min \{ \uBound{u:renormalization} , \uBound{u:intersectionI} \} $. Then, for every $u  \in (0, \uBound{u:propLowerBound})$ we get
	\[
		\prob \left( \mathfs{E}_k \cap  \mathfs{G}^{u}_{0, k} \right)	
					\geq 1 - c e^{ -c2^{k}} \quad \text{ for all } u \in (0, \uBound{u:propLowerBound} ).
	\]

	Let $\gamma$ be an arbitrary path on $\mathcal{I}^u$ between $[0]_u$ and $[L_k e_1]_u$. 
	Let $\delta > 0$ be the absolute constant in Lemma~\ref{lemma:box_intersection}. 
	On the event $ \mathfs{E}_k \cap  \mathfs{G}^{u}_{0, k}$, Lemma~\ref{lemma:box_intersection} implies that $\gamma$ crosses at least
	\[
		\delta L_k / L_0 = \delta L_k / L_u
	\]
	 $0$-good boxes (for the equality above we used the definition of $L_0$ in~\eqref{eq:L0_scale}). Hence, $\gamma$ crosses at least $3^{-d} \delta L_k / L_u$ $0$-good boxes which are non-adjacent.
	
	By the definition of a $0$-good box in~\eqref{eq:multiscale_0-good},  the crossing of each $0$-good box takes length at least $C_{\const{prop:annulus}{\ell} } u^{-\frac{1}{2}+8\varepsilon} L_u$. Then the length of $\gamma$  is at least 
 	\[
 		\frac{1}{3^d} \left(	\delta L_k / L_u \right)\left( C_{\const{prop:annulus}{\ell}} u^{-\frac{1}{2}+8\varepsilon} L_u \right)
 		= 
 		 \left(	 \frac{1}{3^d} \delta C_{\const{prop:annulus}{\ell}} \right)\cdot u^{4\varepsilon}  \cdot L_u  \cdot L_k .
 	\]
  We finish the proof of the proposition by taking $ \Cnt{c:low}  = 3^{-d} \delta C_{\const{prop:annulus}{\ell}} $. Since $L_k = L_u \prod_{j= 0}^k \ell_j$, we can give a bound on $k$ in terms of $n$ in~\eqref{eq:asymp-bound}, recalling that the scales $\Theta = (\theta_{\text{sc}}, L_0, l_0, r_0 )$ depends on the upper bound on the intensity $\uBound{u:propLowerBound}$.
\end{proof}

\subsection*{Acknowledgements}

We would like to thank the anonymous referee, whose extremely dedicated and thorough work has greatly improved this paper. We would like to thank Bal\'azs R\'ath for insightful discussions that were greatly instrumental for the proof of Theorem \ref{thm:main_upper}. We would like to thank Art\"{e}m Sapozhnikov for insightful discussions that were greatly instrumental for the proof of Theorem \ref{thm:main_lower},  and for typing some of the proofs in Section \ref{subsec"chem_random_walk}.

\bibliography{RIbiblio}
\bibliographystyle{plain}

\end{document}